\documentclass{amsart}
\usepackage{dsfont, amssymb, mathtools, IEEEtrantools, subcaption}
\usepackage{calc,changepage}
\usepackage{csquotes}
\usepackage{hyperref}
\usepackage{enumitem, siunitx}

\usepackage{xcolor}
\definecolor{wb}{RGB}{51,153,255}
\usepackage{tikz}
\usetikzlibrary{calc, patterns, decorations.pathreplacing}
\usepackage{tikz-cd}

\usepackage[parfill]{parskip}
\numberwithin{equation}{subsection}
\usepackage{amsrefs}

\newcommand{\defeq}{\vcentcolon=}
\newcommand{\eqdef}{=\vcentcolon}

\makeatletter
\def\moverlay{\mathpalette\mov@rlay}
\def\mov@rlay#1#2{\leavevmode\vtop{%
   \baselineskip\z@skip \lineskiplimit-\maxdimen
   \ialign{\hfil$\m@th#1##$\hfil\cr#2\crcr}}}
\newcommand{\charfusion}[3][\mathord]{
    #1{\ifx#1\mathop\vphantom{#2}\fi
        \mathpalette\mov@rlay{#2\cr#3}
      }
    \ifx#1\mathop\expandafter\displaylimits\fi}
\makeatother

\newcommand{\cupdot}{\charfusion[\mathbin]{\cup}{\cdot}}

\newtheoremstyle{definitions}
 	{\topsep}
	{\topsep}
	{}
	{}
	{\bfseries}
	{:}
	{.5em}
	{}
  
\newtheoremstyle{lemmata}
	{\topsep}
	{\topsep}
	{\itshape} 
	{}
	{\bfseries}
	{:}
	{.5em}
	{}

\theoremstyle{lemmata}
\newtheorem{Theorem}[subsection]{Theorem}
\newtheorem{Lemma}[subsection]{Lemma}
\newtheorem{Corollary}[subsection]{Corollary}
\newtheorem{Proposition}[subsection]{Proposition}

\theoremstyle{definitions}
\newtheorem{Definition}[subsection]{Definition}
\newtheorem{Remark}[subsection]{Remark}
\newtheorem{Remarks}[subsection]{Remarks}
\newtheorem*{Remarks-nn}{Remarks}

\newtheorem{Note}[subsection]{Note}
\newtheorem{Example}[subsection]{Example}

\newtheorem{Picture}[subsection]{Picture}

\DeclareMathOperator{\GL}{GL}

\DeclareMathOperator{\PGL}{PGL}

\DeclareMathOperator{\diag}{diag}
\DeclareMathOperator{\codim}{codim}
\DeclareMathOperator{\id}{id}
\DeclareMathOperator{\SL}{SL}

\title[The behavior of distinguished forms on the fundamental domain]{On Drinfeld modular forms of higher rank VI: The simplicial complex associated with a coefficient form}
\author{Ernst-Ulrich Gekeler}
\date{\today}

\newcommand*\rows{5} 

\begin{document}

\begin{abstract}
	The coefficient forms \( {}_{a} \ell_{k} \) and the para-Eisenstein series \(\alpha_{k}\) are simplicial Drinfeld modular forms. We study the attached simplicial complexes \(\mathcal{BT}^{r}( {}_{a} \ell_{k})\) and
	\(\mathcal{BT}^{r}(\alpha_{k})\), which are full subcomplexes of the Bruhat-Tits building \(\mathcal{BT}^{r}\) of \( \PGL(r, K_{\infty})\). They are connected (if the rank \(r\) is larger than 2), strongly equidimensional
	of codimension 1 in \(\mathcal{BT}^{r}\), boundaryless, and satisfy a symmetry property under the non-trivial involution of the Dynkin diagram \(A_{r-1}\).
\end{abstract}

\maketitle

\section{Introduction and notation} \label{Section.Introduction-and-notation}

\subsection{} This article continues the work of the papers \enquote{On Drinfeld modular forms of higher rank I, \dots, V} (\cite{Gekeler2017}, \cite{Gekeler-ta-1}, \cite{Gekeler2018}, \cite{Gekeler-ta-2}, \cite{Gekeler-ta-3}), labelled
[I], \dots, [V], and notably of [V].

We show the simpliciality of the Drinfeld modular forms \( {}_{a}\ell_{k} \), where \(a \in A \defeq \mathds{F}_{q}[T]\) and \({}_{a}\ell_{k}\) is the coefficient of the generic Drinfeld module \(\phi^{\boldsymbol{\omega}}\) of rank \(r \geq 2\)
on the Drinfeld space \(\Omega^{r}\), with \(a\)-division polynomial
\begin{equation} \label{Eq.A-Division-Polynomial}
	\phi_{a}^{\boldsymbol{\omega}}(X) = aX + \sum_{1 \leq k \leq r \deg a} {}_{a}\ell_{k}(\boldsymbol{\omega}) X^{q^{k}}.
\end{equation}
By definition, a modular form \(f\) is simplicial if the image \(\mathcal{BT}^{r}(f)\) of its zero set \(\Omega^{r}(f)\) under the building map
\[
	\lambda^{r} \colon \Omega^{r} \longrightarrow \mathcal{BT}^{r}(\mathds{Q})
\]	
(see [I] Sect. 2 for details and \cite{BruhatTits1972} for foundational material) to the set of \(\mathds{Q}\)-points of the Bruhat-Tits building \(\mathcal{BT}^{r}\) of \(\PGL(r, \mathds{F}_{q}((T^{-1})) )\) is the set of \(\mathds{Q}\)-points
of a full subcomplex of codimension 1, also denoted by \(\mathcal{BT}^{r}(f)\), of \(\mathcal{BT}^{r}\). Beyond the existence of the simplicial complex \(\mathcal{BT}({}_{a}\ell_{k})\), we show that it has pleasant properties. It
\begin{itemize}
	\item depends only on the degree \(d = \deg a\) of \(a \in A\), and is therefore denoted \(\mathcal{BT}^{r}(d,k)\);
	\item is strongly equidimensional of codimension 1 in \(\mathcal{BT}^{r}\) (hence of dimension \(r-2\));
	\item is boundaryless (for precise definitions, see below); and
	\item satisfies a certain symmetry property with respect to the involution of the Dynkin diagram of the underlying root system of type \(A_{r-1}\).
\end{itemize}
The results, along with similar ones for the related para-Eisenstein series \(\alpha_{k}\), are collected in the Main Theorem 1.8.

\subsection{} The question of simpliciality of \( {}_{a}\ell_{k}\) had been answered to the positive for \(r=2\) in \cite{Gekeler2011}. Here it just means that the (isolated) zeroes of \( {}_{a}\ell_{k} \) map under \(\lambda^{2}\) to vertices
of the tree \(\mathcal{BT}^{2}\). In the present work, we deal with the case of a fixed rank \(r \geq 3\), where the structure is richer and allows deeper findings. Nevertheless, all the present results are still valid for \(r=2\) (but in part empty
or trivial).

A new ingredient is the criterion of Theorem \ref{Theorem.Membership-via-d-characteristic-sequences}, an easy-to-handle device to decide whether a given rational point \(\mathbf{x}\) of the Bruhat-Tits building \(\mathcal{BT}^{r}\) belongs to the vanishing set \(\mathcal{BT}^{r}(d,k)\)
or not. It results from Theorem 3.2 of [V] and the observation Proposition \ref{Proposition.Monotonicity-of-modular-form-map} about the sizes of the torsion points of \(\phi^{\boldsymbol{\omega}}\).

\subsection{} The notation agrees with that of [V] and its predecessors, to which we refer for further information. Thus \(A\) is the polynomial ring \(\mathds{F}[T]\), where \(\mathds{F} = \mathds{F}_{q}\) is the field with \(q\)
elements, \(q\) a power of a prime \(p\), \(K = \mathds{F}(T)\) its fraction field, \(K_{\infty} = \mathds{F}((T^{-1}))\) the completion at infinity, with ring of integers \(\mathcal{O}_{\infty} = F[[T^{-1}]]\) and completed algebraic
closure \(C_{\infty}\). The absolute value \( \lvert \mathbin{.}\rvert \) on \(K_{\infty}\) and \(C_{\infty}\) is normalized such that \(\lvert T \rvert = q \). For \(0 \neq x \in C_{\infty}\) we write \( \log x = \log_{q} \lvert x \rvert = - v_{\infty}(x)\).
We fix a natural number \(r \geq 3\) (occasionally, \(r=2\) is allowed) and let 
\begin{equation}
	\Omega = \Omega^{r} = \mathds{P}^{r-1}(C_{\infty}) \smallsetminus \bigcup H(C_{\infty}))
\end{equation}
be the Drinfeld space \cite{Drinfeld1974}, where \(H\) runs through the set of hyperplanes of \(\mathds{P}^{r-1}\) defined over \(K_{\infty}\). Projective coordinates \( (\omega_{1}: \dots : \omega_{r})\) on \(\Omega\) are normalized
such that \(\omega_{r} = 1\). It is related with the Bruhat-Tits building \(\mathcal{BT} = \mathcal{BT}^{r}\) of \(\PGL(r, K_{\infty})\) through the building map
\[
	\lambda \colon \Omega \longrightarrow \mathcal{BT}(\mathds{Q})
\]
onto the set of points with rational barycentric coordinates of the realization \(\mathcal{BT}(\mathds{R})\) of \(\mathcal{BT}\) (see [II] Sect. 2; beware of the sign error in [II] corrected by [V] 1.3.). The map \(\lambda\)
is equivariant for the actions of the group \(\GL(r, K_{\infty})\) on \(\Omega\) and \(\mathcal{BT}\).

For \(\boldsymbol{\omega} \in \Omega\), we let \(\Lambda_{\boldsymbol{\omega}}\) be the lattice \(\Lambda_{\boldsymbol{\omega}} \defeq A\omega_{1} + \dots + A\omega_{r}\) in \(C_{\infty}\) (well-defined, as we have fixed
\(\omega_{r} = 1\)) with associated Drinfeld module \(\phi^{\boldsymbol{\omega}}\) of rank \(r\). Then \(\phi^{\boldsymbol{\omega}}\) is determined through its \(T\)-operator polynomial
\begin{equation}
	\phi_{T}^{\boldsymbol{\omega}} = TX + \sum_{1 \leq k \leq r} g_{k}(\boldsymbol{\omega}) X^{q^{k}}.
\end{equation}
Similarly, for each \(a \in A\) of degree \(d\), the \(a\)-operator polynomial has shape
\begin{equation}
	\phi_{a}^{\boldsymbol{\omega}}(X) = aX + \sum_{1 \leq k \leq rd} {}_{a}\ell_{k}(\boldsymbol{\omega}) X^{q^{k}}.
\end{equation}
(Thus \(g_{k} = {}_{T}\ell_{k}\).) As a function of \(\boldsymbol{\omega} \in \Omega\), \({}_{a}\ell_{k}\) is a modular form for \(\Gamma \defeq \GL(r,A)\) of weight \(q^{k} - 1\) and type 0 (see [I] Sect. 1). As such, it is determined by its restriction
to the \textbf{fundamental domain} \(\mathbf{F}\) for \(\Gamma\) on \(\Omega\):
\begin{equation}
	\mathbf{F} \defeq \left\{ \begin{array}{r|l}  &  \text{the ordered set \( \{\omega_{r} = 1, \omega_{r-1}, \dots, \omega_{1}\} \)} \\ \boldsymbol{\omega} = (\omega_{1} : \dots : \omega_{r}) \in \Omega & \text{is a successive minimum basis of the} \\ & \text{\(A\)-lattice \(\Lambda_{\boldsymbol{\omega}}\)} \end{array} \right\}
\end{equation}
(See [I] 1.13 and \cite{Gekeler2019} Sect. 3.) 

Recall that the set \(\mathcal{BT}(\mathds{Z})\) of vertices of \(\mathcal{BT}\) is the set of similarity classes \([L]\) of \(O_{\infty}\)-lattices \(L\) in the \(K_{\infty}\)-vector space \(V = K_{\infty}^{r}\). The \textbf{standard apartment}
\(\mathcal{A}\) of \(\mathcal{BT}\) is the full subcomplex with vertex set 
\begin{equation} \label{Eq.Vertex-set-of-full-subcomplex}
	\mathcal{A}(\mathds{Z}) = \{ [L_{\mathbf{n}}] \mid \mathbf{n} \in \mathds{Z}^{r} \},
\end{equation}
where \(L_{\mathbf{n}}\) is the lattice \(L_{\mathbf{n}} = \pi^{n_{1}}O_{\infty} \oplus \dots \oplus \pi^{n_{r}} O_{\infty} \subset V\) with the uniformizer \(\pi \defeq T^{-1}\) of \(K_{\infty}\). It corresponds to the \textbf{standard torus} 
\(\mathbf{T}\) of diagonal matrices in \(\PGL(r, K_{\infty})\). We have
\begin{equation} \label{Eq.Equivalence-of-vertices}
	[L_{\mathbf{n}}] = [L_{\mathbf{n}'}] \Longleftrightarrow \mathbf{n}' - \mathbf{n} = (n,n,\dots, n)
\end{equation}
with some \(n\) in \(\mathds{Z}\). As usual (see [V] 1.2.4) we identify the realization \(\mathcal{A}(\mathds{R})\) of \(\mathcal{A}\) with the euclidean affine space with translation group
\[
	\mathds{R}/\mathds{R}(1,1,\dots,1) \overset{\cong}{\longrightarrow} \{ \mathbf{x} \in \mathds{R}^{r} \mid x_{r} = 0\}
\]
and with the natural choice of origin \( \mathbf{0} = [L_{\mathbf{0}}]\). The \textbf{standard Borel subgroup} \(\mathbf{B}\) of upper triangular matrices determines the \textbf{standard Weyl chamber} \(\mathcal{W}\). Its vertex set is
\begin{equation}
	\mathcal{W}(\mathds{Z}) = \{ [L_{\mathbf{n}}] \mid \mathbf{n} = (n_{1}, \dots, n_{r}) \in \mathds{Z}^{r} \text{ and } n_{1} \geq n_{2} \geq \dots \geq n_{r} = 0\}
\end{equation}
(usually we simply write \(\mathbf{n}\) for the vertex \([L_{\mathbf{n}}]\)) and 
\[
	\mathcal{W}(\mathds{R}) = \{ \mathbf{x} \in \mathds{R}^{r} \mid x_{1} \geq x_{2} \geq \dots \geq x_{r} = 0\}.
\]
Then \(\mathcal{W}\) is a fundamental domain for the action of \(\Gamma\) on \(\mathcal{BT}\), and its relation with \(\mathbf{F}\) is 
\begin{equation}
	\lambda(\mathbf{F}) = \mathcal{W}(\mathds{Q}), \qquad \lambda^{-1}(\mathcal{W}(\mathds{Q})) = \mathbf{F}.
\end{equation}
For \(\mathbf{x} \in \mathcal{BT}(\mathds{Q})\) let \(\Omega_{\mathbf{x}} \defeq \lambda^{-1}(\mathbf{x})\), which for \(\mathbf{x} = (x_{1}, \dots, x_{r} = 0) \in \mathcal{W}(\mathds{Q})\) specializes to 
\begin{equation}
	\mathbf{F}_{\mathbf{x}} \defeq \Omega_{\mathbf{x}} = \{ \boldsymbol{\omega} \in \mathbf{F} \mid \log \omega_{i} = x_{i} \text{ for } 1 \leq i \leq r \}.
\end{equation}
Then \(\Omega_{\mathbf{x}}\) is an affinoid subspace of \(\Omega\), whose geometry depends on the position of \(\mathbf{x}\) in \(\mathcal{BT}(\mathds{Q})\). If for example \(\mathbf{x} \in \mathcal{BT}(\mathds{Z})\) then
\(\Omega_{\mathbf{x}}\) is isomorphic with \(\mathds{P}^{r-1}(C_{\infty}) \smallsetminus \bigcup B_{H}\), where \(B_{H}\) is a closed ball indexed by the \(\mathds{F}\)-hyperplanes \(H\) in \(\mathds{P}^{r-1}/\mathds{F}\), with corresponding intersection patterns, and if \(\mathbf{x}\) is in the interior \(\overset{\circ}{\sigma}\) of a simplex \(\sigma\) of maximal dimension \(r-1\), then \(\Omega_{\mathbf{x}}\) is isomorphic with the \((r-1)\)-fold product
\(\mathbf{S}^{r-1}\) of a sphere \(\mathbf{S} = \{ x \in C_{\infty} \mid \lvert x \rvert = 1 \}\), see [II] Theorem 2.4.

\subsection{} \label{Subsection.Nomenclature-for-holomorphic-function} Given a holomorphic function \(f\) on \(\Omega\), we define 
\begin{align*}
	\Omega(f)				&= \{ \boldsymbol{\omega} \in \Omega \mid f(\boldsymbol{\omega}) = 0\} \\
	\mathbf{F}(f)		&= \mathbf{F} \cap \Omega(f) \\
	\mathcal{BT}(f)	&= \lambda(\Omega(f)) \\
	\mathcal{A}(f)		&= \mathcal{A}(\mathds{Q}) \cap \mathcal{BT}(f)\\
	\mathcal{W}(f)		&= \mathcal{W}(\mathds{Q}) \cap \mathcal{BT}(f) = \lambda(\mathbf{F}(f)).
\end{align*}
Hence \(\mathcal{BT}(f)\) provides a coarse picture of the vanishing set \(\Omega(f)\) of \(f\). If \(f\) is a modular form for \(\Gamma = \GL(r,A)\) then
\begin{equation}
	\Omega(f) = \Gamma \mathbf{F}(f) \quad \text{and} \quad \mathcal{BT}(f) = \Gamma \mathcal{W}(f).
\end{equation}
When studying vanishing properties of \(f\), we may therefore restrict our attention to the behavior of \(f\) on \(\mathbf{F}\).

\subsection{} \label{Subsection.Convention-for-simplicial-complexes} Let \(\mathcal{S}\) be a simplicial subcomplex of \(\mathcal{BT}\). It is a \textbf{full subcomplex} if its simplices are the intersections of simplices 
\(\sigma\) of \(\mathcal{BT}\) with the vertex set \(\mathcal{S}(\mathds{Z})\) of \(\mathcal{S}\). Hence a full subcomplex \(\mathcal{S}\) (e.g., \(\mathcal{S} = \mathcal{A}\) or \(\mathcal{S} = \mathcal{W}\)) is determined by 
\(\mathcal{S}(\mathds{Z})\), and intersections of full subcomplexes are well-defined. It is \textbf{everywhere of codimension} 1 (or \textbf{equidimensional} of codimension 1, or of dimension \(r-2\)) if it has dimension 
\(r-2\) and each \(v \in \mathcal{S}(\mathds{Z})\) belongs to a simplex of \(\mathcal{S}\) of dimension \(r-2\). It is \textbf{strongly equidimensional} of codimension 1 (or of dimension \(r-2\)) if, stronger, each simplex of 
\(\mathcal{S}\) belongs to a simplex of dimension \(r-2\). If \(\mathcal{S}\) is strongly equidimensional of dimension \(r-2\) and moreover satisfies \stepcounter{equation}%

\subsubsection{} Each \((r-3)\)-simplex \(\sigma\) of \(\mathcal{S}\) is face of at least  two \( (r-2)\)-simplices \(\tau_{1}\) and \(\tau_{2}\) of \(\mathcal{S}\), then \(\mathcal{S}\) is called \textbf{boundaryless} 
(or \textbf{without boundary}).

If \fbox{\(r=2\)} then
\begin{center}
	\(\mathcal{S}\) is equidimensional of codimension 1 \(\Longleftrightarrow\) \(\mathcal{S}\) is zero-dimensional,
\end{center}
that is, it consists of a subset of \(\mathcal{BT}(\mathds{Z})\), and the condition \enquote{boundaryless} is empty.

If \fbox{\(r=3\)} then still simple and strong equidimensionality of codimension 1 agree, and \enquote{boundaryless} means that \(\mathcal{S}\) is a graph without endpoints.

If \fbox{\(r=4\)} then equidimensionality of codimension 1 allows subcomplexes \(\mathcal{S}\) like \begin{tikzpicture}[scale=0.3] \draw[fill=gray] (-0.5,0) -- (-1.5,-0.5) -- (-1.5,0.5) -- (-0.5,0) -- (0.5,0) -- (1.5,0.5) -- (1.5,-0.5) -- (0.5,0) -- cycle; \draw[fill=black](-0.5,0) circle (2pt) ; \draw[fill=black] (-1.5,-0.5) circle (2pt); \draw[fill=black] (-1.5,0.5) circle (2pt); \draw[fill=black] (0.5,0) circle (2pt); \draw[fill=black] (1.5,-0.5) circle (2pt); \draw[fill=black] (1.5,0.5) circle (2pt); \end{tikzpicture}. Such \(\mathcal{S}\) are excluded by requiring strong equidimensionality, but still \(\mathcal{S} = \) \begin{tikzpicture}[scale=0.3] \draw[fill=gray] (0,0) -- (-1,-0.5) -- (-1,0.5) -- (0,0) -- (1,-0.5) -- (1,0.5) -- (0,0) -- cycle; \draw[fill=black] (-1,-0.5) circle (2pt); \draw[fill=black] (-1,0.5) circle (2pt); \draw[fill=black] (0,0) circle (2pt); \draw[fill=black] (1,0.5) circle (2pt); \draw[fill=black] (1,-0.5) circle (2pt); \end{tikzpicture} is possible. This however is ruled out by requiring boundarylessness.

\subsection{} \label{Subsection.Simpliciality-of-modular-form} The modular form \(f\) for \(\Gamma\) is \textbf{simplicial} if \(\mathcal{BT}(f)\) is the set of \(\mathds{Q}\)-points of a full subcomplex of \(\mathcal{BT}\) 
(by abuse of language denoted by the same symbol \(\mathcal{BT}(f)\)) which is equidimensional of codimension 1. Once \(f\) is verified to be simplicial, we may study the vanishing set \(\Omega(f)\) and its image in the moduli 
scheme \(M^{r} = \Gamma \backslash \Omega\) via the simplicial complex \(\mathcal{BT}(f)\) and its quotient \(\Gamma \backslash \mathcal{BT}(f)\).

In the case \fbox{\(r=2\)}, simpliciality of \(f\) simply means \(\mathcal{BT}(f) \subset \mathcal{BT}(\mathds{Z})\). Assume \fbox{\(r \geq 3\)}. Then simpliciality of \(f\) is equivalent with the conjunction of conditions
(a), (b), (c) below about the set \(\mathcal{BT}(f)\). If these hold, let \(\mathcal{BT}(f)\) also denote the full subcomplex with vertex set \(\mathcal{BT}(f) \cap \mathcal{BT}(\mathds{Z})\). The conditions are:
\begin{enumerate}[label=(\alph*)]
	\item Let \(\sigma\) be an \((r-1)\)-simplex of \(\mathcal{BT}\) with interior \(\overset{\circ}{\sigma}\). Then \(\overset{\circ}{\sigma}(\mathds{Q}) \cap \mathcal{BT}(f) = \varnothing\);
	\item If \( \{ \mathbf{n}^{(0)}, \dots, \mathbf{n}^{(\ell)}\}\) is an \(\ell\)-simplex of \(\mathcal{BT}\) and \(\mathbf{x} \in \overset{\circ}{\sigma}(\mathds{Q})\), then
	\begin{center}
		\(\mathbf{x} \in \mathcal{BT}(f) \Longleftrightarrow \) for \(j = 0,1,\dots,\ell\), \(\mathbf{n}^{(j)} \in \mathcal{BT}(f)\);
	\end{center}
	\item Each \textbf{vertex} \(\mathbf{n} \in \mathcal{BT}(f) \cap \mathcal{BT}(\mathds{Z})\) belongs to a simplex of \(\mathcal{BT}(f)\) of dimension \(r-2\).
\end{enumerate}
These may be enforced through the conditions:
\begin{enumerate}[label=(\alph*)] \setcounter{enumi}{3}
	\item[(c strong)] \(\mathcal{BT}(f)\) is strongly equidimensional of codimension 1, that is, each \textbf{simplex} \(\sigma\) of \(\mathcal{BT}(f)\) belongs to a simplex of \(\mathcal{BT}(f)\) of dimension \(r-2\);
	\item \(\mathcal{BT}(f)\) is boundaryless;
	\item \(\mathcal{BT}(f)\) is connected.
\end{enumerate}

\subsection{}\label{Subsection.Conditions-to-be-shown-for-a-l-k} As easy examples show (see Example \ref{Example.W2k}), simpliciality of modular forms is not preserved under sums or (if \(r \geq 3\)) products. The more remarkable 
is the fact that all the naturally appearing modular forms \(f\) are simplicial. This has been shown in some cases for the following distinguished modular forms:
\begin{itemize}
	\item \(f = E_{k}\), the Eisenstein series of weight \(k\) (see \cite{Goss1980} or \cite{Gekeler1988});
	\item \(f = \alpha_{k}\), the para-Eisenstein series of weight \(q^{k}-1\) (\(\alpha_{k}(\boldsymbol{\omega})\) is the \(k\)-th coefficient of the expansion \( e_{\Lambda_{\boldsymbol{\omega}}}(z) = \sum_{k \geq 0} \alpha_{k}(\boldsymbol{\omega}) z^{q^{k}}\) of the exponential function \(e_{\Lambda_{\boldsymbol{\omega}}}\) of the \(A\)-lattice \(\Lambda_{\boldsymbol{\omega}}\));
	\item \(f = {}_{a}\ell_{k}\), the coefficient form as in \eqref{Eq.A-Division-Polynomial}, which has weight \(q^{k}-1\).
\end{itemize}

Namely, for \fbox{\(r=2\)} and \(f = E_{k}\) in \cite{Cornelissen1995}, \(f = \alpha_{k}\) in \cite{Gekeler1999}, \(f = {}_{a}\ell_{k}\) in \cite{Gekeler2011}, and for arbitrary \fbox{\(r \geq 2\)} and \(f = g_{k} = {}_{\tau} \ell_{k}\) in 
\cite{Gekeler2019}, \(f = E_{k}\) in [I], and \(f = \alpha_{k}\) or \(f = {}_{a}\ell_{k}\) with \(k \leq \deg a\) in [V].

It is the aim of the present study to establish simpliciality of the \({}_{a}\ell_{k}\) for all \(k\) and to sharpen all the existing results about \(\mathcal{BT}(f)\) by showing the conditions (c strong), (d), and (e) of 
\ref{Subsection.Conditions-to-be-shown-for-a-l-k} for all the forms in question. This leads to the following main theorem:

\begin{Theorem}~
	\begin{enumerate}[label=\(\mathrm{(\roman*)}\)]
		\item The modular forms \(\alpha_{k}\) are simplicial. Let \(\mathcal{BT}(k) \defeq \mathcal{BT}(\alpha_{k})\) be the corresponding full subcomplex of \(\mathcal{BT}\), with full subcomplexes
		\(\mathcal{A}(k) \defeq \mathcal{BT}(k) \cap \mathcal{A}\) and \(\mathcal{W}(k) \defeq \mathcal{BT}(k) \cap \mathcal{W}\).
		\item If \(r \geq 3\), \(\mathcal{BT}(k)\), \(\mathcal{A}(k)\), and \(\mathcal{W}(k)\) are connected.
		\item \(\mathcal{BT}(k)\), \(\mathcal{A}(k)\), and \(\mathcal{W}(k)\) are strongly equidimensional of codimension 1.
		\item \(\mathcal{BT}(k)\) and \(\mathcal{A}(k)\) are boundaryless.
		\item For a given \(a \in A\) of degree \(d \geq 1\) and \(1 \leq k < rd\) (which we from now on assume), the form \( {}_{a}\ell_{k}\) is simplicial. The corresponding subcomplex \(\mathcal{BT}({}_{a}\ell_{k})\) depends only on
		\(d\) and \(k\), and is therefore labelled \(\mathcal{BT}(d,k)\). Let \(\mathcal{A}(d,k) \defeq \mathcal{BT}(d,k) \cap \mathcal{A}\) and \(\mathcal{W}(d,k) \defeq \mathcal{BT}(d,k) \cap \mathcal{W}\).
		\item If \(k \leq d\), then \(\mathcal{W}(d,k) = \mathcal{W}(k)\), so \(\mathcal{A}(d,k) = \mathcal{A}(k)\) and \(\mathcal{BT}(d,k) = \mathcal{BT}(k)\).
		\item If \(r \geq 3\), \(\mathcal{BT}(d,k)\), \(\mathcal{A}(d,k)\), and \(\mathcal{W}(d,k)\) are connected.
		\item \(\mathcal{BT}(d,k)\), \(\mathcal{A}(d,k)\), and \(\mathcal{W}(d,k)\) are strongly equidimensional of codimension 1.
		\item \(\mathcal{BT}(d,k)\) and \(\mathcal{A}(d,k)\) are boundaryless.
		\item Let \( (\, . \,) \mapsto (\, . \,){}\hat{}\) be the simplicial involution on \(\mathcal{BT}\) that corresponds to the inversion on the Dynkin diagram of type \(A_{r-1}\), given on \(\mathcal{W}\) by
		\[
			\mathbf{x} = (x_{1}, x_{2}, \dots, x_{r} = 0) \longmapsto \mathbf{x}{}\hat{} = (x_{1} - x_{r}, x_{1} - x_{r-1}, \dots, x_{1} - x_{1} = 0)
		\]
		(which is trivial if \(r=2\)). Then \(\mathcal{W}(d,rd-k) = \mathcal{W}(d,k){}\hat{}\).
	\end{enumerate}
\end{Theorem}

\begin{Remarks}
	\begin{enumerate}[wide, label=(\roman*)]
		\item Items (i), (ii) and (vi) are in [V], and are included in the Main Theorem for clearness of exposition. This is also the reason for the pedantic repetition of (ii), (iii), (iv) for the complexes
		\(\mathcal{BT}(d,k)\), \(\mathcal{A}(d,k)\), and \(\mathcal{W}(d,k)\) in items (vii), (viii), and (ix).
		\item The intrinsically interesting objects are the \(\mathcal{BT}(k)\) and \(\mathcal{BT}(d,k)\), while the \(\mathcal{A}\)- and \(\mathcal{W}\)-variants are of an auxiliary nature, but important as the proofs take
		place on the \(\mathcal{A}\)- and \(\mathcal{W}\)-level.
		\item Boundarylessness may fail for \(\mathcal{W}(k)\) and \(\mathcal{W}(d,k)\) at the boundary of \(\mathcal{W}\), see Examples \ref{Example.W2k}, \ref{Example.W3k}, \ref{Example.W4k}.
	\end{enumerate}
\end{Remarks}

\subsection{} Now we give a brief outline of the paper. We first establish the mentioned vanishing criterion Theorem \ref{Theorem.Membership-via-d-characteristic-sequences}. Together with Proposition \ref{Proposition.Characterization-of-Wdk-as-set-of-real-points}, 
it reduces the study of the vanishing sets \(\mathcal{BT}(k) = \mathcal{BT}(\alpha_{k})\) and \(\mathcal{BT}(d,k) = \mathcal{BT}({}_{a}\ell_{k})\) to combinatorial questions in the simplicial complex \(\mathcal{BT}\). In 
Section \ref{Section.Wkd-as-subcomplex-of-W} we collect some properties of
\(\mathcal{BT}\), mainly the interplay of root data and the simplicial structure. This enables us to show that \(\mathcal{BT}({}_{a}\ell_{k})\) is in fact (the set of \(\mathds{Q}\)-points of) a full subcomplex of \(\mathcal{BT}\)
(Proposition \ref{Proposition.Characterization-of-Wdk-as-set-of-real-points}).

In Section \ref{Section.Diagrams-of-vertices} the concept of the diagram \(\diag(\mathbf{n})\) or \(d\)-diagram \(\diag^{(d)}(\mathbf{n})\) of a vertex \(\mathbf{n}\) of \(\mathcal{A}\) is introduced. It encodes the properties of 
\(\mathbf{n}\) relevant to deciding whether or not \(\mathbf{n}\) belongs to \(\mathcal{A}(k)\) (resp. \(\mathcal{A}(d,k)\)). After these preparations we can show in Section \ref{Section.Connectedness} the first significant result, 
the connectedness of \(\mathcal{BT}(d,k)\) (Theorem \ref{Theorem.Simplicial-complexes-connected}). It depends on the decomposition \ref{Subsection.Decomposition-of-WdkIZ} of \(\mathcal{W}(d,k)\), an important technical 
step, which is also crucial for establishing properties (c strong) and (d) of \ref{Subsection.Conditions-to-be-shown-for-a-l-k} for \(\mathcal{BT}(d,k)\). These are shown in Section \ref{Section.Strong-equidimensionality} 
(Theorem \ref{Theorem.Simplicial-complexes-strongly-equidimensional}) and \ref{Section.Boundarylessness} (Theorem \ref{Theorem.Simplicial-complexes-boundaryless}), respectively.

The proofs are elementary in a technical sense; the difficulty is to find and elaborate the right structures and criteria that allow to reduce the vast number of cases to a managable level. Besides the mentioned ingredients, we 
basically need to control the behavior of the \(d\)-diagram \(\diag^{(d)}(\mathbf{n})\) under replacing the vertex \(\mathbf{n}\) by one of its neighbors \(\mathbf{n}'\). A systematic study of this question would be laborious,
boring, and partially superfluous; therefore we treat only some special cases, embedded in the text where they are needed (Propositions \ref{Proposition.Sums-of-vertices}, \ref{Proposition.d-characteristic-sequences-of-certain-vertices}, 5.4, items \ref{Subsection.Special-vertex-in-1Wdk}, \ref{Subsection.Second-special-vertex}, \ref{Subsection.Fourth-special-vertex}, Lemmas \ref{Lemma.Boxes-of-sums-of-certain-vertices}, \ref{Lemma.Membership-of-sums-of-vertices}).

In Section \ref{Section.Involution} we describe the involution \( (\,.\,){}\hat{}\) on \(\mathcal{BT}\) and show the symmetry property (x) of the Main Theorem. The final Section \ref{Section.Examples} is devoted to some examples and to concluding remarks.

\begin{Note}[on notation]
	Quite generally, we use the same symbol, say \(\mathcal{S}\), for a simplicial complex \(\mathcal{S}\) and its realization \(\mathcal{S}(\mathds{R})\). In practice, \(\mathcal{S}\) will be one of \(\mathcal{BT}\), \(\mathcal{A}\), 
	\(\mathcal{W}\) or \(\mathcal{BT}(f)\), \(\mathcal{A}(k)\), \(\mathcal{W}(f)\), see \ref{Subsection.Nomenclature-for-holomorphic-function}. As only the \(\mathds{Q}\)-points of \(\mathcal{S}\) play a role, we will often not distinguish in notation between the complex 
	\(\mathcal{S}\), its set \(\mathcal{S}(\mathds{R})\)
	of real points, and the set \(\mathcal{S}(\mathds{Q})\) of points with rational barycentric coordinates, and simply write \enquote{\(\mathcal{S}\)} for all three. It should always be clear from the context which meaning is intended.
	
	Note that, e.g., the apartment \(\mathcal{A}\) is at the same time a full subcomplex of \(\mathcal{BT}\), a (rational or real) point set, and an affine euclidean space (made a rational or real vector space by fixing the origin \(\mathbf{0}\)).
	For the sake of simplicity, all these structures, although phenomenologically different, will often be denoted by the same symbol. Also the rank \(r \in \mathds{N}\) which determines \(\mathcal{BT} = \mathcal{BT}^{r}\),
	\(\mathcal{A} = \mathcal{A}^{r}\) etc., will mostly be omitted from notation.
\end{Note}

\section{The vanishing criterion for \({}_{a}\ell_{k}\)} \label{Section.Vanishing-criterion-for-a-l-k}

\subsection{} In what follows, we fix an element of \(A\) of degree \(d \in \mathds{N}\). (Recall that \(\mathds{N} = \{1,2,3,\dots\}\), while we use \(\mathds{N}_{0}\) for \( \{0,1,2,\dots\}\).) The function \({}_{a}\ell_{k}\) on \(\Omega\), where
\(1 \leq k \leq rd\), is modular of weight \(q^{k}-1\) and type 0 for \(\Gamma = \GL(r,A)\). Our starting point is Theorem 3.2 in [V], which states:
\begin{quote}
	\(\mathbf{x} \in \mathcal{BT}(\mathds{Q})\) belongs to \(\mathcal{BT}({}_{a}\ell_{k})\) if and only if for one (or for each) \(\omega \in \Omega_{\mathbf{x}}\) the \(\mathds{F}\)-lattice \( {}_{a}\phi^{\boldsymbol{\omega}}\) is
	\(k\)-inseparable.
\end{quote}
As it is important for the sequel, we break down the latter statement.

\subsection{} For \(\boldsymbol{\omega} = (\omega_{1}, \dots, \omega_{r} = 1) \in \Omega\), let \(\Lambda_{\boldsymbol{\omega}}\) be the \(A\)-lattice
\[
	\Lambda_{\boldsymbol{\omega}} = \sum_{1 \leq i \leq r} A\omega_{i}
\]
with exponential function
\[
	e_{\boldsymbol{\omega}}(z) = z \sideset{}{^{\prime}} \prod_{\mathbf{a} \in A^{r}}  \left( 1 - \frac{z}{\mathbf{a}\boldsymbol{\omega}} \right)
\]
and associated Drinfeld module \(\phi^{\boldsymbol{\omega}}\). (Here \(\mathbf{a}\boldsymbol{\omega} = \sum_{1 \leq i \leq t} a_{i} \omega_{i}\), and \(\textstyle\sideset{}{^{\prime}}\prod\) is the product over the non-zero \(\mathbf{a}\).)
Its \(a\)-torsion is
\begin{equation}
	{}_{a}\phi^{\boldsymbol{\omega}} \defeq \{ z \in C_{\infty} \mid \phi_{a}^{\boldsymbol{\omega}}(z) = 0\}.
\end{equation}
This is an \(\mathds{F}\)-vector space of dimension \(rd\), with basis
\begin{equation}
	{}_{a} B_{\boldsymbol{\omega}} \defeq \{ e_{s,i}(\boldsymbol{\omega}) \mid 0 \leq s < d, 1 \leq i \leq r \},
\end{equation}
where
\begin{equation}
	e_{s,i}(\boldsymbol{\omega}) \defeq e_{\boldsymbol{\omega}}(T^{s}\omega_{i}/a).
\end{equation}
Suppose that \(\boldsymbol{\omega} \in \mathbf{F}\), and arrange the elements of \( {}_{a}B_{\boldsymbol{\omega}}\) in increasing order \( \lvert \lambda_{1} \rvert \leq \lvert \lambda_{2} \rvert \leq \dots \leq \lvert \lambda_{rd} \rvert\)
where, if \( \lvert e_{s,i}(\boldsymbol{\omega}) \rvert = \lvert e_{s',i'}(\boldsymbol{\omega}) \rvert\), \(e_{s,i}(\boldsymbol{\omega})\) is decreed to come before \(e_{s',i'}(\boldsymbol{\omega})\) if and only if \(i > i'\). Then, e.g.,
\(\lambda_{1} = e_{0,r}(\boldsymbol{\omega})\) and \(\lambda_{rd} = e_{d-1,1}(\boldsymbol{\omega})\), while the position of the intermediate \(e_{s,i}(\boldsymbol{\omega})\) depends on \(\boldsymbol{\omega}\). Now
\({}_{a}\phi^{\boldsymbol{\omega}}\) is \(k\)-\textbf{inseparable} means that \(\lvert \lambda_{k} \rvert = \lvert \lambda_{k+1} \rvert\). This is still difficult to handle, but the following observation is helpful.

\begin{Proposition} \label{Proposition.Monotonicity-of-modular-form-map}
	Let \(\boldsymbol{\omega} \in \mathbf{F}\). The map
	\begin{align*}
		\{ T^{s} \omega_{i} \mid 0 \leq s < d, 1 \leq i \leq r \} 			&\longrightarrow {}_{a}\phi^{\boldsymbol{\omega}} \\
																			T^{s}\omega_{i}	&\longmapsto e_{s,i}(\boldsymbol{\omega})
	\end{align*}
	is strictly monotonically increasing in the sense that \( \lvert T^{s} \omega_{i} \rvert \leq \lvert T^{s'} \omega_{i'} \rvert\) implies \( \lvert e_{s,i}(\boldsymbol{\omega}) \rvert \leq \lvert e_{s',i'}(\boldsymbol{\omega}) \rvert\), with
	(in-)equality on the left hand side implying (in-)equality in \( \lvert e_{s,i}(\boldsymbol{\omega}) \rvert \leq \lvert e_{s',i'}(\boldsymbol{\omega}) \rvert\).
\end{Proposition}

\begin{proof}
	With notations as before, 
	\begin{equation}
		e_{s,i}(\boldsymbol{\omega}) = \frac{T^{s}\omega_{i}}{a} \sideset{}{^{\prime}} \prod_{\mathbf{a} \in A^{r}} \left( 1 - \frac{T^{s}\omega_{i}/a}{\mathbf{a}\boldsymbol{\omega}} \right).
	\end{equation}
	For one of the factors of \(\sideset{}{^{\prime}}\prod\), the absolute value is
	\begin{equation}
		\left\lvert 1- \frac{T^{s}\omega_{i}/a}{\mathbf{a}\boldsymbol{\omega}} \right\rvert = \begin{cases} 1,	&\text{if } \lvert \mathbf{a} \boldsymbol{\omega} \geq \lvert T^{s}\omega_{i}/a \rvert, \\ \left\lvert \frac{T^{s}\omega_{i}/a}{\mathbf{a}\boldsymbol{\omega}} \right \rvert,	&\text{if } \lvert \mathbf{a} \boldsymbol{\omega} \rvert < \lvert T^{s} \omega_{i}/ a \rvert. \end{cases}
	\end{equation}
	Here the value 1 in case \( \lvert \mathbf{a} \boldsymbol{\omega} \rvert = \lvert T^{s}\omega_{i}/a \rvert\) results from \(\boldsymbol{\omega} \in \mathbf{F}\) and the orthogonality of the set
	\( \{ \omega_{1}, \dots, \omega_{r} \}\), see, e.g., \cite{Gekeler2019} Proposition 3.1. If we put for \(x \in \mathds{Q}_{\geq 0}\):
	\begin{equation}
		\varphi(x) \defeq x \sideset{}{^{\prime}} \prod_{\substack{\mathbf{a} \in A^{r} \\ \lvert \mathbf{a} \boldsymbol{\omega} \rvert < x}} \frac{x}{\lvert \mathbf{a} \boldsymbol{\omega} \rvert},
	\end{equation}
	then by the above, \(\lvert e_{s,i}(\boldsymbol{\omega}) = \varphi(\lvert T^{s}\omega_{i}/a \rvert ) \). But \(\varphi\) is certainly monotonically increasing, whence the result.
\end{proof}

\subsection{} \label{Subsection.Detectability-of-k-inseparability} The proposition means that we can detect \(k\)-inseparability of \({}_{a}\phi^{\boldsymbol{\omega}}\) by means of the (simple) family \( \{ T^{s} \omega_{i} \}\), 
or rather on its logarithms \( \{x_{i} + s\} \), where
\begin{equation}
	x_{i} = \log \omega_{i}, \quad \mathbf{x} = (x_{1}, \dots, x_{r}) = \lambda(\boldsymbol{\omega}) \in \mathcal{W}(\mathds{Q}),
\end{equation}
instead of the (complicated) family \( \{e_{s,i}(\boldsymbol{\omega}) \} \).

\begin{Definition} \label{Definition.d-characteristic-sequence}
	Let \(\mathbf{x} = (x_{1}, \dots, x_{r}) \in \mathcal{W}(\mathds{Q})\), so \(x_{1} \geq x_{2} \geq \dots \geq x_{r} = 0\). Consider the unique map
	\begin{align}
		v^{(d)} \colon \{ 1,2, \dots, rd \} 	&\longrightarrow \{ x_{i} + s \mid 1 \leq i \leq r, 0 \leq s < d \} \\
														k		&\longmapsto v_{k}^{(d)} \nonumber
	\end{align}
	which is monotonically increasing and satisfies
	\[
		\#(v^{(d)})^{-1}(x_{i} + s) = \# \{ (s', i') \mid x_{i'} + s' = x_{i} + s \}.
	\]
	That is, the sequence \( (v_{1}^{(d)}, v_{2}^{(d)}, \dots, v_{rd}^{(d)} )\) is the multiset \( \{x_{i} + s\} \), ordered increasingly, and taking multiplicities into account. We call
	 \( (v_{1}^{(d)}, \dots, v_{rd}^{(d)}) = (v_{1}^{(d)}(\mathbf{x}), \dots, v_{rd}^{(d)}(\mathbf{x}))\) the \(d\)-\textbf{characteristic sequence} of \(\mathbf{x}\).
\end{Definition}

Now we can state Theorem 3.2 of [V] as follows.

\begin{Theorem} \label{Theorem.Membership-via-d-characteristic-sequences}
	Let \(\mathbf{x} \in \mathcal{W}(\mathds{Q})\) be as before and \(1 \leq k < rd\). Then \(\mathbf{x} \in \mathcal{W}({}_{a}\ell_{k})\) (that is, there exists \(\mathbf{\omega} \in \mathbf{F}_{\mathbf{x}} = \lambda^{-1}(\mathbf{x})\) 
	such that \({}_{a}\ell_{k}(\mathbf{\omega}) = 0\)) if and only if \( v_{k}^{(d)}(\mathbf{x}) = v_{k+1}^{(d)}(\mathbf{x})\).
\end{Theorem}

\begin{proof}
	This follows from the cited theorem and \ref{Proposition.Monotonicity-of-modular-form-map}-\ref{Definition.d-characteristic-sequence}, using the monotonicity of the logarithm.
\end{proof}

\begin{Remarks}
	\begin{enumerate}[wide, label=(\roman*)]
		\item There is no condition for \(k = rd\). Actually, \({}_{a}\ell_{rd}\) is a power of the discriminant function \(\Delta \defeq {}_{T}\ell_{r} = g_{r}\) and vanishes nowhere on \(\Omega\).
		\item The dependence on \(a\) enters only via \(d \defeq \deg a\). (This has already been observed in [V] Proposition 3.8(i).) Therefore we now write \(\mathcal{W}(d,k)\) for the subset \(\mathcal{W}({}_{a}\ell_{k})\) 
		of \(\mathcal{W}(\mathds{Q})\). (For the moment this is merely a set. Its simplicial structure will be studied in the next section.)
		\item The present \(d\)-characteristic sequence of \(\mathbf{x}\) is strongly related to the characteristic sequence of \( {}_{d}B_{\mathbf{x}}\) as defined in [V] 3.9. The difference is that we now make use of the
		simplification that results from replacing \(e_{s,i}(\boldsymbol{\omega})\) by \(T^{s}\omega_{i}\) or its logarithm \(x_{i}+s\).
	\end{enumerate}
\end{Remarks}

\subsection{} Let \(d\) tend to infinity and consider the map
\begin{align}
	v \colon \mathds{N} 	&\longrightarrow \{ x_{i} + s \mid 1 \leq i \leq r, s \in \mathds{N}_{0} \} \\
									k		&\longmapsto v_{k} = v_{k}(\mathbf{x}) \nonumber
\end{align}
defined analogously, i.e., \(v_{k} = \lim_{d \to \infty} v_{k}^{(d)}\). Then \( (v_{1}, v_{2}, \dots )\) is the \(q\)-logarithm of \( (\nu_{\mathbf{x}}(\lambda_{k}))_{k \in \mathds{N}} \), where \( (\lambda_{k})_{k \in \mathds{N}}\) is the 
characteristic sequence of \(\mathbf{x}\) as defined in [V] 2.4 that controls the zero sets \( \mathcal{W}(k) = \mathcal{W}(\alpha_{k}) \) of the para-Eisenstein series \(\alpha_{k}\). From now on, we use the notion 
\enquote{characteristic sequence of \(\mathbf{x}\)} only for the logarithmic variant \( (v_{k})_{k \in \mathds{N}}\). It enables us to give the following synopsis of [II] Theorem 4.8 and [V] 2.4, 2.8, 2.9, 2.10:
\subsection{} \label{Subsection.Characterization-of-k-inseparability} Fix \(k \in \mathds{N}\) and \(\mathbf{x} \in \mathcal{W}(\mathds{Q})\). There is equivalence between 
\begin{enumerate}[label=(\alph*)]
	\item \(\mathbf{x}\) is \(k\)-inseparable;
	\item for each \(\boldsymbol{\omega} \in \mathbf{F}_{\mathbf{x}} = \lambda^{-1}(\mathbf{x})\), the lattice \(\Lambda_{\boldsymbol{\omega}}\) is \(k\)-inseparable;
	\item there exists \(\boldsymbol{\omega} \in \mathbf{F}_{\mathbf{x}}\) such that \(\alpha_{k}(\boldsymbol{\omega})\) vanishes;
	\item \(\mathbf{x} \in \mathcal{W}(k)\);
	\item \(v_{k}(\mathbf{x}) = v_{k+1}(\mathbf{x})\).
\end{enumerate}
Moreover, the set \(\mathcal{W}(k)\) of all \(\mathbf{x} \in \mathcal{W}(\mathds{Q})\) with properties (a)-(e) is the set of \(\mathds{Q}\)-points of a full subcomplex (labelled by the same symbol \(\mathcal{W}(k)\)) of 
\(\mathcal{W}\), which is equidimensional of codimension 1.

\begin{Remark} \label{Remark.Coincidence-of-Wk-and-Wdk}
	The addendum just stated says that \( \alpha_{k} \) is simplicial. It is easily seen that \( v_{k}^{(d)} = v_{k} \) for \( k \leq d+1\). Therefore, by Theorem \ref{Theorem.Membership-via-d-characteristic-sequences} and \ref{Subsection.Characterization-of-k-inseparability}, the sets (and therefore the simplicial complexes)
	\(\mathcal{W}(d,k)\) and \(\mathcal{W}(k)\) must agree as long as \(k \leq d = \deg a\). This is essentially the content of [V] Theorem 3.12, but the proof given there is considerably more complicated, as the author wasn't aware of the 
	simplification provided by Proposition \ref{Proposition.Monotonicity-of-modular-form-map}.
\end{Remark}

\section{\(\mathcal{W}(k,d) = \mathcal{W}({}_{a}\ell_{k})\) as a subcomplex of \(\mathcal{W}\)} \label{Section.Wkd-as-subcomplex-of-W}

From now on, we suppose that \(r \geq 3\). We want to verify the conditions (a) to (e) of \ref{Subsection.Conditions-to-be-shown-for-a-l-k} for \(f = {}_{a}\ell_{k}\) and, as far as this has not yet been done, for \(f = \alpha_{k}\). Using the
action of \(\Gamma\), we may work in the subcomplex \(\mathcal{A}\) or even \(\mathcal{W}\) instead of \(\mathcal{BT} = \mathcal{BT}^{r}\).

\subsection{} First we collect some facts about the simplicial complex \(\mathcal{A} = \mathcal{A}^{r}\). It is useful to have these in mind even though we will not all of them use in a technical sense.

We briefly write \(\mathbf{n}\) for the class \( [L_{\mathbf{n}}] \in \mathcal{A}(\mathds{Z})\), where \(\mathbf{n} = (n_{1}, \dots, n_{r-1}, 0)\) with \(n_{i} \in \mathds{Z}\). Thus \(\mathcal{A}(\mathds{Z}) \cong \mathds{Z}^{r-1}\)
forms a \(\mathds{Z}\)-lattice in the euclidean space \(\mathcal{A}(\mathds{R}) \cong \mathds{R}^{r-1}\), the lattice of \textbf{coroots}. For \(1 \leq i,j \leq r\), \(i \neq j\), define the \textbf{root}
\begin{align}
	\boldsymbol{\alpha}_{i,j} \colon \mathcal{A}(\mathds{R})	&\longrightarrow \mathds{R} \\
																				\mathbf{x}				&\longmapsto x_{i} - x_{j} \nonumber
\end{align} \stepcounter{subsubsection}%
and for \(1 \leq i < r\) the \textbf{simple root} \(\boldsymbol{\alpha}_{i} \defeq \boldsymbol{\alpha}_{i,i+1}\). We regard \(\mathcal{A}(\mathds{R})\) as the sub-algebra \(\mathfrak{t}\) of diagonal elements in the simple
Lie algebra \(\mathfrak{pgl}(r, \mathds{R})\). Standard theory of Lie algebras and root systems (our prefered reference is \cite{Humphreys1972}) provides the scalar product \( (\,.\,,\,.\,)\) on \(\mathfrak{t}\), derived from the
Killing form, and thus the euclidean structure on \(\mathcal{A}(\mathds{R})\). Then
\begin{equation}
	\Phi \defeq \{ \boldsymbol{\alpha}_{i,j} \mid 1 \leq i \neq j \leq r \}
\end{equation} \stepcounter{subsubsection}%
is a root system of type \(A_{r-1}\) in the dual space \(\mathcal{A}(\mathds{R})^{*} = \mathfrak{t}^{*}\) with \( \{\boldsymbol{\alpha}_{i} \mid 1 \leq i < r\}\) as a root basis and the decomposition \( \Phi = \Phi^{+} \cupdot \Phi^{-}\) into
positive and negative roots, where \( \Phi^{+} = \{ \boldsymbol{\alpha}_{i,j} \mid i < j\}\). We shall not need the precise description of \( (\,.\,,\,.\,)\), which may be looked up in \cite{Humphreys1972}. Obviously,
\begin{equation}
	\mathcal{A}(\mathds{Z}) = \{ \mathbf{x} \in \mathcal{A}(\mathds{R}) \mid \forall \boldsymbol{\alpha} \in \Phi : \boldsymbol{\alpha}(\mathbf{x}) \in \mathds{Z} \},
\end{equation}\stepcounter{subsubsection}%
which explains its description as lattice of coroots, and 
\begin{equation} \label{Eq.Weyl-chamber-of-realization}
	\mathcal{W}(\mathds{R}) = \{ \mathbf{x} \in \mathcal{A}(\mathds{R}) \mid \boldsymbol{\alpha}_{i}(\mathbf{x}) \geq 0 \text{ for } i=1,2, \dots, r-1 \}
\end{equation}\stepcounter{subsubsection}%
with \(i\)-th wall
\[
	\mathcal{W}_{i}(\mathds{R}) = \{ \mathbf{x} \in \mathcal{W}(\mathds{R}) \mid \boldsymbol{\alpha}_{i}(\mathbf{x}) = 0\}.
\]
For each \(\boldsymbol{\alpha} \in \Phi\) we let
\begin{equation}
	H_{\boldsymbol{\alpha}} \defeq \{ \mathbf{x} \in \mathcal{A}(\mathds{R}) \mid \boldsymbol{\alpha}(\mathbf{x}) = 0\}.
\end{equation}\stepcounter{subsubsection}%
Now \( (\,.\,,\,.\,) \) identifies the dual \(\mathcal{A}(\mathds{R})^{*}\) with \(\mathcal{A}(\mathds{R})\); therefore we regard \(\Phi\) as a subset of \(\mathcal{A}(\mathds{R})\) and the numbers \(\boldsymbol{\alpha}(\mathbf{x})\)
in the above formulas as scalar products \( (\boldsymbol{\alpha}, \mathbf{x})\). The following facts for root systems of type \(A_{r-1}\) are well-known (see \cite{Humphreys1972}):
\subsubsection{} All the roots \(\boldsymbol{\alpha} \in \Phi\) have equal length;
\subsubsection{} Let \(\mathds{Z}\Phi\) be the \(\mathds{Z}\)-lattice generated by \(\Phi\). Then \(\mathds{Z}\Phi\) is contained in \(\mathcal{A}(\mathds{Z})\) with index \(r\);
\subsubsection{} The \textbf{Cartan numbers} \(\langle \boldsymbol{\alpha}, \boldsymbol{\beta} \rangle \defeq 2 (\boldsymbol{\alpha}, \boldsymbol{\beta})/(\boldsymbol{\beta}, \boldsymbol{\beta})\) take only values in
\( \{0, \pm 1, \pm 2\}\), where \( \langle \boldsymbol{\alpha}, \boldsymbol{\beta} \rangle = \pm 2\) corresponds to \(\boldsymbol{\alpha} = \pm \boldsymbol{\beta}\). This implies that non-proportional roots are either orthogonal
or define an angle of \(60 \si{\degree}\) or \(120 \si{\degree}\). \stepcounter{equation}\stepcounter{equation}\stepcounter{equation}%

For \( 1 \leq i < r \) define \(\mathbf{n}_{i} \defeq (1,1,\dots, 1,0,\dots,0) \in \mathcal{W}(\mathds{Z})\), with \(i\) 1's. Then
\begin{equation} \label{Eq.Fundamental-dominant-weights}
	(\boldsymbol{\alpha}_{i}, \mathbf{n}_{j}) = \delta_{i,j},
\end{equation} \stepcounter{subsubsection}%
(in the language of root systems, the \(\mathbf{n}_{i}\) are the \textbf{fundamental dominant weights}), and \(\{\mathbf{n}_{i} \mid 1 \leq i < r\}\) is an \(\mathds{N}_{0}\)-basis of the monoid \(\mathcal{W}(\mathds{Z})\). 
Occasionally, we also write \(\mathbf{n}_{0}\) for the zero vector \(\mathbf{0} = (0,\dots,0)\). We further let \(\mathbf{e}_{i} \defeq (0,\dots,0,1,0,\dots,0)\) be the \(i\)-th standard basis vector of \(\mathcal{A}(\mathds{Z})\), 
\(1 \leq i < r\). Throughout, \enquote{\(\leq\)} denotes the product order on \(\mathcal{A}(\mathds{Z})\),
\begin{equation}
	\mathbf{m} \leq \mathbf{n} \Longleftrightarrow \text{ for } 1 \leq i < r, m_{i} \leq n_{i}.
\end{equation} \stepcounter{subsubsection}%
If \(\mathbf{m} \leq \mathbf{n}\), we write \([\mathbf{m}, \mathbf{n}] \defeq \{ \mathbf{a} \in \mathcal{A}(\mathds{Z}) \mid \mathbf{m} \leq \mathbf{a} \leq \mathbf{n}\}\). Then we have:
\subsubsection{} \label{Subsubsection.Vertices-are-neighbors} Two vertices \(\mathbf{m} \neq \mathbf{n}\) of \(\mathcal{A}\) are neighbors if and only if \(\mathbf{m} < \mathbf{n}\) or \(\mathbf{n} < \mathbf{m}\) 
and for all \(i\), \(\lvert m_{i} - n_{i} \rvert \in \{0,1\}\). By definition this means that the \textbf{combinatorial distance} \(d(\mathbf{m}, \mathbf{n})\) equals 1. For general \(\mathbf{m}\), \(\mathbf{n}\), \(d(\mathbf{m}, \mathbf{n})\) 
is defined as the length of the shortest edge path joining \(\mathbf{m}\) and \(\mathbf{n}\). (Note that \(d(\mathbf{m}, \mathbf{n})\) differs from the euclidean distance!) As a consequence, a simplex \(\sigma\) of \(\mathcal{A}\) 
is a totally ordered subset of \(\mathcal{A}(\mathds{Z})\). If \(\mathbf{m} \defeq \min \sigma\), \(\mathbf{n} \defeq \max \sigma\), then 
\(\sigma \subset [\mathbf{m}, \mathbf{n}] \subset [\mathbf{m}, \mathbf{m} + \mathbf{n}_{r-1}]\). \stepcounter{equation}%

For a root \(\boldsymbol{\alpha}\), we let \(s_{\boldsymbol{\alpha}}\) be the reflection at \(H_{\boldsymbol{\alpha}}\):
\begin{align}
	s_{\boldsymbol{\alpha}} \colon \mathcal{A}(\mathds{R})	&\longrightarrow \mathcal{A}(\mathds{R}) \\
																	\mathbf{x}							&\longmapsto \mathbf{x} - 2 \frac{(\boldsymbol{\alpha}, \mathbf{x})}{(\boldsymbol{\alpha}, \boldsymbol{\alpha})} \boldsymbol{\alpha}. \nonumber
\end{align}\stepcounter{subsubsection}%
The \(s_{\boldsymbol{\alpha}}\) generate the \textbf{Weyl group} \(W \cong S_{r}\), which permutes the coordinates, and \(s_{\boldsymbol{\alpha}_{i}}\) corresponds to the transposition \( (i,i+1)\). Due to our normalization \(x_{r} = 0\), this
may look a bit bumpy. For example,
\subsubsection{} \(s_{\boldsymbol{\alpha}_{r-1}}(x_{1}, \dots, x_{r-1}, x_{r} = 0) = \) class of \( (x_{1}, x_{2}, \dots, x_{r-2}, 0, x_{r-1}), \) which equals \( (x_{1} - x_{r-1}, \dots, x_{r-2} - x_{r-1}, - x_{r-1}, 0)\) in normalized form. \stepcounter{equation}%

Now \(W\) regarded as the group of permutation matrices sits inside \(\Gamma = \GL(r,A)\) or even inside \( \PGL(r,A)\), and \(\mathcal{W}\) is a fundamental domain for the action of \(W\) on \(\mathcal{A}\). For a simplex \(\sigma\) of
\(\mathcal{A}\), we let \(\langle \sigma \rangle\) be the affine subspace of \(\mathcal{A}(\mathds{R})\) generated by \(\sigma\). It is an intersection of finitely many hyperplanes parallel with root hyperplanes \(H_{\boldsymbol{\alpha}}\).
If \(\codim \sigma = 1\), we let
\begin{equation} \label{Eq.Reflection-at-<-sigma->}
	s_{\sigma} \colon \mathcal{A}(\mathds{R}) \longrightarrow \mathcal{A}(\mathds{R})
\end{equation}
by the reflection at \(\langle \sigma \rangle \), a simplicial map (that is, it preserves simplices). There exist unique \(\boldsymbol{\alpha} \in \Phi^{+}\) and \(n \in \mathds{Z}\) such that 
\(\langle \sigma \rangle = \{ \mathbf{x} \in \mathcal{A}(\mathds{R}) \mid \boldsymbol{\alpha}(\mathbf{x}) = n\}\). The set \( \{ s_{\sigma} \mid \sigma \text{ a simplex of codimension 1}\}\) generates what is called the 
\textbf{affine Weyl group} \(\tilde{W}\) of \(\Omega\), which may be described as the semi-direct product \(\tilde{W} = W \ltimes \mathcal{A}(\mathds{Z})\). Then \(s_{\sigma}\) is conjugate in \(\tilde{W}\) with \(s_{\boldsymbol{\alpha}}\).
Here we show the picture for \(r = 3\):

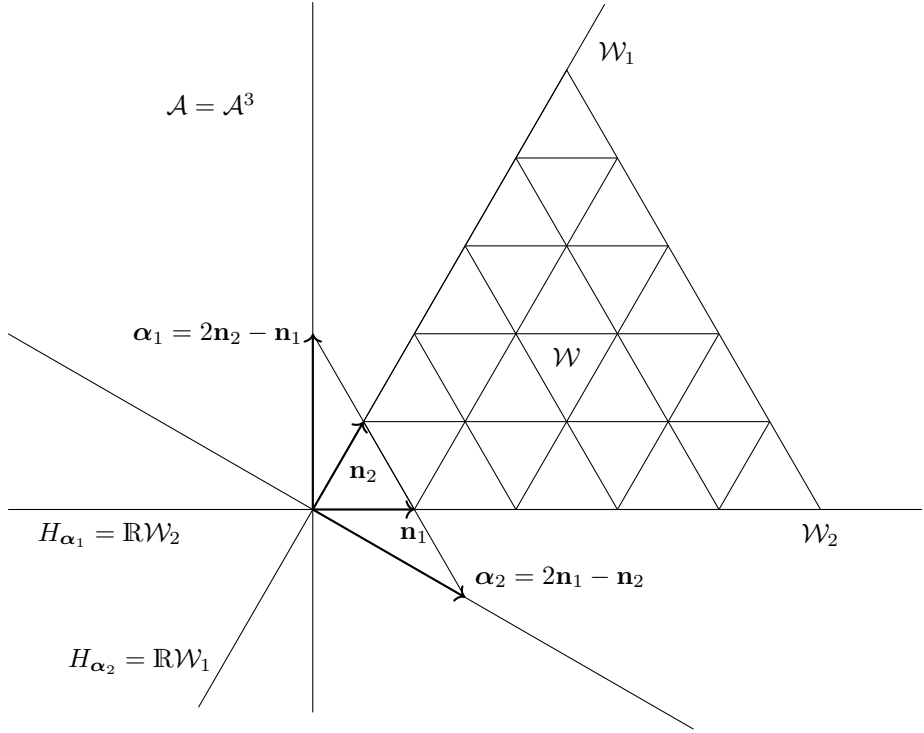
\begin{Picture} \label{Picture.Weyl-Chamber} ~
 	\begin{center}
 		\begin{tikzpicture}[scale=1.35]
 		  	\foreach \row in {0,1, ...,5} {
        			\draw ($\row*(0.5, {0.5*sqrt(3)})$) -- ($(\rows,0)+\row*(-0.5, {0.5*sqrt(3)})$);  
        			\draw ($\row*(1, 0)$) -- ($(\rows/2,{\rows/2*sqrt(3)})+\row*(0.5,{-0.5*sqrt(3)})$); 
        			\draw ($\row*(1, 0)$) -- ($(0,0)+\row*(0.5,{0.5*sqrt(3)})$); 
   		 	}
   		 	\draw (-3,0) -- (6,0);
   		 	\draw (0,-2) -- (0,5);
   			\draw (0,{sqrt(3)}) -- ($(1.5, {-0.5*sqrt(3)})$); 
   		 	\draw ($2.5*(1.5, {-0.5*sqrt(3)})$) -- ($-2*(1.5, {-0.5*sqrt(3)})$);
   		 	\draw ($-0.45*(2.5,{2.5*sqrt(3)})$) -- ($1.15*(2.5,{2.5*sqrt(3)})$);
   		 	\draw[->, thick] (0,0) -- ($(0.5,{0.5*sqrt(3)})$); 
   		 	\draw[->, thick] (0,0) -- (1,0); 
   		 	\draw[->, thick] (0,0) -- (0, {sqrt(3)}) node[left] {\(\boldsymbol{\alpha}_{1} = 2\mathbf{n}_{2}-\mathbf{n}_{1}\)}; 
   		 	\draw[->, thick] (0,0) -- ($(1.5, {-0.5*sqrt(3)})$) node[above right] {\(\boldsymbol{\alpha}_{2} = 2\mathbf{n}_{1} - \mathbf{n}_{2}\)}; 
   		 	\node (n1) at (1,-0.25) {\(\mathbf{n}_{1}\)};
   		 	\node (n2) at (0.5, 0.35) {\(\mathbf{n}_{2}\)};
   		 	\node (W1) at (3,4.5) {\(\mathcal{W}_{1}\)};
   		 	\node (W2) at (5,-0.25) {\(\mathcal{W}_{2}\)};
   		 	\node (W) at (2.5,1.5) {\(\mathcal{W}\)};
   		 	\node (A3) at (-1,4) {\(\mathcal{A} = \mathcal{A}^{3}\)};
   		 	\node (Ha1) at (-2,-0.25) {\(H_{\boldsymbol{\alpha}_{1}} = \mathds{R}\mathcal{W}_{2}\)};
   		 	\node (Ha2) at (-1.7,-1.5) {\(H_{\boldsymbol{\alpha}_{2}} = \mathds{R}\mathcal{W}_{1}\)};
   		\end{tikzpicture}
 	\end{center}
\end{Picture}

\subsection{} Having now the necessary background and notations, we start to investigate conditions (a), \dots, (e) for \(f = {}_{a}\ell_{k}\) and, as far as necessary, for \(f = \alpha_{k}\). Clearly, for all of them except for (d), we may restrict
to the full subcomplex \(\mathcal{W}\) of \(\mathcal{BT}\), as \(\mathcal{BT} = \Gamma \mathcal{W}\).

\subsection{} We first observe that \(\mathbf{x} = (x_{1}, \dots, x_{r-1}, 0) \in \mathcal{W}(\mathds{Q})\) belongs to the interior \(\overset{\circ}{\sigma}\) of some maximal simplex \(\sigma\) of \(\mathcal{W}\) (which has dimension 
\(r-1\)) if and only if the \(x_{i}\) (\( 1 \leq i \leq r\)) are pairwise incongruent modulo \(\mathds{Z}\). For such \(\mathbf{x}\), the equality
\begin{equation}
	x_{i} + s = x_{i'} + s' \qquad (1 \leq i,i' \leq r, 0 \leq s,s' < d)
\end{equation}
holds only in the trivial case \( (i,s) = (i',s')\). By Theorem \ref{Theorem.Membership-via-d-characteristic-sequences}, \(\mathbf{x}\) does not belong to \(\mathcal{W}({}_{a}\ell_{k}) = \mathcal{W}(d,k)\). Therefore:

\begin{Proposition}
	Criterion \(\mathrm{(a)}\) of \ref{Subsection.Simpliciality-of-modular-form} is fulfilled for \({}_{a}\ell_{k}\).
\end{Proposition}

\subsection{} \label{Subsection.Understanding-Conditions} In order to treat the other criteria (b), (c), (c strong), (d), (e), we do away with the rationality requirement 
\(\mathbf{x} \in \mathcal{W}(\mathds{Q}) \subset \mathcal{W}(\mathds{R})\) and regard the characteristic sequences \( (v_{k}(\mathbf{x}))_{n \in \mathds{N}}\) and \( (v_{k}^{(d)}(\mathbf{x}))_{1 \leq k \leq d}\) as defined 
on \(\mathcal{W}(\mathds{R})\) (along with the necessary generalizations of definitions, notations, and assertions). This is feasible since
the critical condition \(v_{k}^{(d)}(\mathbf{x}) = v_{k+1}^{(d)}(\mathbf{x})\) (and similarly for \(v_{k}\)) doesn't make use of the rationality of \(\mathbf{x}\).

\subsection{} Let \(\sigma = \{ \mathbf{m}^{(0)}, \dots, \mathbf{m}^{(\ell)}\}\) be a non-maximal simplex in \(\mathcal{W}\) of dimension \(\ell\), where \(1 \leq \ell < r-1\). Write
\[
	\mathbf{m}^{(j)} = (m_{1}^{(j)}, \dots, m_{r}^{(j)}),
\]
where all the \(m_{i}^{(j)} \in \mathds{N}_{0}\), \(m_{1}^{(j)} \geq \dots \geq m_{r-1}^{(j)} \geq m_{r}^{(j)} = 0\). For \(\boldsymbol{\beta} \in \Phi\), its restriction to \(\sigma(\mathds{R})\) satisfies 
\begin{equation} \label{Equation.Restriction-of-beta-to-sigma-IR}
	\boldsymbol{\beta}(\sigma(\mathds{R})) = \{ b \} \text{ or } [b,b+1]
\end{equation}
with some \(b \in \mathds{Z}\). Now consider the functions \(\mathbf{x} \mapsto x_{i} + s\) (resp. \(\mathbf{x} \mapsto x_{i'} + s'\)) on \(\sigma(\mathds{R})\), where \(1 \leq i,i' \leq r\), \(0 \leq s,s' < d \). If
\begin{equation} \label{Equation.xi+s+smaller-than-xiprime+sprime}
	x_{i} + s < x_{i'} + s'
\end{equation} 
for one \(\mathbf{x} \in \overset{\circ}{\sigma}(\mathds{R})\), then by \eqref{Equation.Restriction-of-beta-to-sigma-IR} this relation will persist all over \(\overset{\circ}{\sigma}(\mathds{R})\). Similarly, equality in 
\eqref{Equation.xi+s+smaller-than-xiprime+sprime} holds or doesn't hold for all \(\mathbf{x} \in \overset{\circ}{\sigma}(\mathds{R})\) simultaneously, but \eqref{Equation.xi+s+smaller-than-xiprime+sprime} may degenerate to equality
if \(\mathbf{x}\) approaches a boundary point \(\mathbf{y}\) of \(\sigma(\mathds{R})\). Letting \( (v_{k}^{(d)}(\mathbf{x}))_{1 \leq k \leq rd}\) be the \(d\)-characteristic sequence of \(\mathbf{x} \in \sigma(\mathds{R})\), it follows that
\begin{equation}
	\lim_{\mathbf{x} \to \mathbf{y}} v_{k}^{(d)}(\mathbf{x}) = v_{k}^{(d)}(\mathbf{y}).
\end{equation}
In particular, if \(\mathbf{x} \in \overset{\circ}{\sigma}(\mathds{R})\) belongs to \(\mathcal{W}(d,k) = \{ x \in \mathcal{W}(\mathds{R}) \mid v_{k}^{(d)}(\mathbf{x}) = v_{k+1}^{(d)}(\mathbf{x})\}\), then all the boundary points of 
\(\sigma(\mathds{R})\), and in particular the vertices \(\mathbf{m}^{(j)}\) belong to \(\mathcal{W}(d,k)\).

Conversely, assume that \(\mathbf{m}^{(0)}, \dots, \mathbf{m}^{(\ell)}\) belong to \(\mathcal{W}(d,k)\), and let \(\mathbf{x} \in \overset{\circ}{\sigma}(\mathds{R})\),
\begin{equation}
	\mathbf{x} = \sum_{0 \leq j \leq \ell} t_{j} \mathbf{m}^{(j)}
\end{equation}
with barycentric coordinates \(t_{j}\), \(0 \leq t_{j} \leq 1\), \(\sum t_{j} = 1\). Write \(v_{k}^{(d)}(\mathbf{x}) = x_{i} + s\), \(v_{k+1}^{(d)}(\mathbf{x})= x_{i'} + s'\) with suitable data \( (i,s) \) and \( (i', s') \). Then
\[
	x_{i'} - x_{i} = \sum_{0 \leq j \leq \ell} t_{j}( m_{i'}^{(j)} - m_{i}^{(j)} )
\]
and
\begin{equation} \label{Eq.Difference-of-respective-d-characteristic-sequences}
	v_{k+1}^{(d)}(\mathbf{x}) - v_{k}^{(d)}(\mathbf{x}) = \sum_{j} t_{j}(m_{i'}^{(j)} - m_{i}^{(j)}) + s' - s.
\end{equation}
Letting \(\mathbf{x}\) tend to \(\mathbf{m}^{(j)}\), we find for each \(j\):
\[
	0 = v_{k+1}^{(d)}(\mathbf{m}^{(j)}) - v_{k}^{(d)}(\mathbf{m}^{(j)}) = m_{i'}^{(j)} - m_{i}^{(j)} + s' - s,
\]
that is, \(m_{i'}^{(j)} - m_{i}^{(j)} = s -s'\). Inserting in \eqref{Eq.Difference-of-respective-d-characteristic-sequences} yields
\begin{equation}
	v_{k+1}^{(d)}(\mathbf{x}) - v_{k}^{(d)}(\mathbf{x}) = \Big( \sum t_{j} \Big) (s-s') + s' -s = 0.
\end{equation}
Therefore, \(\mathbf{x} \in \mathcal{W}(d,k)\). As this holds for each \(\mathbf{x} \in \overset{\circ}{\sigma}(\mathds{R})\), we have verified criterion (b) of \ref{Subsection.Simpliciality-of-modular-form} for \(f = {}_{a}\ell_{k}\), i.e.,

\begin{Proposition} \label{Proposition.Characterization-of-Wdk-as-set-of-real-points}
	\(\mathcal{W}(d,k)\) (resp. \(\mathcal{A}(d,k)\), resp. \(\mathcal{BT}(d,k)\)) is the set of real points (compare \ref{Subsection.Understanding-Conditions}) of a full subcomplex of \(\mathcal{W}\) (resp. \(\mathcal{A}\), resp. \(\mathcal{BT}\)), which will be denoted 
	by the same symbol \(\mathcal{W}(d,k)\) (resp. \(\mathcal{A}(d,k)\), resp. \(\mathcal{BT}(d,k)\)).
\end{Proposition}

\subsection{} As we just saw, both the point set and the simplicial complex \(\mathcal{W}(d,k)\) are fully determined by its vertex set
\[
	\mathcal{W}(k,d)(\mathds{Z}) = \{ \mathbf{n} \in \mathcal{W}(\mathds{Z}) \mid v_{k}^{(d)}(\mathbf{n}) = v_{k+1}^{(d)}(\mathbf{n}) \}.
\]
We may therefore restrict our attention to that set and its simplicial structure.

\section{Diagrams of vertices} \label{Section.Diagrams-of-vertices}

In the following, \(\mathbf{m}\), \(\mathbf{n}\), \dots denote vertices of \(\mathcal{A}\) or \(\mathcal{W}\) in the format \(\mathbf{m} = (m_{1}, m_{2}, \dots, m_{r}=0)\) with the coordinates \(m_{i} \in \mathds{Z}\). 
In order to discuss the condition for \(\mathbf{n} \in \mathcal{W}\) to belong to \(\mathcal{W}(k)\) or \(\mathcal{W}(d,k)\), we define the following structures:

\begin{Definition}
	Let \(\mathbf{n} = (n_{1}, n_{2}, \dots, n_{r} = 0)\) be a vertex of \(\mathcal{W}\). 
	\begin{enumerate}[label=(\roman*)]
		\item The \textbf{diagram} \(\diag(\mathbf{n})\) of \(\mathbf{n}\) is the subset of pairs \( (i,v) \) in \( \{1,2,\dots,r\} \times \mathds{Z} \) that satisfy \(n_{i} \leq v\). Elements of \(\diag(\mathbf{n})\) are called
		\textbf{boxes}. The entries of the box \(B = (i,v)\) are called the \textbf{index} \(i = i(B)\) and the \textbf{value} \(v = V(B)\) of \(B\). We impose the following total order on the set of boxes: 
		\begin{equation} \label{Eq.Order-of-boxes}
			B = (i,v) \leq B' = (i', v') \vcentcolon \Longleftrightarrow v < v' \text{ or } (v = v' \text{ and } i \geq i')
		\end{equation}
		and number them accordingly: \(B_{1} = (r,0) < B_{2} < \dots\), or \(B_{1}(\mathbf{n}) < B_{2}(\mathbf{n}) < \dots \) to indicate the dependence on \(\mathbf{n}\). We may visualize \(\diag(\mathbf{n})\) as the
		\( (r \times \mathbf{N}_{0}) \)-matrix with first \(n_{i}\) entries in the \(i\)-th row empty, and the others containing the number of the corresponding box (see Example \ref{Example.Some-vertex-diagrams}(i)). By construction
		\begin{equation} \label{Eq.d-characteristic-sequence-of-box}
			v(B_{k}(\mathbf{n})) = v_{k}(\mathbf{n})
		\end{equation}
		with the characteristic sequence \( v_{k}(\mathbf{n}) \) of \(\mathbf{n}\), and so
		\begin{equation}
			\mathbf{n} \in W(k) \Longleftrightarrow B_{k}(\mathbf{n}) \text{ and } B_{k+1}(\mathbf{n}) \text{ belong to the same column of } \diag(\mathbf{n}).
		\end{equation}
		In case \(\mathbf{n} \in \mathcal{W}(k)\), we define the \textbf{critical index}
		\begin{equation}
			\rho_{k}(\mathbf{n}) \defeq i(B_{k+1}(\mathbf{n})),
		\end{equation}
		so that \(1 \leq \rho_{k}(\mathbf{n}) < r\).
		\item Fix additionally \(d \in \mathds{N}\). The \(d\)-\textbf{diagram} \(\diag^{(d)}(\mathbf{n})\) is the subset of those \( (i,v)\) in \(\{1,2,\dots, r\} \times \mathds{Z}\) that satisfy
		\begin{equation}
			n_{i} \leq v < n_{i} + d.
		\end{equation}
		Elements are called \(d\)-\textbf{boxes}. They are ordered by the same rule \eqref{Eq.Order-of-boxes} and arranged
		\[
			B_{1}^{(d)}(\mathbf{n}) = (r,0) < B_{2}^{(d)}(\mathbf{n}) < \dots < B_{rd}^{(d)}(\mathbf{n}) = (1,n_{1} + d-1).
		\]
		Like \(\diag(\mathbf{n})\), we visualize \(\diag^{(d)}(\mathbf{n})\) as an \( (r \times (n_{1} +d-1))\)-matrix with some empty entries and the others containing the \(d\)-box number 
		(see Example \ref{Example.Some-vertex-diagrams}(ii)). As before
		\begin{equation}
			v(B_{k}^{(d)}(\mathbf{n})) = v_{k}^{(d)}(\mathbf{n})
		\end{equation}
		and
		\begin{equation}
			n \in \mathcal{W}(d,k) \Longleftrightarrow B_{k}^{(d)}(\mathbf{n}) \text{ and } B_{k+1}^{(d)}(\mathbf{n}) \text{ belong to the same column of } \diag^{(d)}(\mathbf{n}).
		\end{equation}
		For \(\mathbf{n} \in \mathcal{W}(d,k)\) we define the \(d\)-\textbf{critical index}
		\begin{equation}
			\rho_{k}^{(d)}(\mathbf{n}) \defeq i(B_{k+1}^{(d)}(\mathbf{n})),
		\end{equation}
		which satisfies \(1 \leq \rho_{k}^{(d)}(\mathbf{n}) < r\).
	\end{enumerate}
\end{Definition}

\begin{Example} \label{Example.Some-vertex-diagrams}
	We let \(r=4\) and \(\mathbf{n} = (4,3,1,0)\).
	\begin{enumerate}
		\item \(d = \infty\) (i.e., no finite \(d\) specified)
		\begin{center}
			\begin{tikzpicture}[scale=0.75]
				\foreach \row in {0,1, ...,4} {
        				\draw (0,\row) -- (7.5,\row);  
   		 		}
   		 		\foreach \column in {0,1,...,6} {
   		 			\draw(\column,0) -- (\column,4); 
   		 		}
	   		 	\draw (2,2) -- (4,4);
				\draw (0,1) -- (3,4);
				\draw (0,2) -- (2,4);
				\draw (0,3) -- (1,4);
				\node (r1) at (-0.5,3.5) {1};
				\node (r2) at (-0.5, 2.5) {2};
				\node (r3) at (-0.5, 1.5) {3};
				\node (r4) at (-0.5, 0.5) {4};
				\node (iv) at (-0.5,-0.5) {\(i/v\)};
				\node (c0) at (0.5,-0.5) {0};
				\node (c1) at (1.5,-0.5) {1};
				\node (c2) at (2.5,-0.5) {2};
				\node (c3) at (3.5, -0.5) {3};
				\node (c4) at (4.5,-0.5) {4};
				\node (c5) at (5.5,-0.5) {5};
				\node (1) at (0.5,0.5) {1};
				\node (2) at (1.5,0.5) {2};
				\node (3) at (1.5,1.5) {3};
				\node (4) at (2.5,0.5) {4};
				\node (5) at (2.5,1.5) {5};
				\node (6) at (3.5,0.5) {6};
				\node (7) at (3.5,1.5) {7};
				\node (8) at (3.5,2.5) {8};
				\node (9) at (4.5,0.5) {9};
				\node (10) at (4.5,1.5) {10};
				\node (11) at (4.5,2.5) {11};
				\node (12) at (4.5,3.5) {12};
				\node (13) at (5.5,0.5) {13};
				\node (14) at (5.5,1.5) {14};
				\node (15) at (5.5,2.5) {15};
				\node (16) at (5.5,3.5) {16};
				\node (d0) at (6.75, -0.5) {\(\cdots\)};
				\node (d1) at (6.75, 0.5) {\(\cdots\)};
				\node (d2) at (6.75, 1.5) {\(\cdots\)};
				\node (d3) at (6.75, 2.5) {\(\cdots\)};
				\node (d4) at (6.75, 3.5) {\(\cdots\)};
				\node (label) at (-2,2) {\(\diag(\mathbf{n})\)};
			\end{tikzpicture}
		\end{center}
		Then, e.g., \(\mathbf{n} \in \mathcal{W}(7)\) and \(\rho_{7}(\mathbf{n}) = 2\).
		\item \(d = 3\)
		\begin{center}
			\begin{tikzpicture}[scale=0.75]
				\foreach \row in {0,1, ...,4} {
        				\draw (0,\row) -- (7,\row);  
   		 		}
   		 		\foreach \column in {0,1,...,7} {
   		 			\draw(\column,0) -- (\column,4); 
   		 		}
   		 		\draw (0,3) -- (1,4);
   		 		\draw (0,2) -- (2,4);
   		 		\draw (0,1) -- (3,4);
   		 		\draw (2,2) -- (4,4);
   		 		\draw (3,0) -- (5,2);
   		 		\draw (4,0) -- (7,3);
   		 		\draw (5,0) -- (7,2);
   		 		\draw (6,0) -- (7,1);
   		 		\node (label) at (-2,2) {\(\diag^{(3)}(\mathbf{n})\)};
   		 		\node (r1) at (-0.5,3.5) {1};
				\node (r2) at (-0.5, 2.5) {2};
				\node (r3) at (-0.5, 1.5) {3};
				\node (r4) at (-0.5, 0.5) {4};
				\node (c0) at (0.5,-0.5) {0};
				\node (c1) at (1.5,-0.5) {1};
				\node (c2) at (2.5,-0.5) {2};
				\node (c3) at (3.5, -0.5) {3};
				\node (c4) at (4.5,-0.5) {4};
				\node (c5) at (5.5,-0.5) {5};
				\node (c6) at (6.5,-0.5) {6};
				\node (1) at (0.5,0.5) {1};
				\node (2) at (1.5,0.5) {2};
				\node (3) at (1.5,1.5) {3};
				\node (4) at (2.5,0.5) {4};
				\node (5) at (2.5,1.5) {5};
				\node (6) at (3.5,1.5) {6};
				\node (7) at (3.5,2.5) {7};
				\node (8) at (4.5,2.5) {8};
				\node (9) at (4.5,3.5) {9};
				\node (10) at (5.5,2.5) {10};
				\node (11) at (5.5,3.5) {11};
				\node (12) at (6.5,3.5) {12};
			\end{tikzpicture}
		\end{center}
		Here, for example, \(\mathbf{n} \in \mathcal{W}(3,6)\) with \(\rho_{6}^{(3)}(\mathbf{n}) = 2\), and \(\mathbf{n} \in \mathcal{W}(3,8) \) with \(\rho_{8}^{(3)}(\mathbf{n}) = 1\).
	\end{enumerate}
\end{Example}

\subsection{} As an exercise we determine which of the \(\mathbf{n}_{j}\) of \eqref{Eq.Fundamental-dominant-weights} belong to \(\mathcal{W}(d,k)\). We have
\begin{equation}
	\diag^{(d)}(\mathbf{n}_{0}) = \diag^{(d)}(\mathbf{0})
\end{equation}\stepcounter{subsubsection}%
\begin{center}
	\begin{tikzpicture}[scale=0.75, every node/.style={scale=0.8}]
		\draw (0,0) -- (9,0);
		\draw (0,1) -- (9,1);
		\draw (0,2) -- (9,2);
		\draw (0,5) -- (9,5);
		\draw (0,6) -- (9,6);
		\draw (0,7) -- (9,7); 
		\draw (0,0) -- (0,7);
		\draw (1.25,0) -- (1.25,7);
		\draw (2.5,0) -- (2.5,7);
		\draw (7,0) -- (7,7);
		\draw (9,0) -- (9,7);
		\node (r1) at (-0.5,6.5) {1};
		\node (r2) at (-0.5,5.5) {2}; 
		\node (rr-1) at (-0.75,1.5) {\(r-1\)};
		\node (rr) at (-0.5,0.5) {\(r\)};
		\node (l) at (-0.5,-0.5) {\(i/v\)};
		\node (c0) at (0.625,-0.5) {0};
		\node (c1) at (1.875,-0.5) {1};
		\node (cd-1) at (8,-0.5) {\(d-1\)};
		\node (r) at (0.625,6.5) {\(r\)};
		\node (r-1) at (0.625, 5.5) {\(r-1\)};
		\node (2) at (0.625,1.5) {2};
		\node (1) at (0.625,0.5) {1};
		\node (2r) at (1.875,6.5) {\(2r\)};
		\node (2r-1) at (1.875,5.5) {\(2r-1\)};
		\node (r+2) at (1.875,1.5) {\(r+2\)};
		\node (r+1) at (1.875,0.5) {\(r+1\)};
		\node (dr) at (8,6.5) {\(dr\)};
		\node (dr-1) at (8,5.5) {\(dr-1\)};
		\node (d-1r+1) at (8,0.5) {\((d-1)r+1\)};
	\end{tikzpicture}
\end{center}
and for \(1 \leq j < r\)
\begin{equation}
	\diag^{(d)}(\mathbf{n}_{j})
\end{equation}\stepcounter{subsubsection}%
\begin{center}
	\begin{tikzpicture}[scale=0.75, every node/.style={scale=0.9}]
		\draw (0,0) -- (15.5,0);
		\draw (0,1) -- (13.5,1);
		\draw (0,4) -- (13.5,4);
		\draw (0,5) -- (15.5,5);
		\draw (1.5,6) -- (15.5,6);
		\draw (1.5,8) -- (15.5,8);
		\draw (0,9) -- (15.5,9);
		\draw (0,0) -- (0,9);
		\draw (1.5,0) -- (1.5,9);
		\draw (3.5,0) -- (3.5,9); 
		\draw (10.5,0) -- (10.5,9);
		\draw (13.5,0) -- (13.5,9);
		\draw (15.5,0) -- (15.5,9);
		\draw (0,5) -- (1.5,9);
		\draw (13.5,0) -- (15.5,5);
		\node (r1) at (-0.5,8.5) {1};
		\node (rj) at (-0.5,5.5) {\(j\)};
		\node (rj+1) at (-0.75, 4.5) {\(j+1\)};
		\node (rr) at (-0.5,0.5) {\(r\)};
		\node (tl) at (-0.5,-0.5) {\(i/v\)};
		\node (c0) at (0.75,-0.5) {0};
		\node (c1) at (2.5,-0.5) {1};
		\node (cd-1) at (12,-0.5) {\(d-1\)};
		\node (d) at (14.5,-0.5) {\(d\)};
		\node (r-j) at (0.75, 4.5) {\(r-j\)};
		\node (d1) at (0.75, 2.5) {\(\cdots\)};
		\node (1) at (0.75,0.5) {1};
		\node (2r-j) at (2.5,8.5) {\(2r-j\)};
		\node (2r-2j) at (2.5,4.5) {\(2r-2j\)};
		\node (r-j+1) at (2.5,0.5) {\(r-j+1\)};
		\node (d2) at (7, 0.5) {\(\cdots\)};
		\node (d3) at (7, 2.5) {\(\cdots\)};
		\node (d4) at (7, 4.5) {\(\cdots\)};
		\node (d5) at (7, 5.5) {\(\cdots\)};
		\node (d6) at (7, 7) {\(\cdots\)};
		\node (d7) at (7,8.5) {\(\cdots\)}; 
		\node (dr-j) at (12,8.5) {\(dr-j\)};
		\node (dr-2j) at (12,4.5) {\(dr-2j\)};
		\node (dr-r-j+1) at (12,0.5) {\((d-1)r-j+1\)};
		\node (dr) at (14.5,8.5) {\(dr\)};
		\node (dr-j+1) at (14.5,5.5) {\(dr-j+1\)};
	\end{tikzpicture}
\end{center}
Hence:
\subsubsection{} \label{Subsubsection.Membership-of-Wdk} \(\mathbf{n}_{j}\) does not belong to \(\mathcal{W}(d,k)\) if and only if \(k+j \equiv 0 \pmod{r}\). That is, if \(\sigma = \{ \mathbf{n}_{i} \mid 0 \leq i < r \}\) 
denotes the standard \( (r-1) \)-simplex in \(\mathcal{W}\) with \(j\)-th face \(\sigma^{(j)} \defeq \sigma \smallsetminus \{ \mathbf{n}_{j} \}\), then \stepcounter{equation}%
\begin{equation} \label{Eq.Characterization-of-simplices-in-Wdk}
	\sigma^{(j)} \text{ is a simplex in } \mathcal{W}(d,k) \Longleftrightarrow k+j \equiv 0 \pmod{r}.
\end{equation}
Note that \ref{Subsubsection.Membership-of-Wdk} and \eqref{Eq.Characterization-of-simplices-in-Wdk} agree with the corresponding results for \(\mathcal{W}(k)\), i.e., for \(\alpha_{k}\) ([V] Proposition 2.11).

\begin{Lemma} \label{Lemma.d-characteristic-sequences-on-WIZ}
	\begin{enumerate}[label=\(\mathrm{(\roman*)}\)]
		\item As functions on the ordered set \(\mathcal{W}(\mathds{Z})\) of vertices of \(\mathcal{W}\), \(v_{k}\) (\(k \in \mathds{N}\)) and \(v_{k}^{(d)} (1 \leq k \leq rd)\) are montonically increasing.
		\item \(v_{k}\) and \(v_{k}^{(d)}\) grow at most by 1 on each simplex \(\sigma \subset \mathcal{W}\).
	\end{enumerate}
\end{Lemma}

\begin{proof}
	\begin{enumerate}[wide, label=(\roman*)]
		\item For \(v_{k}\), it suffices to show that \(v_{k}(\mathbf{n}) \leq v_{k}(\mathbf{n}')\) whenever \(\mathbf{n}\) and \(\mathbf{n}' \defeq \mathbf{n} + \mathbf{e}_{i}\) belong to \(\mathcal{W}\). This is obvious from \eqref{Eq.d-characteristic-sequence-of-box} and a view
		to \(\diag(\mathbf{n})\). The argument for \(v_{k}^{(d)}\) is identical.
		\item In view of (i) and \ref{Subsubsection.Vertices-are-neighbors}, if suffcies to show that for each \(\mathbf{n} \in \mathcal{W}\), \(v_{k}(\mathbf{n} + \mathbf{n}_{r+1}) \leq v_{k}(\mathbf{n}) +1\), and similarly for \(v_{k}^{(d)}\). Again, this is seen on \(\diag(\mathbf{n})\)
		and \(\diag^{(d)}(\mathbf{n})\), respectively, as \(\mathbf{n} \leadsto \mathbf{n} + \mathbf{n}_{r-1}\) shifts the first \(r-1\) rows of diagrams one unit to the right and slightly changes the numbering.
	\end{enumerate}
\end{proof}

Our next task is to describe how the properties \(\mathbf{n} \in \mathcal{W}(k)\) or \(\mathbf{n} \in \mathcal{W}(d,k)\) behave under changing \(\mathbf{n} \in \mathcal{W}\).

\begin{Proposition} \label{Proposition.Sums-of-vertices}
	Let \(\mathbf{n}\) be a vertex of \(\mathcal{W}\) and \(\mathbf{n}' \defeq \mathbf{n} + \mathbf{n}_{i}\) with some \(1 \leq i < r \).
	\begin{enumerate}[label=\(\mathrm{(\roman*)}\)]
		\item If \(\mathbf{n} \in \mathcal{W}(k)\), then \( \mathbf{n}' \in \mathcal{W}(k) \Leftrightarrow \rho_{k}(\mathbf{n}) \neq i\). In this case
		\[
			v_{k}(\mathbf{n}') = v_{k+1}(\mathbf{n}') = \begin{cases} v_{k}(\mathbf{n}),	&\text{if } i < \rho_{k}(\mathbf{n}), \\ v_{k}(\mathbf{n}) + 1,	&\text{if } i > \rho_{k}(\mathbf{n}). \end{cases}
		\]
		\item Suppose that \(\mathbf{n} \in \mathcal{W}(d,k)\). Then \(\mathbf{n}' \in \mathcal{W}(d,k) \Leftrightarrow \rho_{k}^{(d)}(\mathbf{n}) \neq i\), and in this case
		\[
			v_{k}^{(d)}(\mathbf{n}') = v_{k+1}^{(d)}(\mathbf{n}') = \begin{cases} v_{k}^{(d)}(\mathbf{n}),	&\text{if } i < \rho_{k}^{(d)}(\mathbf{n}), \\ v_{k}^{(d)}(\mathbf{n}) + 1,	&\text{if } i > \rho_{k}^{(d)}(\mathbf{n}). \end{cases}
		\]
	\end{enumerate}
\end{Proposition}

\begin{proof}
	\begin{enumerate}[wide, label=(\roman*)]
		\item Let \fbox{\( i < \rho_{k}(\mathbf{n})\)} and \(H \defeq \{ j \mid 1 \leq j \leq i \text{ and } n_{j} < v_{k}(\mathbf{n}) \},\) of cardinality \(h\). Under \(\mathbf{n} \leadsto \mathbf{n}'\) (which on \(\diag\) shifts the first
		\(i\) rows by 1 to the right), precisely \(h\) boxes \(< B_{k}(\mathbf{n})\) become larger than \(B_{k}(\mathbf{n})\), viz, those of shape \( (j, v_{k}(\mathbf{n}) -1) \) with \(j \in H\). Therefore,
		\begin{equation} \label{Eq.Boxes-of-n-and-nprime}
			B_{k-h}(\mathbf{n}') = B_{k}(\mathbf{n}) \quad \text{and} \quad B_{k+1-h}(\mathbf{n}') = B_{k+1}(\mathbf{n}).
		\end{equation}
		Since the column of \(\diag(\mathbf{n}')\) with value \(v_{k}(\mathbf{n})\) contains at least the further \(h\) boxes \( (j,v_{k}(\mathbf{n})) \) with \(j \in H\), also \(B_{k}(\mathbf{n}')\) and \(B_{k+1}(\mathbf{n}')\)
		belong to that column, which shows (i) for \( i < \rho_{k}(\mathbf{n})\).
		
		Let next \fbox{\(i = \rho_{k}(\mathbf{n})\)} and \(H\) and \(h = \#(H)\) be as before. Then \eqref{Eq.Boxes-of-n-and-nprime} still holds, but there are only \(h\) boxes in the column of \(\diag(\mathbf{n}')\) above 
		\(B_{k-h}(\mathbf{n})\). Therefore \(v_{k}(\mathbf{n}') = v_{k}(\mathbf{n}) \neq v_{k+1}(\mathbf{n}') = v_{k}(\mathbf{n}) + 1\).
		
		Now suppose \fbox{\(i > \rho_{k}(\mathbf{n})\)}. Let \(B_{k+1}'(\mathbf{n}) \defeq (\rho_{k}(\mathbf{n}), v_{k}(\mathbf{n}) + 1)\) and \(B_{k}'(\mathbf{n}) \defeq (\rho_{k}(\mathbf{n}) + 1, v_{k}(\mathbf{n}) +1)\) be the boxes 
		\(B_{k+1}(\mathbf{n})\) and \(B_{k}(\mathbf{n})\) shifted by 1 to the right (which are boxes of \(\diag(\mathbf{n}')\)). Then there are precisely \(h \defeq \#(H)\) many boxes of \(\diag(\mathbf{n}')\)
		(namely the \(B = (j, v_{k+1}(\mathbf{n}) )\) with \(j \in H \defeq \{j \mid i < j \leq r\}\)) less than \(B_{k}'(\mathbf{n})\), but not less than \(B_{k+1}(\mathbf{n})\). Hence \(B_{k}'(\mathbf{n}) = B_{k+h}(\mathbf{n}')\)
		and \(B_{k+1}'(\mathbf{n}) = B_{k+1+h}(\mathbf{n}')\), and \(B_{k}(\mathbf{n}') = (\rho_{k}(\mathbf{n}) + 1 +h, v_{k}(\mathbf{n}) +1)\), which gives the result in this case.
		\item The argument is analogous, replacing \(\diag\) with \(\diag^{(d)}\), \(\rho_{k}\) with \(\rho_{k}^{(d)}\), and in case \( i \leq \rho_{k}^{(d)}(\mathbf{n})\), \(H\) with 
		\(H^{(d)} \defeq \{ j \mid 1 \leq j \leq i \text{ and } v_{k}^{(d)}(\mathbf{n}) - d \leq n_{j} < v_{k}^{(d)}(\mathbf{n}) \}\), and in case \(i > \rho_{k}^{(d)}(\mathbf{n})\), \(H\) with 
		\(H^{(d)} \defeq \{ j \mid i < j \leq r \text{ and } v_{k}^{(d)}(\mathbf{n}) + 1-d < n_{j} \leq v_{k}^{(d)}(\mathbf{n}) + 1\}\).
	\end{enumerate}
\end{proof}

\subsection{} In order to study \(\mathcal{W}(k)\) and \(\mathcal{W}(d,k)\) at the boundary \(\bigcup_{1 \leq i \leq r} \mathcal{W}_{i}\) of \(\mathcal{W}\), it is convenient to enlarge the domain of definition of diagrams, 
characteristic sequences and related notions to \(\mathbf{n} \in \mathcal{A}(\mathds{Z})\). Namely, allow the Definition 4.1 also for \(\mathbf{n} = (n_{1}, n_{2}, \dots, n_{r-1}, 0) \in \mathcal{A}(\mathds{Z})\), where 
\(n_{1}, \dots, n_{r-1}\) are arbitrary. This yields \(\diag(\mathbf{n})\) and \(\diag^{(d)}(\mathbf{n})\) as well as \(v_{k}(\mathbf{n}) \defeq v_{k}(B_{k}(\mathbf{n}))\) and \(v_{k}^{(d)}(\mathbf{n}) \defeq v(B_{k}^{(d)}(\mathbf{n}))\) 
for such \(\mathbf{n}\). (We have however no use for \(\rho_{k}(\mathbf{n})\) or \(\rho_{k}^{(d)}(\mathbf{n})\) if \(\mathbf{n} \notin \mathcal{W}\).). Then it is easy to see
\subsubsection{} \label{Subsubsection.Validity-of-lemma-for-AIZ} Both parts of Lemma \ref{Lemma.d-characteristic-sequences-on-WIZ} are still valid for \(v_{k}\), \(v_{k}^{(d)}\) as functions on \(\mathcal{A}(\mathds{Z})\). \stepcounter{equation}%

Now \(B_{1}(\mathbf{n})\) will in general differ from \( (r,0) \), since \(v_{1}(\mathbf{n})\) may be less than 0. Therefore we set
\begin{equation}
	\tilde{v}_{k}(\mathbf{n}) \defeq v_{k}(\mathbf{n}) - v_{1}(\mathbf{n}), \qquad \tilde{v}_{k}^{(d)}(\mathbf{n}) \defeq v_{k}^{(d)}(\mathbf{n}) - v_{1}^{(d)}(\mathbf{n}),
\end{equation}
where \(v_{1}(\mathbf{n}) = v_{1}^{(d)}(\mathbf{n}) = \min_{1 \leq i \leq n} \{ n_{i} \}\). As the Weyl group \(W\) permutes the indices \(i\), i.e., the rows of diagrams, we see:
\subsubsection{} The normalized characteristic sequences \( \tilde{v}_{k}(\mathbf{n}) )_{k \in \mathds{N}}\) and \( (\tilde{v}_{k}^{(d)}(\mathbf{n}))_{1 \leq k \leq rd} \) are invariant under applying \(w\in W\) to \(\mathbf{n} \in \mathds{Z}\). \stepcounter{equation}%

Since the simplicial complexes \(\mathcal{A}(k)\) and \(\mathcal{A}(d,k)\) as well as their vertex sets are \(W\)-stable and intersect as \(\mathcal{W}(k)\) resp. \(\mathcal{W}(d,k)\) with \(\mathcal{W}\), we also find:
\begin{equation}
	\begin{split}
		\mathbf{n} \in \mathcal{A}(k)	&\Longleftrightarrow \tilde{v}_{k}(\mathbf{n}) = \tilde{v}_{k+1}(\mathbf{n})	\Longleftrightarrow v_{k}(\mathbf{n}) = v_{k+1}(\mathbf{n}) \\
		\mathbf{n} \in \mathcal{A}(d,k)	&\Longleftrightarrow \tilde{v}_{k}^{(d)}(\mathbf{n}) = \tilde{v}_{k+1}^{(d)}(\mathfrak{n}) 	\Longleftrightarrow v_{k}^{(d)}(\mathbf{n}) = v_{k+1}^{(d)}(\mathbf{n}).
	\end{split}
\end{equation}
Now we can compare the behavior of \(v_{k}\) and \(v_{k}^{(d)}\) on the vertices \(\mathbf{n} \in \mathcal{W}\) and \(\mathbf{n}' \defeq \mathbf{n} - \mathbf{n}_{r-1} \in \mathcal{A}\).

\begin{Proposition} \label{Proposition.d-characteristic-sequences-of-certain-vertices}
	With notation as above, the following hold for \(k \geq 2\).
	\begin{enumerate}[label=\(\mathrm{(\roman*)}\)]
		\item \(v_{k}(\mathbf{n}) = v_{k-1}(\mathbf{n}') +1\);
		\item \[
			v_{k}^{(d)}(\mathbf{n})	= \begin{cases} v_{k-1}^{(d)}(\mathbf{n}') + 1,	&\text{if } v_{k}^{(d)}(\mathbf{n}) < d, \\ v_{k}^{(d)}(\mathbf{n}') + 1,	&\text{if } v_{k}^{(d)}(\mathbf{n}) \geq d. \end{cases}
		\]
		\item Suppose that \(v_{k}^{(d)}(\mathbf{n}) \geq d\) and for some \(i\) with \(1 \leq i < r\), \(n_{i} > 0\) and \(n_{i+1} = 0\) holds. Then
		\[
			v_{k}^{(d)}(\mathbf{n}) = v_{k}^{(d)}(\mathbf{n} - \mathbf{n}_{i}) + 1 = v_{k-1}^{(d)}(\mathbf{n} - \mathbf{n}_{i}) + 1.
		\]
	\end{enumerate}
\end{Proposition}

\begin{proof}
	First note that \( \diag(\mathbf{n}')\) (resp. \(\diag^{(d)}(\mathbf{n}')\)) is obtained from \(\diag(\mathbf{n})\) (resp. \(\diag^{(d)}(\mathbf{n})\)) by shifting the first \(r-1\) rows to the left and changing the numbering accordingly.
	\begin{enumerate}[wide, label=(\roman*)]
		\item Let \(B_{k}(\mathbf{n}) = (i,v)\). Going through the cases, the assertion is obvious for \(v=0\) or \(i = r\); hence suppose \(v = v_{k}(\mathbf{n}) > 0\), \(i < r\) and put \(B_{k}'(\mathbf{n}) \defeq (i,v-1)\). Comparing the
		diagrams, one first finds that for all \(\ell\) with \(1 \leq \ell < k-1\) the inequality \(B_{\ell}(\mathbf{n}') < B_{k}'(\mathbf{n})\) holds, while for \(\ell > k-1\), \(B_{\ell}(\mathbf{n}') > B_{k}'(\mathbf{n})\), which thus
		agrees with \(B_{k-1}(\mathbf{n}')\).
		\item Let \(B_{k}^{(d)}(\mathbf{n}) = (i,v)\). Again the cases \(v = 0\) or \(i=r\) are seen directly. If \(i < r \) and \(0 < v < d\), the same argument as in (i) yields that \(B_{k}^{(d)}(\mathbf{n})' \defeq (i, v-1)\) equals
		\(B_{k-1}^{(d)}(\mathbf{n}')\), while for \(v \geq d\), \(B_{k}^{(d)}(\mathbf{n})' = B_{k}^{(d)}(\mathbf{n}')\), which in both cases leads to the stated formulas.
		\item Again, this is seen by a look to \(\diag^{(d)}(\mathbf{n})\) and \(\diag^{(d)}(\mathbf{n} - \mathbf{n}_{i})\).
	\end{enumerate}
\end{proof}

\begin{Corollary} \label{Corollary.Membership-of-AIZ}
	As before, let \(\mathbf{n} \in \mathcal{W}(\mathds{Z})\) be given and \(\mathbf{n}' \defeq \mathbf{n} - \mathbf{n}_{r-1} \in \mathcal{A}(\mathds{Z})\).
	\begin{enumerate}[label=\(\mathrm{(\roman*)}\)]
		\item For \(k \geq 2\) we have \(\mathbf{n} \in \mathcal{W}(k) \Leftrightarrow \mathbf{n}' \in \mathcal{A}(k-1)\).
		\item Let \(2 \leq k < rd\). If \(v_{k+1}^{(d)}(\mathbf{n}) < d\) then
		\[
			\mathbf{n} \in \mathcal{W}(d,k) \Longleftrightarrow \mathbf{n}' \in \mathcal{A}(d, k-1).
		\]
		If \(v_{k+1}^{(d)}(\mathbf{n}) \geq d\) then
		\[
			\mathbf{n} \in \mathcal{W}(d,k) \Longleftrightarrow \mathbf{n}' \in \mathcal{A}(d,k).
		\]
	\end{enumerate}
\end{Corollary}

\begin{proof}
	This follows from the proposition and the properties of \(v_{k}\) and \(v_{k}^{(d)}\).
\end{proof}

\begin{Remark}
	(i) was first stated and proved in [V] Proposition 2.8, with a more complicated proof. It allows an induction procedure both to calculate \(\mathcal{W}(k)\) and to show some of its properties. Unfortunately this approach doesn't work for
	\(\mathcal{W}(d,k)\).
\end{Remark}

\section{Connectedness of \(\mathcal{W}(d,k)\), \(\mathcal{A}(d,k)\), and \(\mathcal{BT}(d,k)\)} \label{Section.Connectedness}

We use the preceding to show that \(\mathcal{W}(d,k)\) and the related complexes are connected. As the case \(d=1\) is well-known ( \(\mathcal{W}(1,k) = \mathcal{W}(g_{k}) = \mathcal{W}_{r-k}\) is the vanishing set of the 
coefficient function \(g_{k} = {}_{T}\ell_{k}\) which has been dealt with in [I] and [II]), we assume throughout that \(d \geq 2\). We also know that \(\mathcal{W}(d,1) = \mathcal{W}(1)\) (see Remark \ref{Remark.Coincidence-of-Wk-and-Wdk}); so it suffices to study 
\(\mathcal{W}(d,k)\) with \fbox{\(2 \leq k < rd\)}, which we from now on assume.

\subsection{} The first step is to present a convenient decomposition (see \ref{Subsection.Decomposition-of-WdkIZ}) of \( \mathcal{W}(d,k)(\mathds{Z})\), which will turn out useful also for other purposes. The frequently occurring vector \(\mathbf{n}_{r-1}\) will be abbreviated:
\begin{equation}
	\mathbf{y} \defeq \mathbf{n}_{r-1} = (1,1,\dots,1,0).
\end{equation}
Given \(k\) as above, we split
\begin{equation}
	\mathcal{W}(d,k)(\mathds{Z}) = {}_{1} \mathcal{W}(d,k) \cupdot {}_{2}\mathcal{W}(d,k),
\end{equation}
where
\begin{align*}
	{}_{1}\mathcal{W}(d,k) 	&\defeq \{ \mathbf{n} \in \mathcal{W}(d,k)(\mathds{Z}) \mid v_{k}^{(d)}(\mathbf{n}) < d\}, \\ 
	{}_{2}\mathcal{W}(d,k) 	&\defeq \{ \mathbf{n} \in \mathcal{W}(d,k)(\mathds{Z}) \mid v_{k}^{(d)}(\mathbf{n}) \geq d \}.
\end{align*}
Then by \ref{Proposition.d-characteristic-sequences-of-certain-vertices}(ii) and \ref{Corollary.Membership-of-AIZ}(ii),
\begin{align}
	{}_{1}\mathcal{W}(d,k)	&= (\mathcal{A}(d,k-1)(\mathds{Z}) + \mathbf{y}) \cap \{ \mathbf{n} \in \mathcal{W}(\mathds{Z}) \mid v_{k+1}^{(d)}(\mathbf{n}) < d\}
	\intertext{and}
	{}_{2}\mathcal{W}(d,k)	&= ( \mathcal{A}(d,k)(\mathds{Z}) + \mathbf{y}) \cap \{ \mathbf{n} \in \mathcal{W}(\mathds{Z}) \mid v_{k+1}^{(d)}(\mathbf{n}) \geq d\}.
\end{align}
\subsection{} \label{Subsection.Maximal-vertex-in-Wdk} Let \(\mathbf{n} = (n_{1}, \dots, n_{r-1}, 0) \in {}_{2}\mathcal{W}(d,k)\) and \(i\) be maximal such that \(n_{i} > 0\). By \ref{Proposition.d-characteristic-sequences-of-certain-vertices}(iii),
\[
	v_{\ell}^{(d)}(\mathbf{n} - \mathbf{n}_{i}) = v_{\ell}^{(d)}(\mathbf{n}) - 1 \quad \text{for} \quad \ell = k+1, k,
\]
and in particular \(\mathbf{n}^{(1)} \defeq \mathbf{n} - \mathbf{n}_{i} \in \mathcal{W}(d,k)\). Putting
\begin{equation}
	\mathbf{n}^{(0)} = \mathbf{n}, \quad \mathbf{n}^{(j)} \defeq (\mathbf{n}^{(j-1)})^{(1)} \quad \text{for} \quad j>0,
\end{equation}
we find a sequence in \(\mathcal{W}(d,k)\) with
\begin{equation}
	v_{k}^{(d)}(\mathbf{n}^{(j)}) = v_{k}^{(d)}(\mathbf{n}) - j,
\end{equation}
as long as \(\mathbf{n}^{(j-1)} \in {}_{2}\mathcal{W}(d,k)\). Note that \(\mathbf{n}^{(j-1)}\) and \(\mathbf{n}^{(j)}\) are neighbors; so the sequence is an edge path in the complex \(\mathcal{W}(d,k)\).

\subsection{} Let \(j\) be maximal such that \(n^{(j-1)} \in {}_{2}W(d,k)\). Then by construction
\begin{equation}
	v_{k}^{(d)}(\mathbf{n}^{(j)}) = v_{k+1}^{(d)}(\mathbf{n}^{(j)}) = d-1.
\end{equation}
Moreover, Proposition \ref{Proposition.d-characteristic-sequences-of-certain-vertices}(iii) gives that \(v_{k-1}^{(d)}(\mathbf{n}^{(j)})\) has the same value \(d-1\). That is,
\begin{equation}
	\mathbf{n}^{(j)} \in {}_{3}\mathcal{W}(d,k) \defeq \{ \mathbf{n} \in \mathcal{W}(d,k-1)(\mathds{Z}) \cap \mathcal{W}(d,k)(\mathds{Z}) \mid v_{k}^{(d)}(\mathbf{n}) = d-1 \}.
\end{equation}
We note that for \(\mathbf{n} \in {}_{3}\mathcal{W}(d,k)\)
\begin{equation} \label{Eq.Rho-for-elements-of-3Wdk}
	\rho_{k}^{(d)}(\mathbf{n}) + 1 = \rho_{k-1}^{(d)}(\mathbf{n}) < r
\end{equation}
holds.

\begin{Lemma}
	Assume that \(\mathbf{n} \in \mathcal{W}(d,k)\) satisfies \(v_{k}^{(d)}(\mathbf{n}) \geq d-1\) and \(\rho_{k}^{(d)}(\mathbf{n}) < r-1\). Then \(\mathbf{n}' \defeq \mathbf{n} + \mathbf{y}\) has the following properties:
	\begin{align}
		\mathbf{n}' 	&\in \mathcal{W}(d,k); \\
		v_{k}^{(d)}(\mathbf{n}')	&= v_{k}^{(d)}(\mathbf{n}) + 1; \\
		\rho_{k}^{(d)}(\mathbf{n}')	&= \rho_{k}^{(d)}(\mathbf{n}). \label{Eq.Equality-of-rho-under-certain-conditions}
	\end{align}
\end{Lemma}

\begin{proof}
	The first two are special cases of Proposition \ref{Proposition.Sums-of-vertices}(ii). A view to the \(d\)-diagrams reveals that if \(B_{k}^{(d)}(\mathbf{n}) = (i,v)\) (so \(B_{k+1}^{(d)}(\mathbf{n}) = (i-1,v)\)), then 
	\(B_{k}^{(d)}(\mathbf{n}') = (i,v+1)\) and \(B_{k+1}^{(d)}(\mathbf{n}') = (i-1,v+1)\). This shows \eqref{Eq.Equality-of-rho-under-certain-conditions}.
\end{proof}

\subsection{} \label{Subsection.Consequences-of-lemma-to-3Wdk} Applying the lemma to \(\mathbf{n} \in {}_{3}\mathcal{W}(d,k)\), we find that for each \(j \in \mathds{N}_{0}\) the assertions 
\begin{align}
	\mathbf{n} + j \mathbf{y}	&\in \mathcal{W}(d,k) \\
	v_{k}^{(d)}(\mathbf{n} + j \mathbf{y})	&= v_{k}^{(d)}(\mathbf{n}) +j \\
	\rho_{k}^{(d)}(\mathbf{n} + j \mathbf{y})	&= \rho_{k}^{(d)}(\mathbf{n})
\end{align}
hold. 
\subsection{} Define finally
\begin{align*}
	{}_{4}\mathcal{W}(d,k) 	&\defeq \{ \mathbf{n} \in \mathcal{W}(d,k)(\mathds{Z}) \mid n_{r-1} = 0 \text{ and } v_{k}^{(d)}(\mathbf{n}) \geq d \} \\
												&\hphantom{\vcentcolon}= \{ \mathbf{n} \in \mathcal{W}(d,k)(\mathds{Z}) \cap \mathcal{W}_{r-1} \mid v_{k}^{(d)}(\mathbf{n}) \geq d \}.
\end{align*}
Then
\begin{equation}
	{}_{3}\mathcal{W}(d,k) \text{ and } {}_{4}\mathcal{W}(d,k) \text{ are disjoint}.
\end{equation}
Like those of \({}_{3}\mathcal{W}(d,k)\), the elements of \( {}_{4}\mathcal{W}(d,k)\) satisfy the condition \eqref{Eq.Rho-for-elements-of-3Wdk}. Therefore, also the three properties of \ref{Subsection.Consequences-of-lemma-to-3Wdk} hold for \(\mathbf{n} \in {}_{4}\mathcal{W}(d,k)\).
\subsection{} Given \(\mathbf{n} \in {}_{2}\mathcal{W}(d,k)\), consider the sequence defined in \ref{Subsection.Maximal-vertex-in-Wdk}: \(\mathbf{n} = \mathbf{n}^{(0)}, \mathbf{n}^{(1)}, \dots, \mathbf{n}^{(j)} \in {}_{3}\mathcal{W}(d,k)\). 
If there exists \(\ell < j\) such that \(\mathbf{n}^{(\ell)} \in {}_{4}\mathcal{W}(d,k)\), then the remaining terms \(\mathbf{n}^{(\ell+1)}, \dots, \mathbf{n}^{(j-1)}\) stay in \({}_{4}\mathcal{W}(d,k)\). Let \(\ell(\mathbf{n})\) be minimal with 
that property, and \(\ell(\mathbf{n}) \defeq j = j(\mathbf{n})\) if no such \(\ell\) exists. Now consider the map
\begin{align}
	\pi \colon {}_{2}\mathcal{W}(d,k) 	&\longrightarrow ({}_{3}\mathcal{W}(d,k) \cupdot {}_{4}\mathcal{W}(d,k) ) \times \mathds{N}_{0}, \\
										\mathbf{n}		&\longmapsto (\mathbf{n}^{(\ell(\mathbf{n}))}, \ell(\mathbf{n}) ) \nonumber
\end{align}
which is well-defined and injective. The complement of \(\pi({}_{2}\mathcal{W}(d,k))\) in the right hand side is \({}_{3}\mathcal{W}(d,k) \times \{0\}\). Together, this means that we have the following basic decomposition of 
\(\mathcal{W}(d,k)(\mathds{Z})\):

\subsection{} \label{Subsection.Decomposition-of-WdkIZ} \(\mathcal{W}(d,k)(\mathds{Z}) = {}_{1}\mathcal{W}(d,k) \cup {}_{5}\mathcal{W}(d,k)\), where 
\[
	{}_{5}\mathcal{W}(d,k) \defeq ( {}_{3}\mathcal{W}(d,k) \cupdot {}_{4}\mathcal{W}(d,k)) + \mathds{N}_{0}\mathbf{y},
\] 
and the intersection of the two sets \({}_{1}\mathcal{W}(d,k)\) and \({}_{5}\mathcal{W}(d,k)\) is \({}_{3}\mathcal{W}(d,k)\). The intersection pattern of the sets \({}_{i}\mathcal{W} \defeq {}_{i}\mathcal{W}(d,k)\) as subsets of 
\(\mathcal{W}(d,k)(\mathds{Z})\) may be visualized as follows.

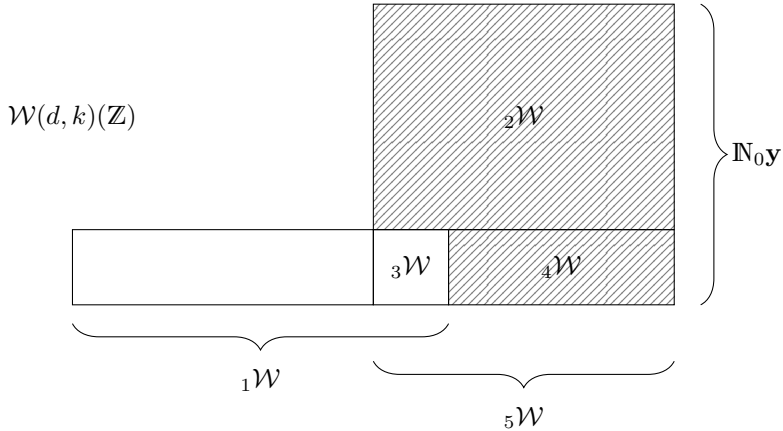
\begin{Picture} ~ \label{Picture.Weyl-chamber}
	\begin{center}
		\begin{tikzpicture}
			\draw[pattern=north east lines, pattern color=gray] (0,1) rectangle (4,4);
			\draw (0,0) rectangle (1,1);
			\draw[pattern=north east lines, pattern color=gray] (1,0) rectangle (4,1);
			\draw (-4,0) rectangle (0,1);
			\node (3W) at (0.5,0.5) {\({}_{3}\mathcal{W}\)};
			\node (4W) at (2.5,0.5) {\({}_{4}\mathcal{W}\)};
			\node (2W) at (2,2.5) {\({}_{2}\mathcal{W}\)};
			\node (wdk) at (-4,2.5) {\(\mathcal{W}(d,k)(\mathds{Z})\)};
			\draw[decorate,decoration={brace,amplitude=10pt},yshift=0pt] (1,-0.25) -- (-4,-0.25) node [black,midway, yshift=-0.75cm] {\({}_{1}\mathcal{W}\)};
			\draw[decorate,decoration={brace,amplitude=10pt},yshift=0pt] (4,-0.75) -- (0,-0.75) node [black,midway, yshift=-0.75cm] {\({}_{5}\mathcal{W}\)};
			\draw[decorate,decoration={brace,amplitude=10pt},xshift=10pt,yshift=0pt] (4,4) -- (4,0) node [black,midway, xshift=0.75cm] {\(\mathds{N}_{0}\mathbf{y}\)}; 
		\end{tikzpicture}
	\end{center}
	Note that \({}_{2}\mathcal{W}\) may be empty, in which case \(\mathcal{W}(d,k)(\mathds{Z}) = {}_{1}\mathcal{W}\) and also \({}_{3}\mathcal{W}(d,k) = \varnothing\). This holds if and only if \(k \leq d\), as then each 
	\(\mathbf{n} \in \mathcal{W}(d,k)(\mathds{Z})\) satisfies \(v_{k}^{(d)}(\mathbf{n}) < d\).
\end{Picture}

\subsection{} \label{Subsection.Connectedness-via-edgepath} We are going to show that \(\mathcal{W}(d,k)\) is connected by constructing for each vertex \(\mathbf{n} \in \mathcal{W}(d,k)\) an edge path 
in \(\mathcal{W}(d,k)\) to one of the \(\mathbf{n}_{i}\) (\(0 \leq i < r\)), which suffices by \ref{Subsubsection.Membership-of-Wdk} and \eqref{Eq.Characterization-of-simplices-in-Wdk}. Note that for \(\mathbf{n} \in \mathcal{W}\) the distance \(d(\mathbf{n}, \mathbf{0})\) to \(\mathbf{n}_{0} = \mathbf{0}\) 
is given by 
\begin{equation}
	d(\mathbf{n}, \mathbf{0}) = n_{1}.
\end{equation}
As for \(\mathbf{n} \in {}_{2}\mathcal{W}(d,k)\) the sequence \(\mathbf{n} = \mathbf{n}^{(0)}, \dots, \mathbf{n}^{(j)}\) ends in \({}_{3}\mathcal{W}(d,k) \subset {}_{1}\mathcal{W}(d,k)\), we may suppose that 
\(\mathbf{n} \in {}_{1}\mathcal{W}(d,k)\). Roughly speaking, we will substract a suitably chosen \(\mathbf{n}_{i}\) from \(\mathbf{n}\) such that \(\mathbf{n}' \defeq \mathbf{n} - \mathbf{n}_{i}\) still belongs to 
\({}_{1}\mathcal{W}(d,k)\) and \(d(\mathbf{n}', \mathbf{0}) < d(\mathbf{n}, \mathbf{0})\).
This works except for a handful of very special cases, which can be bypassed by ad hoc means.
\subsection{} \label{Subsection.Special-vertex-in-1Wdk} So let's start with \(\mathbf{n} \in {}_{1}\mathcal{W}(d,k)\), with
\begin{equation}
	\rho \defeq \rho_{k}^{(d)}(\mathbf{n}) \in \{1,2,\dots,r-1\}, \qquad v \defeq v_{k}^{(d)}(\mathbf{n}) < d.
\end{equation}
Suppose there exists some \(i\) with \(1 \leq i < \rho\) and \(n_{i} > n_{i+1}\). Given \(i\), let
\begin{align} \label{Eq.Notation-for-special-vertex-in-1Wdk}
	H	&\defeq \{ j \mid 1 \leq j \leq i \text{ and } (j,v) \text{ is a box of } \diag^{(d)}(\mathbf{n}) \} \\
		&\hphantom{\vcentcolon}= \{ j \mid 1 \leq j \leq i \text{ and } n_{j} \leq v \} \quad \text{and} \nonumber\\
	h	&\defeq \#(H). \nonumber
\end{align}
Put \(\mathbf{n}' \defeq \mathbf{n} - \mathbf{n}_{i}\), which lies in \(\mathcal{W}\) with \(d(\mathbf{n}', \mathbf{0}) = d(\mathbf{n}, \mathbf{0}) - 1\). Under \(\mathbf{n} \leadsto \mathbf{n}'\), the \(d\)-boxes \((j,v)\) of \(\mathbf{n}\) 
become the \(d\)-boxes \( (j, v-1) \) of \(\mathbf{n}'\), and are smaller than \(B_{k}^{(d)}(\mathbf{n}) = (\rho+1, v)\) and \(B_{k+1}^{(d)}(\mathbf{n}) = (\rho, v)\), these latter regarded as \(d\)-boxes of \(\mathbf{n}'\). The
other order relations remain unchanged; hence
\begin{equation}
	(\rho+1, v) = B_{k+h}^{(d)}(\mathbf{n}'), \qquad (\rho, v) = B_{k+1+h}^{(d)}(\mathbf{n}').
\end{equation} 
If now \(n_{i} > v\), then \(h=0\) and \(\mathbf{n}' \in {}_{1}\mathcal{W}(d,k)\) with smaller distance to \(\mathbf{0}\), but the quantities \(\rho\) and \(v\) unchanged. Applying this suitably often with \(i = \rho-1, \dots, 2,1\), we arrive at some
\(\mathbf{n} \in {}_{1}\mathcal{W}(d,k)\) with
\begin{equation} \label{Eq.Vertex-in-1Wdk-with-n1-leq-v}
	n_{1} \leq v,
\end{equation}
which we from now an assume. In particular, all the \( (j,v) \) with \(1 \leq j < \rho\) are boxes of \(\mathbf{n}\) for such \(\mathbf{n}\). This implies: If \(\mathbf{n}' \defeq \mathbf{n} - \mathbf{n}_{i}\) for some \(i < \rho\) with
\(n_{i} > n_{i+1}\), then the \(h\) of \eqref{Eq.Notation-for-special-vertex-in-1Wdk} equals \(i\), and so
\begin{equation}
	B_{k+i}^{(d)}(\mathbf{n}') = B_{k}^{(d)}(\mathbf{n}), \qquad B_{k+1+i}^{(d)}(\mathbf{n}') = B_{k+1}^{(d)}(\mathbf{n}).
\end{equation}
Suppose that \fbox{\(i < r- \rho\)}. Then \(\rho +1+i \leq r\), so
\[
	B_{k}^{(d)}(\mathbf{n}') = (\rho+1+i, v), \qquad B_{k+1}^{(d)}(\mathbf{n}') = (\rho + i, v),
\]
and \(\mathbf{n}'\) still belongs to \({}_{1}\mathcal{W}(d,k)\). 

If \fbox{\(i=r-\rho\)}, \(B_{k+1}^{(d)}(\mathbf{n}') = (r,v)\) but \(B_{k}^{(d)}(\mathbf{n}')\) is of shape \( (*, v-1) \), hence \(\mathbf{n}' \notin \mathcal{W}(d,k)\).

Suppose that \fbox{\(r- \rho < i < \rho\)} (which of course is possible only for \(r < 2\rho\)). Then \(\rho + i > r\), so both \(B_{k}^{(d)}(\mathbf{n}')\) and \(B_{k+1}^{(d)}(\mathbf{n}')\) have value \(v-1\), and 
\(\mathbf{n}' \in {}_{1}\mathcal{W}(d,k)\).

\subsection{} \label{Subsection.Second-special-vertex} Let \(\mathbf{n} \in {}_{1}\mathcal{W}(d,k)\) be as in \ref{Subsection.Special-vertex-in-1Wdk} and subject to \eqref{Eq.Vertex-in-1Wdk-with-n1-leq-v}. Suppose 
there exists some \(i\) with \(\rho < i < r\) and \(n_{i} > 0\), and that \(i\) is maximal with this property. For \(\mathbf{n}' \defeq \mathbf{n} - \mathbf{n}_{i}\) and \(h \defeq r-i\),
\begin{equation}
	B_{k-h}^{(d)}(\mathbf{n}') = (\rho+1, v-1), \qquad B_{j+1-h}^{(d)}(\mathbf{n}') = (\rho, v-1),
\end{equation}
since precisely the \(h\) boxes \( (j,v) \) with \( i < j \leq r\) are less than \( (\rho+1, v) = B_{k+1}^{(d)}(\mathbf{n})\) as boxes of \(\mathbf{n}\), but not less than \( (\rho, v-1) \) regarded as \(d\)-boxes of \(\mathbf{n}'\) (and the other
order relations unchanged).

If now \fbox{\(i > r-\rho\)}, i.e., \(h < \rho\), then \(B_{k}^{(d)}(\mathbf{n}') = (\rho+1-h, v-1)\) and \(B_{k+1}^{(d)}(\mathbf{n}') = (\rho-h, v-1)\), and \(\mathbf{n}' \in {}_{1}\mathcal{W}(d,k)\). If \fbox{\(i = r-\rho\)}, i.e., \(h = \rho\),
then \(\mathbf{n}' \notin \mathcal{W}(d,k)\), while for \fbox{\(i<r-\rho\)}, i.e., \(h > \rho\), \(B_{k}^{(d)}(\mathbf{n}')\) and \(B_{k+1}^{(d)}(\mathbf{n}')\) are of shape \( (*, v)\) and so \(\mathbf{n}' \in {}_{1}\mathcal{W}(d,k)\).

\subsection{} \label{Subsection.Third-special-vertex} Suppose \(\mathbf{n} \in {}_{1}\mathcal{W}(d,k)\) is such that a lowering of \(d(\mathbf{n}, \mathbf{0}) = n_{1} \leq v\) inside \({}_{1}\mathcal{W}(d,k)\) through 
\(\mathbf{n} \leadsto \mathbf{n} - \mathbf{n}_{i}\) is not possible by means of the devices in \ref{Subsection.Special-vertex-in-1Wdk} and \ref{Subsection.Second-special-vertex}. Then with \(B_{k+1}^{(d)}(\mathbf{n}) = (\rho, v)\), the following hold:
\subsubsection{} \begin{itemize}
	\item \(n_{1} \leq v\);
	\item for \(i < \min(\rho, r-\rho)\), \(n_{i} = n_{i+1}\);
	\item for \(r-\rho < i < \rho\), \(n_{i} = n_{i+1}\);
	\item for \(i > \max(\rho, r-\rho)\), \(n_{i} = n_{i+1}\);
	\item for \(\rho < i < r-\rho\), \(n_{i} = n_{i+1}\).
\end{itemize}
(Here the first three conditions come from \ref{Subsection.Special-vertex-in-1Wdk}, the last two from \ref{Subsection.Second-special-vertex}.) That is, if \(n_{i} > n_{i+1}\) then \(i \in \{ \rho, r-\rho\}\).
\subsection{} \label{Subsection.Fourth-special-vertex} Consider the case \fbox{\(n_{\rho} > n_{\rho+1}\)}, and assume \fbox{\( \rho < r- \rho\)}. Let \(\mathbf{n}' \defeq \mathbf{n} - \mathbf{n}_{\rho}\), \(h \defeq r-\rho\). 
With arguments as before we find:
\begin{equation}
	B_{k+1-h}^{(d)}(\mathbf{n}') = (\rho, v-1),
\end{equation}
so \(B_{k+\rho-h}^{(d)}(\mathbf{n}') = (1,v-1)\) and \(B_{k+\rho+1-h}^{(d)}(\mathbf{n}') = (r,v)\). As by our assumption \(\rho+1-h \leq 0\), we see that both \(B_{k}^{(d)}(\mathbf{n}')\) and \(B_{k+1}^{(d)}(\mathbf{n}')\) lie on
the column of \(\diag^{(d)}(\mathbf{n}')\) with value \(v\). That is, \(\mathbf{n}' \in {}_{1}\mathcal{W}(d,k)\). Similarly, if \fbox{\(r-\rho < \rho\)}, \(\mathbf{n}' \defeq \mathbf{n} - \mathbf{n}_{\rho}\) is seen to lie in 
\({}_{1}\mathcal{W}(d,k)\), where both \(B_{k}^{(d)}(\mathbf{n}')\) and \(B_{k+1}^{(d)}(\mathbf{n}')\) have value \(v-1\).

\subsection{} Suppose that our \(\mathbf{n}\) of \ref{Subsection.Third-special-vertex} is such that \(d(\mathbf{n}, \mathbf{0})\) cannot be lowered in \({}_{1}\mathcal{W}(d,k)\) through applying \ref{Subsection.Special-vertex-in-1Wdk}, \ref{Subsection.Second-special-vertex}, or \ref{Subsection.Fourth-special-vertex}. Then also \(n_{\rho} = n_{\rho+1}\), that is
\begin{equation}
	\mathbf{n} = t \mathbf{n}_{r-\rho} \text{ with some } t \in \mathds{N}_{0}, \qquad t \leq v < d.
\end{equation}
If \(t \leq 1\), we are ready (see \ref{Subsection.Connectedness-via-edgepath}); so let \(t > 1\). The \(d\)-diagram \(\diag^{(d)}(\mathbf{n})\) looks:
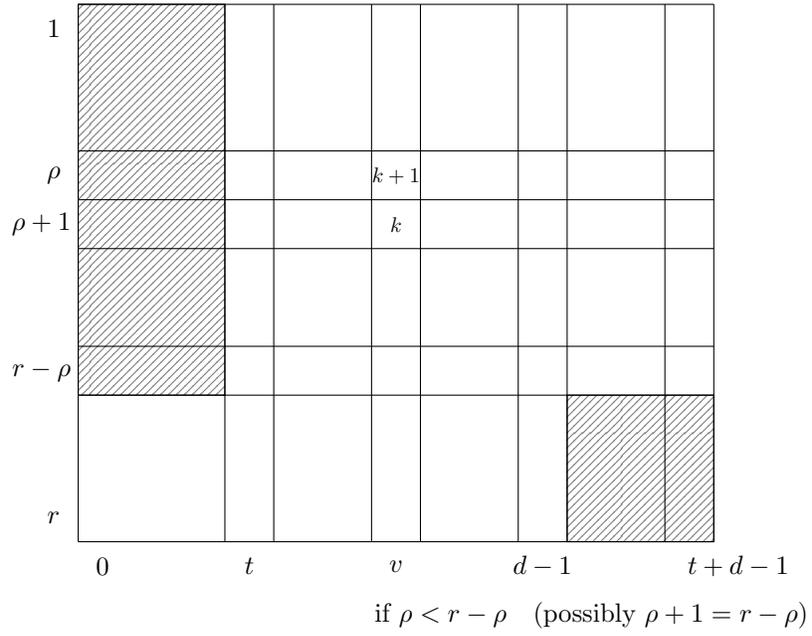
\begin{center}
	\begin{tikzpicture}[scale=0.65]
		\draw[pattern=north east lines, pattern color=gray] (0,3) rectangle (3,11);
		\draw[pattern=north east lines, pattern color=gray] (10,0) rectangle (13,3);
		\draw (0,0) -- (13,0);
		\draw (0,3) -- (13,3);
		\draw (0,4) -- (13,4);
		\draw (0,6) -- (13,6);
		\draw (0,7) -- (13,7);
		\draw (0,8) -- (13,8);
		\draw (0,11) -- (13,11);
		\draw (0,0) -- (0,11);
		\draw (3,0) -- (3,11);
		\draw (4,0) -- (4,11);
		\draw (6,0) -- (6,11);
		\draw (7,0) -- (7,11);
		\draw (9,0) -- (9,11);
		\draw (10,0) -- (10,11);
		\draw (12,0) -- (12,11);
		\draw (13,0) -- (13,11);
		\node (r1) at (-0.5, 10.5) {1};
		\node (rrho) at (-0.5, 7.5) {\(\rho\)};
		\node (rrho+1) at (-0.75,6.5) {\(\rho+1\)};
		\node (rr-rho) at (-0.75,3.5) {\(r- \rho\)};
		\node (rr) at (-0.5,0.5) {\(r\)};
		\node (c0) at (0.5,-0.5) {\(0\)};
		\node (ct) at (3.5,-0.5) {\(t\)};
		\node (cv) at (6.5,-0.5) {\(v\)};
		\node (cd-1) at (9.5,-0.5) {\(d-1\)};
		\node (ct+d-1) at (13.5,-0.5) {\(t+d-1\)};
		\node[scale=0.8] (k) at (6.5,6.5) {\(k\)};
		\node[scale=0.8] (k+1) at (6.5,7.5) {\(k+1\)};
		\node (text) at (10.5,-1.5) {\(\text{if } \rho < r-\rho \quad (\text{possibly } \rho+1=r-\rho)\)};
	\end{tikzpicture}
\end{center}
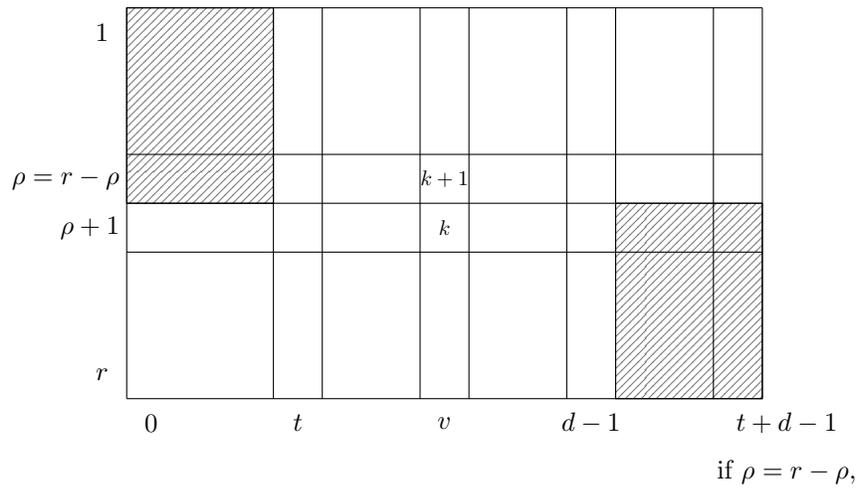
\begin{center}
	\begin{tikzpicture}[scale=0.65]
		\draw[pattern=north east lines, pattern color=gray] (0,4) rectangle (3,8);
		\draw[pattern=north east lines, pattern color=gray] (10,0) rectangle (13,4);
		\draw (0,0) -- (13,0);
		\draw (0,3) -- (13,3);
		\draw (0,4) -- (13,4);
		\draw (0,5) -- (13,5);
		\draw (0,8) -- (13,8);
		\draw (0,0) -- (0,8);
		\draw (3,0) -- (3,8);
		\draw (4,0) -- (4,8);
		\draw (6,0) -- (6,8);
		\draw (7,0) -- (7,8);
		\draw (9,0) -- (9,8);
		\draw (10,0) -- (10,8);
		\draw (12,0) -- (12,8);
		\draw (13,0) -- (13,8);
		\node (r1) at (-0.5,7.5) {1};
		\node (rr-rho) at (-1.25,4.5) {\(\rho = r-\rho\)};
		\node (rrho+1) at (-0.75,3.5) {\(\rho+1\)};
		\node (rr) at (-0.5,0.5) {\(r\)};
		\node (c0) at (0.5,-0.5) {\(0\)};
		\node (ct) at (3.5,-0.5) {\(t\)};
		\node (cv) at (6.5,-0.5) {\(v\)};
		\node (cd-1) at (9.5,-0.5) {\(d-1\)};
		\node (ct+d-1) at (13.5,-0.5) {\(t+d-1\)};
		\node (label) at (13.5,-1.5) {\(\text{if } \rho = r-\rho,\)};
		\node[scale=0.8] (k) at (6.5,3.5) {\(k\)};
		\node[scale=0.8] (k+1) at (6.5,4.5) {\(k+1\)};
	\end{tikzpicture}
\end{center}
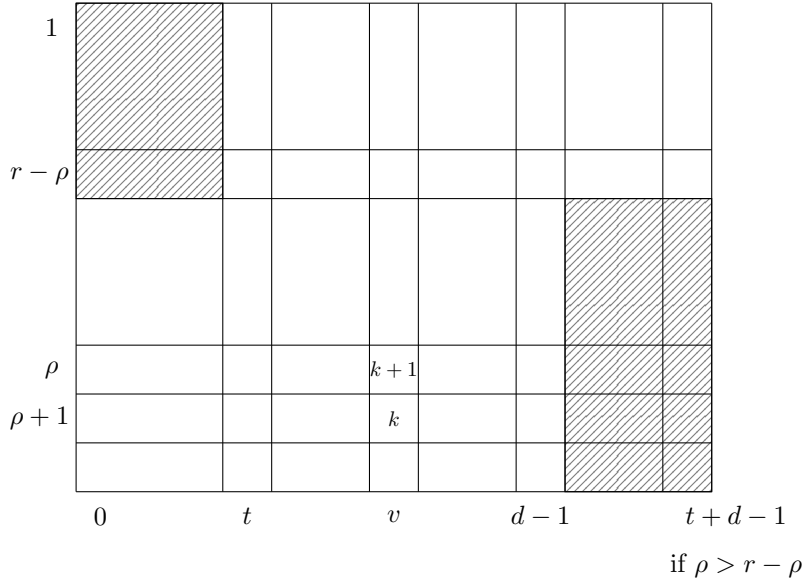
\begin{center}
	\begin{tikzpicture}[scale=0.65]
		\draw[pattern=north east lines, pattern color=gray] (0,6) rectangle (3,10);
		\draw[pattern=north east lines, pattern color=gray] (10,0) rectangle (13,6);
		\draw (0,0) -- (13,0);
		\draw (0,1) -- (13,1);
		\draw (0,2) -- (13,2);
		\draw (0,3) -- (13,3);
		\draw (0,6) -- (13,6);
		\draw (0,7) -- (13,7);
		\draw (0,10) -- (13,10);
		\draw (0,0) -- (0,10);
		\draw (3,0) -- (3,10);
		\draw (4,0) -- (4,10);
		\draw (6,0) -- (6,10);
		\draw (7,0) -- (7,10);
		\draw (9,0) -- (9,10);
		\draw (10,0) -- (10,10);
		\draw (12,0) -- (12,10);
		\draw (13,0) -- (13,10);
		\node (r1) at (-0.5,9.5) {1};
		\node (rr-rho) at (-0.75,6.5) {\(r-\rho\)};
		\node (rrho) at (-0.5,2.5) {\(\rho\)};
		\node (rrho+1) at (-0.75,1.5) {\(\rho+1\)};
		\node (c0) at (0.5,-0.5) {0};
		\node (ct) at (3.5,-0.5) {\(t\)};
		\node (cv) at (6.5,-0.5) {\(v\)};
		\node (cd-1) at (9.5,-0.5) {\(d-1\)};
		\node (ct+d-1) at (13.5,-0.5) {\(t+d-1\)};
		\node (label) at (13.5,-1.5) {\(\text{if } \rho > r-\rho\)};
		\node[scale=0.8] (k) at (6.5,1.5) {\(k\)};
		\node[scale=0.8] (k+1) at (6.5,2.5) {\(k+1\)};
	\end{tikzpicture}
\end{center}
In each case, some of the quantities \(t \leq v \leq d-1\) may coincide, and the shadowed areas correspond to non-boxes. Now replace \(\mathbf{n} = (\underbrace{t,\dots,t}_{r-\rho}, \underbrace{0,\dots,0}_{\rho})\) by
\begin{align}
	\mathbf{n}'	&\defeq (\underbrace{t,\dots,t,t-1,}_{r-\rho} \underbrace{0,\dots,0}_{\rho}) 	&&\text{if } \rho \leq r-\rho, \quad \text{and by} \\
	\mathbf{n}'	&\defeq (\underbrace{t,\dots,t,}_{r-\rho} \underbrace{1,0,\dots,0}_{\rho})	 	&&\text{if } \rho > r-\rho. \nonumber
\end{align}
Then upon \(\mathbf{n} \leadsto \mathbf{n}'\), the \(k\)-th and the \((k+1)\)-th \(d\)-box move down one step if \(\rho \leq r-\rho\) and up one step if \(\rho > r - \rho\), but in both cases, 
\(v(B_{k}^{(d)}(\mathbf{n}')) = v(B_{k+1}^{(d)}(\mathbf{n}')) = v\), and \(\mathbf{n}'\) is a neighbor of \(\mathbf{n}\) in \({}_{1}\mathcal{W}(d,k)\). As \(d(\mathbf{n}', \mathbf{0}) = d(\mathbf{n}, \mathbf{0}) = t\) didn't increase, 
we may resume with devices \ref{Subsection.Special-vertex-in-1Wdk}, \ref{Subsection.Second-special-vertex}, \ref{Subsection.Fourth-special-vertex}, by which we finally arrive at some \(\mathbf{n}_{i}\) with \(1 \leq i < r\). Thus we have shown:

\begin{Proposition}
	\(\mathcal{W}(d,k)\) is connected as a simplicial complex.
\end{Proposition}

We may now state and prove the first main result.

\begin{Theorem} \label{Theorem.Simplicial-complexes-connected}
	The simplicial complexes \(\mathcal{W}(d,k)\), \(\mathcal{A}(d,k)\), and \(\mathcal{BT}(d,k)\) are connected.
\end{Theorem}

\begin{proof}
	\begin{enumerate}[wide, label=(\roman*)]
		\item For \(\mathcal{W}(d,k)\), see above. We have
		\[
			\mathcal{A}(d,k) = W \mathcal{W}(d,k) = \bigcup_{w \in W} w\mathcal{W}(d,k)
		\]
		with the Weyl group \(W\), and the different sheets \(w\mathcal{W}(d,k)\) are connected along the vertices \(w\mathbf{n}_{i}\); so \(\mathcal{A}(d,k)\) is connected. 
		\item Recall that \(\mathcal{BT}(d,k) = \Gamma \mathcal{W}(d,k)\). Let \enquote{\(\sim\)} be the equivalence relation on \(\Gamma = \GL(r,A)\)
		\begin{multline}
			\gamma \sim \delta \vcentcolon \Longleftrightarrow \gamma \mathcal{A}(d,k) \text{ lies in the same connected } \\ \text{component of \(\mathcal{BT}(d,k)\) as \(\delta\mathcal{A}(d,k)\) (\(\gamma, \delta \in \Gamma\))}.
		\end{multline}
		Then \(\Delta \defeq \{ \gamma \in \Gamma \mid \gamma \sim 1\}\) is a subgroup of \(\Gamma\). It contains \(\Gamma_{\mathbf{n}} \defeq \{ \gamma \in \Gamma \mid \gamma \mathbf{n} = \mathbf{n}\}\) for each
		\(\mathbf{n} \in \mathcal{W}(d,k)\) as well as the subgroup \(W\) of permutation matrices of \(\Gamma\). Thus we have to show that \(\Gamma\) equals the group \(\Delta\) generated by all these.
		\item Given \(\mathbf{n} = (n_{1}, \dots, n_{r-1}, 0) \in \mathcal{W}(\mathds{Z})\), write the set \( \{n_{i} \mid 1 \leq i \leq r\}\) as \( \{m_{1}, \dots, m_{s}\}\) with \(m_{1} > m_{2} > \dots > m_{s} = 0\), where \(m_{j}\) occurs \(r_{j}\) times
		in \(n_{1}, \dots, n_{r}\), and \(\sum_{1 \leq j \leq s} r_{j} = r\). Then (as is well-known and easily checked):
		\begin{equation} \label{Eq.Block-structure-of-matrices-in-Gamma-n}
			\Gamma_{\mathbf{n}} \text{ is the group of matrices with a block structure}
		\end{equation}
		\begin{center}
			\begin{tikzpicture}[scale=0.65]
				\draw (0,0) rectangle (8,8);
				\draw (0,6) rectangle (2,8);
				\draw (2,6) rectangle (3,8);
				\draw (2,5) rectangle (3,6);
				\draw (6,6) rectangle (8,8);
				\draw (6,2) rectangle (8,3);
				\draw (6,0) rectangle (8,2);
				\draw[dotted] (3,6) -- (6,6);
				\draw[dotted] (6,6) -- (6,3);
				\node (A1) at (1,7) {\(A_{1}\)};
				\node (B12) at (2.5,7) {\(B_{1,2}\)};
				\node (A2) at (2.5,5.5) {\(A_{2}\)};
				\node (B1s) at (7,7) {\(B_{1,s}\)};
				\node (Bs-1s) at (7, 2.5) {\(B_{s-1,s}\)};
				\node (As) at (7,1) {\(A_{s}\)};
				\node (dots) at (4.5,4.5) {\(\ddots\)};
				\node (0) at (1,1) {0};
			\end{tikzpicture}
		\end{center}
		where \(A_{j} \in \GL(r_{j}, \mathds{F})\), for \( 1 \leq i < j \leq s\), \(B_{i,j}\) is an \( (r_{i} \times r_{j})\)-matrix with entries \(a \in A = \mathds{F}[T]\) of degree \(\deg a \leq m_{i} - m_{j}\), and the blocks below the
		diagonal vanish.
		\item Suppose \fbox{\(k \leq d\)}. Then \(\mathcal{W}(d,k) = \mathcal{W}(k)\), which contains some infinite half-line \(\mathbf{n}, \mathbf{n} + \mathbf{n}_{1}, \mathbf{n} + 2\mathbf{n}_{1}, \dots,\) as can be seen from 
		\(\mathcal{W}(1) = \mathcal{W}_{r-1}\) and Corollary \ref{Corollary.Membership-of-AIZ}(i). If \fbox{\(k>d\)} then by the basic decomposition \ref{Subsection.Decomposition-of-WdkIZ}, each \(\mathbf{n} \in {}_{3}\mathcal{W}(d,k) \cup {}_{4}\mathcal{W}(d,k)\) gives rise to an infinite sequence 
		\(\mathbf{n}\), \(\mathbf{n} + \mathbf{y}\), \(\mathbf{n} + 2\mathbf{y}\), \dots (\( \mathbf{y} = \mathbf{n}_{r-1}\)). In any case, there exists an infinite half-line in \(\mathcal{W}(d,k)\): 
		\begin{equation}
			\mathbf{n} = \mathbf{n}^{(0)}, \quad \mathbf{n}^{(1)} = \mathbf{n} + \mathbf{n}_{i}, \quad \mathbf{n}^{(2)} = \mathbf{n} + 2\mathbf{n}_{i}, \dots
		\end{equation}
		with some \(i\), \(1\leq i < r\). As follows from \eqref{Eq.Block-structure-of-matrices-in-Gamma-n}, the union \(\bigcup_{j \geq 0} \Gamma_{\mathbf{n}^{(j)}}\) contains the subgroup \(U_{i}\) of \(\Gamma\) of matrices with an \(i \times (r-i)\) block structure 
		\begin{center}
			\begin{tikzpicture}[scale=0.5]
				\draw (0,0) rectangle (3,3);
				\draw (1,0) -- (1,3);
				\draw (0,2) -- (3,2);
				\node (0) at (0.5, 1) {0};
				\node (idi) at (0.5,2.5) {\(\mathrm{Id}_{i}\)};
				\node (idr-i) at (2,1) {\(\mathrm{Id}_{r-i}\)};
				\node (star) at (2,2.5) {*};
			\end{tikzpicture},
		\end{center} 
		where the entries of the \(i \times (r-i)\)-matrix \(*\) are arbitrary elements of \(A\).
		\item By \ref{Subsubsection.Membership-of-Wdk}, each \(\mathbf{n}_{j}\) (\(0 \leq j < r\)) with \(j+k \not\equiv 0 \pmod{r}\) belongs to \(\mathcal{W}(d,k)\). This fact, together with \eqref{Eq.Block-structure-of-matrices-in-Gamma-n} and \(W \subset \Delta\), implies that in fact \(\GL(r, \mathds{F}) \subset \Delta\).
		Hence it suffices to show that \(\GL(r, \mathds{F}))\) and \(U_{i}\) together generate \(\Gamma\). This is an exercise in matrix groups.
		\item First, \(\Delta\) contains the root group \(B_{i,i+1}\) of matrices with 1 on the diagonal, an arbitrary element of \(A\) at the \( (i,i+1)\)-entry, and all other coefficients 0. By the action of \(W\), it also contains the
		other root groups \(B_{\ell, \ell+1}\) (\(1 \leq \ell < r\)), hence the group of strictly upper triangular matrices with coefficients in \(A\), as well as its transpose. It is well-known that (since \(A\) is a principal ideal domain) these
		together generate \(\SL(r,A)\), and then together with \(\GL(r, \mathds{F})\), the full group \(\Gamma = \GL(r,A)\). Hence the connected component of \(\mathcal{A}(d,k)\) in \(\mathcal{BT}(d,k)\) equals
		\(\mathcal{BT}(d,k)\), and we are done.
	\end{enumerate}
\end{proof}

\section{Strong equidimensionality} \label{Section.Strong-equidimensionality}

In [V] we established equidimensionality of \(\mathcal{BT}(k)\) and its subcomplexes \(\mathcal{A}(k)\) and \(\mathcal{W}(k)\); that is, each vertex is contained in a simplex of dimension \(\dim \mathcal{BT}(k) = r-2\). Here we show,
both for \(\mathcal{BT}(k)\) and \(\mathcal{BT}(d,k)\), the stronger statement \ref{Subsection.Simpliciality-of-modular-form} (c strong) that each simplex is contained in a simplex of dimension \(r-2\). Note that by the very statement, and since each simplex
of \(\mathcal{BT}\) is \(\Gamma\)-conjugate to one in \(\mathcal{W}\), it suffices to show strong equidimensionality for \(\mathcal{W}(k)\) and \(\mathcal{W}(d,k)\), respectively.

We keep our assumptions on the numbers \(r\), \(d\) and \(k\): We always have \(r \geq 3\), \(d \geq 2\), and \(2 \leq k < rd\).

\subsection{} As the quantities \(v_{\ell}^{(d)}(\mathbf{m})\) and \(B_{\ell}^{(d)}(\mathbf{m})\) are defined only for \( \ell \leq rd\), we assume in all formulas below involving these that \(\ell \leq rd\). We call
\(j \in \{1,2,\dots,r-1\}\) \textbf{admissible} for \(\mathbf{m} = (m_{1}, \dots, m_{r-1}, 0) \in \mathcal{W}(\mathds{Z})\) if
\begin{equation}
	j=1 \quad \text{or} \quad m_{j} < m_{j-1}.
\end{equation}
This is equivalent with \(\mathbf{m} + \mathbf{e}_{j} \in \mathcal{W}\). A vertex \(\mathbf{m} \in \mathcal{W}(d,k)\) is \textbf{capped} with respect to \(k\) or \(k\)-\textbf{capped} if 
\begin{equation}
	\rho \defeq \rho_{k}^{(d)}(\mathbf{m}) = 1 \quad \text{or} \quad m_{\rho -1} > v(B_{k}^{(d)}(\mathbf{m})).
\end{equation}
This means that in the column of \(B_{k}^{(d)}(\mathbf{m})\) and \(B_{k+1}^{(d)}(\mathbf{m})\) in \(\diag^{(d)}(\mathbf{m})\) there are no boxes above \(B_{k+1}^{(d)}(\mathbf{m})\) or, equivalently, \(\mathbf{m} \notin \mathcal{W}(d,k+1)\).

In the following lemmas, we mainly argue with \(d\)-diagrams of various vertices. Usually we draw only the relevant parts of these and informally call them \enquote{pictures}. A star \(*\) in an entry of the corresponding
matrix indicates a box of \(\diag^{(d)}(\mathbf{m})\), briefly a box of \(\mathbf{m}\). If the entry is shadowed \begin{tikzpicture}[scale=0.3] \draw[pattern=north east lines, pattern color=gray] (0,0) rectangle (1,1);\end{tikzpicture}, it is a
\textbf{non-box} (=not a box). Entries neither starred nor shadowed are irrelevant for the argument.

\begin{Lemma} \label{Lemma.Boxes-of-sums-of-certain-vertices}
	Let \(\mathbf{m}\) be a vertex of \(\mathcal{W}(d,k)\) and \(j\) admissible for \(\mathbf{m}\). Put \(\mathbf{m}' \defeq \mathbf{m} + \mathbf{e}_{j} \in \mathcal{W}\), \(v \defeq v_{k}^{(d)}(\mathbf{m})\), 
	\(\rho \defeq \rho_{k}^{(d)}(\mathbf{m})\), \(v' \defeq v_{k}^{(d)}(\mathbf{m}')\), and, provided that \(\mathbf{m}' \in \mathcal{W}(d,k)\), \(\rho' \defeq \rho_{k}^{(d)}(\mathbf{m}')\).
	\begin{enumerate}[label=\(\mathrm{(\roman*)}\)]
		\item Let \(j \neq \rho\) and \(m_{j} \leq v-d\) or \(m_{j} \geq v\). Then for \(\ell = k\) or \(k+1\), \(B_{\ell}^{(d)}(\mathbf{m}') = B_{\ell}^{(d)}(\mathbf{m})\). Thus \(\mathbf{m}' \in \mathcal{W}(d,k)\) and \(v' = v\), \(\rho' = \rho\).
		\item Let \(j \neq \rho\) and \(v-d < m_{j} < v\). Then for \(\ell = k,k+1,k+2\),
		\begin{equation}
			B_{\ell}^{(d)}(\mathbf{m}) = B_{\ell-1}^{(d)}(\mathbf{m}').
		\end{equation}
		If then \(\mathbf{m}\) is \(k\)-capped (which can only happen if \(j > \rho\)), then \(\mathbf{m}' \notin \mathcal{W}(d,k)\). Otherwise, \(\mathbf{m}' \in \mathcal{W}(d,k)\), \(v' = v\), \(\rho' = \rho-1\).
		\item Let \(j=\rho\). If \(\mathbf{m}\) is \(k\)-capped then \(\mathbf{m}' \notin \mathcal{W}(d,k)\); otherwise \(\mathbf{m}' \in \mathcal{W}(d,k)\), \(v' = v\), \(\rho' = \rho-1\).
	\end{enumerate}
\end{Lemma}

\begin{proof}
	We must go through the cases \(\rho < j\), \( \rho = j\), \(\rho > j\) with possible subcases. 
	
	Assume \fbox{\(\rho < j\)}. As \(m_{j} < m_{\rho} \leq v\), we have \(m_{j} < v\). Assume further that \fbox{\(v < d\)}. The picture for the \(d\)-diagram of \(\mathbf{m}\) looks 
	
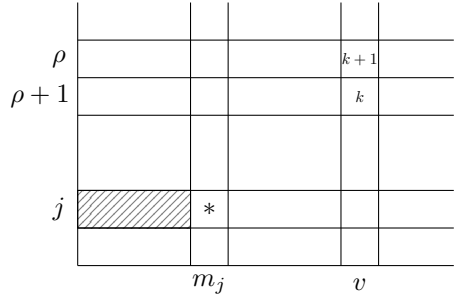
\begin{center}
		\begin{tikzpicture}[scale=0.5]
			\draw[pattern=north east lines, pattern color=gray] (0,1) rectangle (3,2);
			\draw (0,0) -- (10,0);
			\draw (0,1) -- (10,1);
			\draw (0,2) -- (10,2);
			\draw (0,4) -- (10,4);
			\draw (0,5) -- (10,5);
			\draw (0,6) -- (10,6);
			\draw (0,0) -- (0,7);
			\draw (3,0) -- (3,7);
			\draw (4,0) -- (4,7);
			\draw (7,0) -- (7,7);
			\draw (8,0) -- (8,7);
			\node (rj) at (-0.5,1.5) {\(j\)};
			\node (rrho+1) at (-1,4.5) {\(\rho+1\)};
			\node (rrho) at (-0.5,5.5) {\(\rho\)};
			\node (cmj) at (3.5,-0.5) {\(m_{j}\)};
			\node (cv) at (7.5,-0.5) {\(v\)};
			\node[scale=0.6] (k) at (7.5,4.5) {\(k\)};
			\node[scale=0.6] (k+1) at (7.5,5.5) {\(k+1\)};
			\node (star) at (3.5,1.5) {\(*\)};
		\end{tikzpicture}
	\end{center}
	As some box (here marked with a \(*\)) to the left of the \(v\)-column is erased under \(\mathbf{m} \leadsto \mathbf{m}'\), we have \( B_{\ell}^{(d)}(\mathbf{m}) = B_{\ell-1}^{(d)}(\mathbf{m}')\) for \(\ell = k,k+1,k+2\). If
	\(\mathbf{m}\) is not \(k\)-capped, then \(B_{k+2}^{(d)}(\mathbf{m}) = B_{k+1}^{(d)}(\mathbf{m}')\) has still value \(v\), so \(\mathbf{m}' \in \mathcal{W}(d,k)\), \(v' = v\), \(\rho' = \rho-1\). If \(\mathbf{m}\) is \(k\)-capped then
	\(B_{k+1}^{(d)}(\mathbf{m}')\) has value \(v+1\), and \(\mathbf{m}' \notin \mathcal{W}(d,k)\). We call this phenomenon a \textbf{column break}.
	
	Assume \fbox{\(v \geq d\)}. In this case, if \(m_{j} \leq v-d\), then \(B_{\ell}^{(d)}(\mathbf{m}') = B_{\ell}^{(d)}(\mathbf{m})\) for \(\ell = k,k+1\) (\(\mathbf{m} \leadsto \mathbf{m}'\) shifts the \(j\)-line of 
	\(\diag^{(d)}\) one to the right, but still \( (j,v) \) is a non-box for \(\diag^{(d)}(\mathbf{m}')\)), so \(\mathbf{m}' \in \mathcal{W}(d,k)\). Otherwise, \(v-d < m_{j} < v\), and the discussion and results of the subcase \(v < d\) prevail.
	So for \(\rho < j\), everything comes out as asserted.
	
	Next, consider the case \fbox{\(\rho = j\)}. We have \(m_{j} = m_{\rho} \leq v < m_{\rho} + d\). If \fbox{\(m_{\rho} = v\)} then, as \(j = \rho\) is admissible, \(\mathbf{m}\) is \(k\)-capped,
	\(B_{k}^{(d)}(\mathbf{m}') = B_{k}^{(d)}(\mathbf{m})\) with value \(v\), while \(B_{k+1}^{(d)}(\mathbf{m}')\) belongs to the next column. Thus we have a column break and \(\mathbf{m}' \notin \mathcal{W}(d,k)\). If
	\fbox{\(m_{\rho} < v\) and \(\mathbf{m}\) is \(k\)-capped}, then \(B_{k-1}^{(d)}(\mathbf{m}') = B_{k}^{(d)}(\mathbf{m})\), \(B_{k}^{(d)}(\mathbf{m}') = B_{k+1}^{(d)}(\mathbf{m})\) and \(B_{k+1}^{(d)}(\mathbf{m}')\) is the
	\((v+1)\)-th column. If \fbox{\(m_{\rho} < v\) and \(\mathbf{m}\) is not \(k\)-capped} then again \(B_{k-1}^{(d)}(\mathbf{m}') = B_{k}^{(d)}(\mathbf{m})\), \(B_{k}^{(d)}(\mathbf{m}') = B_{k+1}^{(d)}(\mathbf{m})\), but 
	\(B_{k+1}^{(d)}(\mathbf{m}') = (\rho-1,v)\). All three subcases are as asserted.
	
	Finally, assume \fbox{\( \rho > j\)}. As \(m_{j} \geq m_{\rho} > v-d\), we have \(m_{j} > v-d\). If \fbox{\(m_{j} \geq v\)} then the numbering of boxes \(B_{\ell}^{(d)}\) of \(\diag^{(d)}\) with \(\ell \leq k+1\) is not affected by
	\(\mathbf{m} \leadsto \mathbf{m}'\). In particular \(B_{\ell}^{(d)}(\mathbf{m}') = B_{\ell}^{(d)}(\mathbf{m})\) for \(\ell = k, k+1\), so \(\mathbf{m}' \in \mathcal{W}(d,k)\), \(v' = v\), \(\rho' = \rho\). If otherwise
	\fbox{\(v-d < m_{j} < v\)} then as before \(B_{\ell}^{(d)}(\mathbf{m}) = B_{\ell-1}^{(d)}(\mathbf{m}')\) for \(\ell=k,k+1,k+2\), so the reasoning above implies: \(B_{k}^{(d)}(\mathbf{m}') = \mathbf{B}_{k-1}^{(d)}(\mathbf{m})\)
	and \(B_{k+1}^{(d)}(\mathbf{m}') = B_{k+2}^{(d)}(\mathbf{m})\) have the same value \(v\), as our assumptions imply that \(\mathbf{m}\) is not \(k\)-capped. Hence in this subcase \(\mathbf{m}' \in \mathcal{W}(d,k)\), too, and
	\(v' = v\), \(\rho' = \rho -1\).
\end{proof}

\begin{Corollary} \label{Corollary.Condition-for-k-capped}
	Let \(\mathbf{m}\) be a vertex of \(\mathcal{W}(d,k)\) for which there is an admissible \(j\) such that \(\mathbf{m} + \mathbf{e}_{j} \notin \mathcal{W}(d,k)\). Then \(j \geq \rho_{k}^{(d)}(\mathbf{m})\), and \(\mathbf{m}\) is 
	\(k\)-capped. \hfill \mbox{\(\square\)}
\end{Corollary}

Suppose that the vertex \(\mathbf{m}'\) of \(\mathcal{W} \smallsetminus \mathcal{W}(d,k)\) results from \(\mathbf{m} \leadsto \mathbf{m}' = \mathbf{m} + \mathbf{e}_{j}\) with a column break as in Lemma \ref{Lemma.Boxes-of-sums-of-certain-vertices}. What 
happens upon adding another \(\mathbf{e}_{j'}\) to \(\mathbf{m}'\)?

\begin{Lemma} \label{Lemma.Membership-of-sums-of-vertices}
	Let \(\mathbf{m} \in \mathcal{W}(d,k)\) and \(\mathbf{m}' = \mathbf{m} + \mathbf{e}_{j}\) be as in Lemma \ref{Lemma.Boxes-of-sums-of-certain-vertices}, where \(\mathbf{m}' \in \mathcal{W}(d,k)\). Let further \(j' \neq j\) be admissible for \(\mathbf{m}'\) and 
	\(\mathbf{m}'' \defeq \mathbf{m}' + \mathbf{e}_{j'}\). Then
	\begin{enumerate}[label=\(\mathrm{(\roman*)}\)]
		\item if \(j' < \rho \defeq \rho_{k}^{(d)}(\mathbf{m})\) then \(\mathbf{m}'' \notin \mathcal{W}(d,k)\);
		\item if \(\rho \leq j' < j\) then \(\mathbf{m}'' \in \mathcal{W}(d,k)\);
		\item if \(j < j'\) then \(\mathbf{m}'' \in \mathcal{W}(d,k)\) if and only if \(v-d < m_{j'}\), where \(v \defeq v_{k}^{(d)}(\mathbf{m})\).
	\end{enumerate}
\end{Lemma}	

\begin{proof}
	First note that \(j \geq \rho\) by Corollary \ref{Corollary.Condition-for-k-capped}, so that (i), (ii), (iii) cover all possible cases. Besides other trivial facts about \(d\)-diagrams, we will use:
	\subsubsection{} \label{Subsubsection.Consecutive-set} Given a value \(w\) and a vertex \(\mathbf{n}\) of \(\mathcal{W}\), the set \( \{ i \in \mathds{N}_{0} \mid (i,w) \text{ is a box of } \diag^{(d)}(\mathbf{n}) \}\) is 
	consecutive, that is, has no gaps. \stepcounter{equation}%
	
	Let us first treat the \textbf{Case (a)} \fbox{\(j = \rho\) and \(m_{\rho} = v\)}. Here the relevant parts of \(\diag^{(d)}(\mathbf{m})\) and \(\diag^{(d)}(\mathbf{m}')\) are given by the pictures 
	
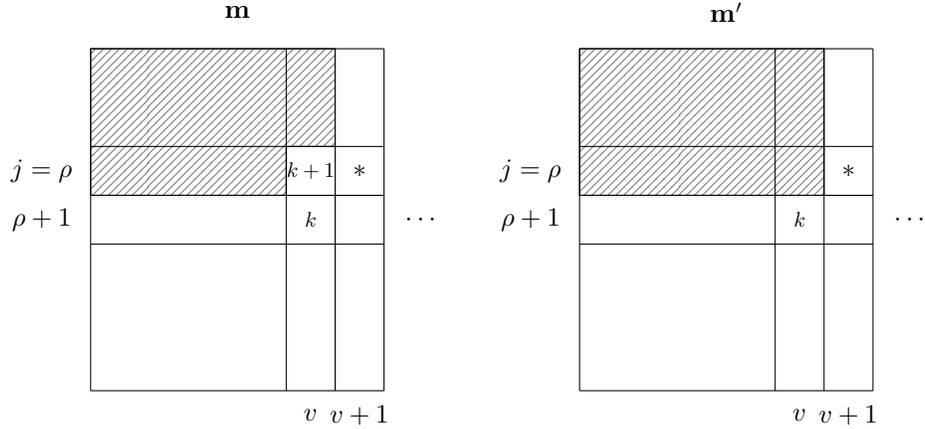
\begin{center}
		\begin{tikzpicture}[scale=0.65]
			\draw[pattern=north east lines, pattern color=gray] (0,4) -- (0,7) -- (5,7) -- (5,5) -- (4,5) -- (4,4) -- cycle;
			\draw (0,0) -- (6,0);
			\draw (0,3) -- (6,3);
			\draw (0,4) -- (6,4);
			\draw (0,5) -- (6,5);
			\draw (0,7) -- (6,7);
			\draw (0,0) -- (0,7);
			\draw (4,0) -- (4,7);
			\draw (5,0) -- (5,7);
			\draw (6,0) -- (6,7);
			\node (m) at (3,7.75) {\(\mathbf{m}\)};
			\node (rho+1) at (-1,3.5) {\(\rho+1\)};
			\node (jerho) at (-1,4.5) {\(j=\rho\)};
			\node (v) at (4.5,-0.5) {\(v\)};
			\node (v+1) at (5.5,-0.5) {\(v+1\)};
			\node (star) at (5.5,4.5) {\(*\)};
			\node[scale=0.8] (k) at (4.5,3.5) {\(k\)};
			\node[scale=0.8] (k+1) at (4.5,4.5) {\(k+1\)};
			\node (dots) at (6.75,3.5) {\(\cdots\)};
			
			\draw[pattern=north east lines, pattern color=gray] (10,4) rectangle (15,7);
			\draw (10,0) -- (16,0);
			\draw (10,3) -- (16,3);
			\draw (10,4) -- (16,4);
			\draw (10,5) -- (16,5);
			\draw (10,7) -- (16,7);
			\draw (10,0) -- (10,7);
			\draw (14,0) -- (14,7);
			\draw (15,0) -- (15,7);
			\draw (16,0) -- (16,7);
			\node (m2) at (13,7.75) {\(\mathbf{m}'\)};
			\node (2rho+1) at (9,3.5) {\(\rho+1\)};
			\node (2jerho) at (9,4.5) {\(j=\rho\)};
			\node (2v) at (14.5,-0.5) {\(v\)};
			\node (2v+1) at (15.5,-0.5) {\(v+1\)};
			\node (2star) at (15.5,4.5) {\(*\)};
			\node[scale=0.8] (2k) at (14.5,3.5) {\(k\)};
			\node (2dots) at (16.75,3.5) {\(\cdots\)};
		\end{tikzpicture}
	\end{center}
	In particular, \( i(B_{k+1}^{(d)}(\mathbf{m}')) > \rho \). This implies that if \fbox{\(j' < \rho\)}, the numbering of \(B_{\ell}^{(d)}\) for \(\ell \leq k+1\) remains unchanged under \(\mathbf{m}' \leadsto \mathbf{m}''\), so 
	\(\mathbf{m}'' \notin \mathcal{W}(d,k)\). Thus suppose \fbox{\(j' > \rho\)}. If \( (j', v)\) is a non-box for \(\mathbf{m}\) (or for \(\mathbf{m}'\), which amounts to the same), then still 
	\(B_{k}^{(d)}(\mathbf{m}'') = B_{k}^{(d)}(\mathbf{m}') = B_{k}^{(d)}(\mathbf{m}) = (\rho+1,v)\), while \(v(B_{k+1}^{(d)}(\mathbf{m}'')) = v+1\), thus \(\mathbf{m}'' \notin \mathcal{W}(d,k)\). So assume that \( (j',v) \) is a box of \(\mathbf{m}\)
	(and \(\mathbf{m}'\)). If \(j' > \rho+1\) then \(m_{j'} < v\) by the admissibility of \(j'\), so \(B_{k}^{(d)}(\mathbf{m}') = B_{k-1}^{(d)}(\mathbf{m}'')\), and both \(B_{k}^{(d)}(\mathbf{m}'')\) and \(B_{k+1}^{(d)}(\mathbf{m}'')\) lie in the
	\((v+1)\)-column, which gives \(\mathbf{m}'' \in \mathcal{W}(d,k)\). The same argument works if \(j' = \rho+1\) and \(m_{j'} = m_{\rho+1} < v\). If \(j' = \rho + 1\) and \(m_{j'} = v\), then \( (\rho+1,v)\) is a non-box for \(\mathbf{m}''\), and
	both \(B_{k}^{(d)}(\mathbf{m}'')\) and \(B_{k+1}^{(d)}(\mathbf{m}'')\) have value \(v+1\). So the case (a) is as announced, taking into account that
	\begin{equation} \label{Eq.Characterization-of-boxes}
		(j',v) \text{ a box of } \mathbf{m} \Longleftrightarrow v-d < m_{j'} \leq v \Longleftrightarrow v-d < m_{j'},
	\end{equation}
	as \(m_{j'} > v\) is impossible. 
	
	Now consider \textbf{Case (b)} \fbox{\(j=\rho\) and \(m_{\rho} <v\)}. The pictures of \(\mathbf{m}\) and \(\mathbf{m}'\) are
	\begin{center}
		\begin{tikzpicture}[scale=0.65]
			\draw[pattern=north east lines, pattern color=gray] (0,5) rectangle (5,7);
			\draw (0,0) -- (6,0);
			\draw (0,3) -- (6,3);
			\draw (0,4) -- (6,4);
			\draw (0,5) -- (6,5);
			\draw (0,6) -- (6,6);
			\draw (0,7) -- (6,7);
			\draw (0,0) -- (0,7);
			\draw (3,0) -- (3,7);
			\draw (4,0) -- (4,7);
			\draw (5,0) -- (5,7);
			\draw (6,0) -- (6,7);
			\node (m) at (3,7.75) {\(\mathbf{m}\)};
			\node (rho+1) at (-1,3.5) {\(\rho+1\)};
			\node (jerho) at (-1,4.5) {\(j=\rho\)};
			\node (v-1) at (3.5,-0.5) {\(v-1\)};
			\node (v) at (4.5,-0.5) {\(v\)};
			\node (v+1) at (5.5,-0.5) {\(v+1\)};
			\node (star1) at (3.5,4.5) {\(*\)};
			\node (star2) at (3.5,3.5) {\(*\)};
			\node[scale=0.8] (k) at (4.5,3.5) {\(k\)};
			\node[scale=0.8] (k+1) at (4.5,4.5) {\(k+1\)};
			\node (dots) at (6.75,3.5) {\(\cdots\)};
			
			\draw[pattern=north east lines, pattern color=gray] (10,5) rectangle (15,7);
			\draw (10,0) -- (16,0);
			\draw (10,3) -- (16,3);
			\draw (10,4) -- (16,4);
			\draw (10,5) -- (16,5);
			\draw (10,6) -- (16,6);
			\draw (10,7) -- (16,7);
			\draw (10,0) -- (10,7);
			\draw (13,0) -- (13,7);
			\draw (14,0) -- (14,7);
			\draw (15,0) -- (15,7);
			\draw (16,0) -- (16,7);
			\node (2m) at (13,7.75) {\(\mathbf{m}'\)};
			\node (2rho+1) at (9,3.5) {\(\rho+1\)};
			\node (2jerho) at (9,4.5) {\(j=\rho\)};
			\node (2v-1) at (13.5,-0.5) {\(v-1\)};
			\node (2v) at (14.5,-0.5) {\(v\)};
			\node (2v+1) at (15.5,-0.5) {\(v+1\)};
			\node (2star1) at (13.5,3.5) {\(*\)};
			\node (2star2) at (15.5,4.5) {\(*\)};
			\node[scale=0.8] (2k) at (14.5,4.5) {\(k\)};
			\node[scale=0.8] (2k-1) at (14.5,3.5) {\(k-1\)};
			\node (2dots) at (16.75,3.5) {\(\cdots\)};
		\end{tikzpicture}
	\end{center}
	Again, \(i(B_{k+1}^{(d)}(\mathbf{m}')) \geq \rho\), and if \fbox{\(j' < \rho\)} then \(B_{\ell}^{(d)}(\mathbf{m}'') = B_{\ell}^{(d)}(\mathbf{m}')\) for \(\ell \leq k+1\), so \(\mathbf{m}'' \notin \mathcal{W}(d,k)\). For \fbox{\(j' > \rho\)} and
	\( (j',v) \) a non-box of \(\mathbf{m}\) or \(\mathbf{m}'\), \(B_{\ell}^{(d)}(\mathbf{m}'') = B_{\ell}^{(d)}(\mathbf{m}')\) for \(\ell = k,k+1\), hence \(\mathbf{m}'' \notin \mathcal{W}(d,k)\). If however \( (j',v)\) is a box of \(\mathbf{m}\), we find
	the picture for \(\mathbf{m}''\)
	\begin{center}
		\begin{tikzpicture}[scale=0.65]
			\draw[pattern=north east lines, pattern color=gray] (0,5) rectangle (5,7);
			\draw (0,0) -- (6,0);
			\draw (0,1) -- (6,1);
			\draw (0,2) -- (6,2);
			\draw (0,3) -- (6,3);
			\draw (0,4) -- (6,4);
			\draw (0,5) -- (6,5);
			\draw (0,7) -- (6,7);
			\draw (0,0) -- (0,7);
			\draw (3,0) -- (3,7);
			\draw (4,0) -- (4,7);
			\draw (5,0) -- (5,7);
			\draw (6,0) -- (6,7);
			\node (m) at (3,7.75) {\(\mathbf{m}''\)};
			\node (rho+1) at (-1,3.5) {\(\rho+1\)};
			\node (jerho) at (-1,4.5) {\(j=\rho\)};
			\node (jprime) at (-0.5,1.5) {\(j'\)};
			\node (v-1) at (3.5,-0.5) {\(v-1\)};
			\node (v) at (4.5,-0.5) {\(v\)};
			\node (v+1) at (5.5,-0.5) {\(v+1\)};
			\node (star1) at (3.5,1.5) {\(*\)};
			\node (star2) at (4.5,1.5) {\(*\)};
			\node (star3) at (5.5,1.5) {\(*\)};
			\node (star4) at (3.5,3.5) {\(*\)};
			\node (star5) at (5.5,4.5) {\(*\)};
			\node[scale=0.8] (k-2) at (4.5,3.5) {\(k-2\)};
			\node[scale=0.8] (k-1) at (4.5,4.5) {\(k-1\)};
			\node (dots) at (6.75,3) {\(\cdots\)};
		\end{tikzpicture}
	\end{center}
	As by \ref{Subsubsection.Consecutive-set} the \((v+1)\)-column contains at least 2 boxes, \(B_{k}^{(d)}(\mathbf{m}'')\) and \(B_{k+1}^{(d)}(\mathbf{m}'')\) are among them, which gives \(\mathbf{m}'' \in \mathcal{W}(d,k)\). Using \eqref{Eq.Characterization-of-boxes} and summarizing, 
	case (b) is as stated.
	
	For the remaining cases, we distinguish (c) \fbox{\(j' < \rho < j\)}, (d) \fbox{\(\rho \leq j' < j\)}, and (e) \fbox{\(\rho < j < j'\)}.
	
	\textbf{Case (c)} is easy: As \(\mathbf{m}\) is \(k\)-capped, \(B_{\ell}^{(d)}(\mathbf{m}'') = B_{\ell}^{(d)}(\mathbf{m}')\) for \(\ell \leq k+1\), which gives the assertion.
	
	\textbf{Case (d)} \fbox{\(\rho \leq j' < j\)} with its subcases (d1) \fbox{\(\rho = j', m_{\rho} = v\)}, (d2) \fbox{\(\rho = j'\), \(m_{\rho} < v\)}, (d3) \fbox{\(\rho < j'\)}. In subcases (d1) and (d2), the pictures for \(\mathbf{m}\) and \(\mathbf{m}'\) are
	
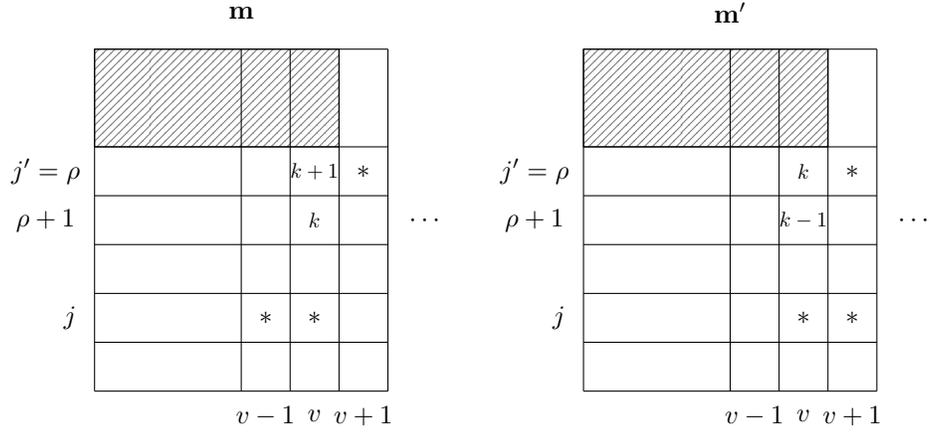
\begin{center}
		\begin{tikzpicture}[scale=0.65]
			\draw[pattern=north east lines, pattern color=gray] (0,5) rectangle (5,7);
			\draw (0,0) -- (6,0);
			\draw (0,1) -- (6,1);
			\draw (0,2) -- (6,2);
			\draw (0,3) -- (6,3);
			\draw (0,4) -- (6,4);
			\draw (0,5) -- (6,5);
			\draw (0,7) -- (6,7);
			\draw (0,0) -- (0,7);
			\draw (3,0) -- (3,7);
			\draw (4,0) -- (4,7);
			\draw (5,0) -- (5,7);
			\draw (6,0) -- (6,7);
			\node (m) at (3,7.75) {\(\mathbf{m}\)};
			\node (rho+1) at (-1,3.5) {\(\rho+1\)};
			\node (jerho) at (-1,4.5) {\(j'=\rho\)};
			\node (jprime) at (-0.5,1.5) {\(j\)};
			\node (v-1) at (3.5,-0.5) {\(v-1\)};
			\node (v) at (4.5,-0.5) {\(v\)};
			\node (v+1) at (5.5,-0.5) {\(v+1\)};
			\node (star1) at (3.5,1.5) {\(*\)};
			\node (star2) at (4.5,1.5) {\(*\)};
			\node (star3) at (5.5,4.5) {\(*\)};
			\node[scale=0.8] (k) at (4.5,3.5) {\(k\)};
			\node[scale=0.8] (k+1) at (4.5,4.5) {\(k+1\)};
			\node (dots) at (6.75,3.5) {\(\cdots\)};
			
			\draw[pattern=north east lines, pattern color=gray] (10,5) rectangle (15,7);
			\draw (10,0) -- (16,0);
			\draw (10,1) -- (16,1);
			\draw (10,2) -- (16,2);
			\draw (10,3) -- (16,3);
			\draw (10,4) -- (16,4);
			\draw (10,5) -- (16,5);
			\draw (10,7) -- (16,7);
			\draw (10,0) -- (10,7);
			\draw (13,0) -- (13,7);
			\draw (14,0) -- (14,7);
			\draw (15,0) -- (15,7);
			\draw (16,0) -- (16,7);
			\node (2m) at (13,7.75) {\(\mathbf{m}'\)};
			\node (2rho+1) at (9,3.5) {\(\rho+1\)};
			\node (2jerho) at (9,4.5) {\(j'=\rho\)};
			\node (2jprime) at (9.5,1.5) {\(j\)};
			\node (2v-1) at (13.5,-0.5) {\(v-1\)};
			\node (2v) at (14.5,-0.5) {\(v\)};
			\node (2v+1) at (15.5,-0.5) {\(v+1\)};
			\node (2star1) at (15.5,1.5) {\(*\)};
			\node (2star2) at (14.5,1.5) {\(*\)};
			\node (2star3) at (15.5,4.5) {\(*\)};
			\node[scale=0.8] (2k-1) at (14.5,3.5) {\(k-1\)};
			\node[scale=0.8] (2k) at (14.5,4.5) {\(k\)};
			\node (2dots) at (16.75,3.5) {\(\cdots\)};
		\end{tikzpicture}
	\end{center}
	since \(j\) (resp. \(j'\)) is admissible for \(\mathbf{m}\) (resp. \(\mathbf{m}'\)) and \(\mathbf{m}\) is \(k\)-capped. For \(\mathbf{m}''\), we find
	
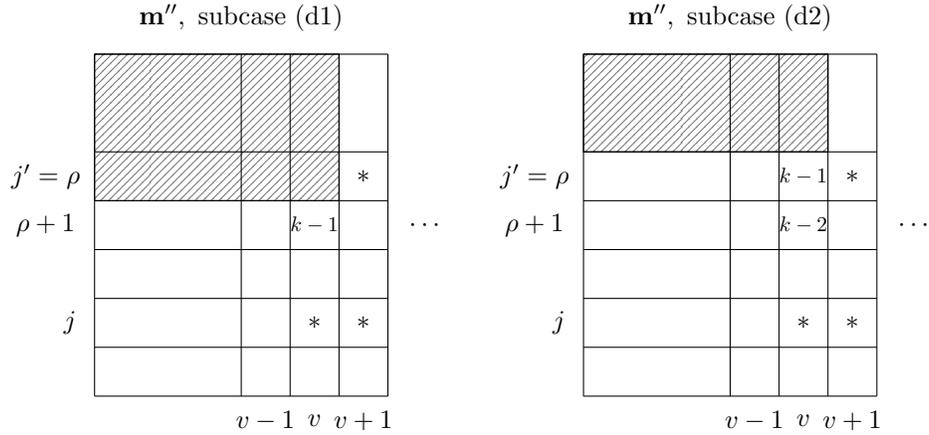
\begin{center}
		\begin{tikzpicture}[scale=0.65]
			\draw[pattern=north east lines, pattern color=gray] (0,4) rectangle (5,7);
			\draw (0,0) -- (6,0);
			\draw (0,1) -- (6,1);
			\draw (0,2) -- (6,2);
			\draw (0,3) -- (6,3);
			\draw (0,4) -- (6,4);
			\draw (0,5) -- (6,5);
			\draw (0,7) -- (6,7);
			\draw (0,0) -- (0,7);
			\draw (3,0) -- (3,7);
			\draw (4,0) -- (4,7);
			\draw (5,0) -- (5,7);
			\draw (6,0) -- (6,7);
			\node (m) at (3,7.75) {\(\mathbf{m}'', \text{ subcase (d1)}\)};
			\node (rho+1) at (-1,3.5) {\(\rho+1\)};
			\node (jerho) at (-1,4.5) {\(j'=\rho\)};
			\node (jprime) at (-0.5,1.5) {\(j\)};
			\node (v-1) at (3.5,-0.5) {\(v-1\)};
			\node (v) at (4.5,-0.5) {\(v\)};
			\node (v+1) at (5.5,-0.5) {\(v+1\)};
			\node (star2) at (4.5,1.5) {\(*\)};
			\node (star3) at (5.5,1.5) {\(*\)};
			\node (star5) at (5.5,4.5) {\(*\)};
			\node[scale=0.8] (k-1) at (4.5,3.5) {\(k-1\)};
			\node (dots) at (6.75,3.5) {\(\cdots\)};
			
			
			\draw[pattern=north east lines, pattern color=gray] (10,5) rectangle (15,7);
			\draw (10,0) -- (16,0);
			\draw (10,1) -- (16,1);
			\draw (10,2) -- (16,2);
			\draw (10,3) -- (16,3);
			\draw (10,4) -- (16,4);
			\draw (10,5) -- (16,5);
			\draw (10,7) -- (16,7);
			\draw (10,0) -- (10,7);
			\draw (13,0) -- (13,7);
			\draw (14,0) -- (14,7);
			\draw (15,0) -- (15,7);
			\draw (16,0) -- (16,7);
			\node (2m) at (13,7.75) {\(\mathbf{m}'', \text{ subcase (d2)}\)};
			\node (2rho+1) at (9,3.5) {\(\rho+1\)};
			\node (2jerho) at (9,4.5) {\(j'=\rho\)};
			\node (2jprime) at (9.5,1.5) {\(j\)};
			\node (2v-1) at (13.5,-0.5) {\(v-1\)};
			\node (2v) at (14.5,-0.5) {\(v\)};
			\node (2v+1) at (15.5,-0.5) {\(v+1\)};
			\node (2star2) at (14.5,1.5) {\(*\)};
			\node (2star3) at (15.5,1.5) {\(*\)};
			\node (2star5) at (15.5,4.5) {\(*\)};
			\node[scale=0.8] (2k-1) at (14.5,4.5) {\(k-1\)};
			\node[scale=0.8] (2k-2) at (14.5,3.5) {\(k-2\)};
			\node (2dots) at (16.75,3.5) {\(\cdots\)};
		\end{tikzpicture}
	\end{center}
	since \( (\rho, v-1) \) had been a non-box (resp. a box) for \(\mathbf{m}\) and \(\mathbf{m}'\). In subcase (d3) we get the pictures 
	
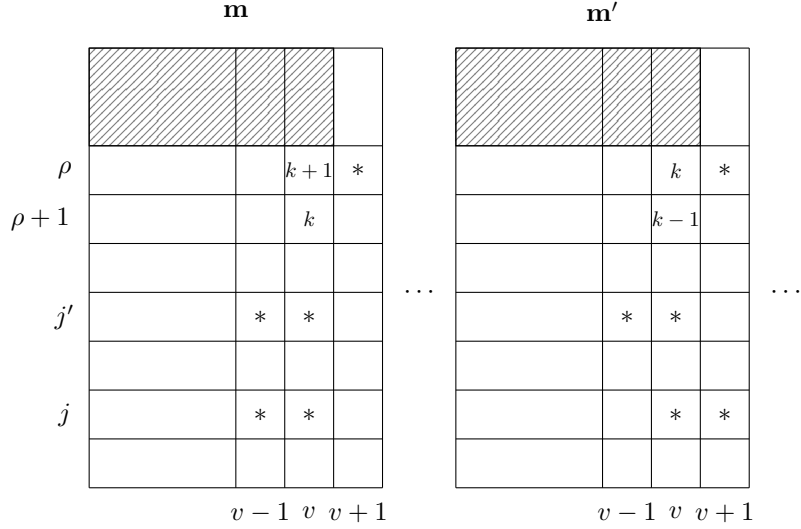
\begin{center}
		\begin{tikzpicture}[scale=0.65]
			\draw[pattern=north east lines, pattern color=gray] (0,7) rectangle (5,9);
			\draw (0,0) -- (6,0);
			\draw (0,1) -- (6,1);
			\draw (0,2) -- (6,2);
			\draw (0,3) -- (6,3);
			\draw (0,4) -- (6,4);
			\draw (0,5) -- (6,5);
			\draw (0,6) -- (6,6);
			\draw (0,7) -- (6,7);
			\draw (0,9) -- (6,9);
			\draw (0,0) -- (0,9);
			\draw (3,0) -- (3,9);
			\draw (4,0) -- (4,9);
			\draw (5,0) -- (5,9);
			\draw (6,0) -- (6,9);
			\node (m) at (3,9.75) {\(\mathbf{m}\)};
			\node (rho+1) at (-1,5.5) {\(\rho+1\)};
			\node (jerho) at (-0.5,6.5) {\(\rho\)};
			\node (j) at (-0.5,1.5) {\(j\)};
			\node (jprime) at (-0.5,3.5) {\(j'\)};
			\node (v-1) at (3.5,-0.5) {\(v-1\)};
			\node (v) at (4.5,-0.5) {\(v\)};
			\node (v+1) at (5.5,-0.5) {\(v+1\)};
			\node (star1) at (3.5,1.5) {\(*\)};
			\node (star2) at (4.5,1.5) {\(*\)};
			\node (star3) at (3.5,3.5) {\(*\)};
			\node (star4) at (4.5,3.5) {\(*\)};
			\node (star5) at (5.5,6.5) {\(*\)};
			\node[scale=0.8] (k) at (4.5,5.5) {\(k\)};
			\node[scale=0.8] (k+1) at (4.5,6.5) {\(k+1\)};
			\node (dots) at (6.75,4) {\(\cdots\)};
			
			\draw[pattern=north east lines, pattern color=gray] (7.5,7) rectangle (12.5,9);
			\draw (7.5,0) -- (13.5,0);
			\draw (7.5,1) -- (13.5,1);
			\draw (7.5,2) -- (13.5,2);
			\draw (7.5,3) -- (13.5,3);
			\draw (7.5,4) -- (13.5,4);
			\draw (7.5,5) -- (13.5,5);
			\draw (7.5,6) -- (13.5,6);
			\draw (7.5,7) -- (13.5,7);
			\draw (7.5,9) -- (13.5,9);
			\draw (7.5,0) -- (7.5,9);
			\draw (10.5,0) -- (10.5,9);
			\draw (11.5,0) -- (11.5,9);
			\draw (12.5,0) -- (12.5,9);
			\draw (13.5,0) -- (13.5,9);
			\node (mprime) at (10.5,9.75) {\(\mathbf{m}'\)};
			\node (2v-1) at (11,-0.5) {\(v-1\)};
			\node (2v) at (12,-0.5) {\(v\)};
			\node (2v+1) at (13,-0.5) {\(v+1\)};
			\node (2star1) at (13,1.5) {\(*\)};
			\node (2star2) at (12,1.5) {\(*\)};
			\node (2star3) at (11,3.5) {\(*\)};
			\node (2star4) at (12,3.5) {\(*\)};
			\node (2star5) at (13,6.5) {\(*\)};
			\node[scale=0.8] (2k-1) at (12,5.5) {\(k-1\)};
			\node[scale=0.8] (2k) at (12,6.5) {\(k\)};
			\node (2dots) at (14.25,4) {\(\cdots\)};
			
		\end{tikzpicture}
	\end{center}
	and
	
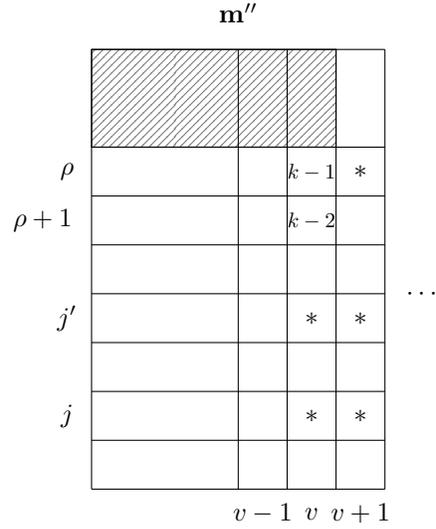
\begin{center}
		\begin{tikzpicture}[scale=0.65]
			\draw[pattern=north east lines, pattern color=gray] (0,7) rectangle (5,9);
			\draw (0,0) -- (6,0);
			\draw (0,1) -- (6,1);
			\draw (0,2) -- (6,2);
			\draw (0,3) -- (6,3);
			\draw (0,4) -- (6,4);
			\draw (0,5) -- (6,5);
			\draw (0,6) -- (6,6);
			\draw (0,7) -- (6,7);
			\draw (0,9) -- (6,9);
			\draw (0,0) -- (0,9);
			\draw (3,0) -- (3,9);
			\draw (4,0) -- (4,9);
			\draw (5,0) -- (5,9);
			\draw (6,0) -- (6,9);
			\node (mprimeprime) at (3,9.75) {\(\mathbf{m}''\)};
			\node (rho+1) at (-1,5.5) {\(\rho+1\)};
			\node (jerho) at (-0.5,6.5) {\(\rho\)};
			\node (j) at (-0.5,1.5) {\(j\)};
			\node (jprime) at (-0.5,3.5) {\(j'\)};			
			\node (3v-1) at (3.5,-0.5) {\(v-1\)};
			\node (3v) at (4.5,-0.5) {\(v\)};
			\node (3v+1) at (5.5,-0.5) {\(v+1\)};
			\node (3star1) at (5.5,1.5) {\(*\)};
			\node (3star2) at (4.5,1.5) {\(*\)};
			\node (3star3) at (5.5,3.5) {\(*\)};
			\node (3star4) at (4.5,3.5) {\(*\)};
			\node (3star5) at (5.5,6.5) {\(*\)};
			\node[scale=0.8] (3k-2) at (4.5,5.5) {\(k-2\)};
			\node[scale=0.8] (3k-1) at (4.5,6.5) {\(k-1\)};
			\node (3dots) at (6.75,4) {\(\cdots\)};
		\end{tikzpicture}
	\end{center}
	where again the admissibilities of \(j, j'\) and \ref{Subsubsection.Consecutive-set} are used. In all three subcases, the \((v+1)\)-column of \(\diag^{(d)}(\mathbf{m}'')\) contains at least 2 boxes, which must include \(B_{\ell}^{(d)}(\mathbf{m}'')\) for
	\(\ell = k\) and \(k+1\). So \(\mathbf{m}'' \in \mathcal{W}(d,k)\) in case (d). The remaining \textbf{Case (e)} comes out by the arguments already used, and is left to the reader.
\end{proof}

\begin{Remarks} \label{Remarks.d-characteristic-sequences-of-sums}
	\begin{enumerate}[wide, label=(\roman*)]
		\item Let \(\mathbf{m}\) and \(\mathbf{m}'\) be as in Lemma \ref{Lemma.Boxes-of-sums-of-certain-vertices} and \(v \defeq v_{k}^{(d)}(\mathbf{m})\). Then \(v_{k}^{(d)}(\mathbf{m}') = v\) and 
		\[
			v_{k+1}^{(d)}(\mathbf{m}') = \begin{cases} v,	&\text{if } \mathbf{m} \in \mathcal{W}(d,k), \\ v+1,	&\text{if } \mathbf{m} \notin \mathcal{W}(d,k). \end{cases}
		\]
		\item Let \(\mathbf{m}\), \(\mathbf{m}'\) and \(\mathbf{m}''\) be as in Lemma \ref{Lemma.Membership-of-sums-of-vertices}, and suppose that \(\mathbf{m}'' \in \mathcal{W}(d,k)\). Then \(v_{k}^{(d)}(\mathbf{m}'') = v_{k}^{(d)}(\mathbf{m}) +1\).
		\item In the situation of Lemma \ref{Lemma.Membership-of-sums-of-vertices} the condition on \(j'\) that \(\mathbf{m}''\) belong to \(\mathcal{W}(d,k)\) does not depend on \(j\), but only on \(\mathbf{m}\).
		Here (i) and (ii) are immediate from the Lemmas, and (iii) comes out as the condition \(v-d < m_{j'}\) that occurs in Lemma \ref{Lemma.Membership-of-sums-of-vertices}(iii) is automatically fulfilled if \(\rho_{k}^{(d)}(\mathbf{m}) \leq j' < j\).
	\end{enumerate}
\end{Remarks}

\begin{Lemma} \label{Lemma.On-ordered-vertices}
	Let \(\mathbf{m}, \mathbf{n}\) be vertices of \(\mathcal{W}(d,k)\) with \(\mathbf{m} < \mathbf{n}\) and \(d(\mathbf{m}, \mathbf{n}) = 1\). Let \(S\) be the set \(S = \{ i \mid 1 \leq i < r \text{ and } m_{i} < n_{i}\}\) of cardinality \(s\). 
	Then there exists a sequence \(\mathbf{m} = \mathbf{m}^{(0)} < \mathbf{m}^{(1)} < \dots < \mathbf{m}^{(s)} = \mathbf{n}\) of elements of \(\mathcal{W}(\mathds{Z})\) such that all the \(\mathbf{m}^{(j)}\) with one possible 
	exception belong to \(\mathcal{W}(d,k)\). The exception holds if and only if \(v \defeq v_{k}^{(d)}(\mathbf{m}) < v_{k}^{(d)}(\mathbf{n})\). In that case, the unique \(\mathbf{m}^{(t)} \notin \mathcal{W}(d,k)\) is the minimal one 
	for which \(v_{k+1}^{(d)}(\mathbf{m}^{(t)}) = v+1\), and
	\begin{equation} \label{Eq.d-characteristic-sequences-of-sequence}
		v_{k}^{(d)}(\mathbf{m}^{(j)}) = \begin{cases} v,	&j \leq t, \\ v+1,	&j>t, \end{cases}, \qquad v_{k+1}^{(d)}(\mathbf{m}^{(j)}) = \begin{cases} v,	&j<t, \\ v+1,	&j \geq t. \end{cases}
	\end{equation}
\end{Lemma}

First note that
\begin{itemize}
	\item \(v_{k}^{(d)}(\mathbf{n}) = v_{k+1}^{(d)}(\mathbf{n}) = v\) or \(v+1\);
	\item for \(i \in S\), \(n_{i} = m_{i} +1\), so \(\mathbf{n} = \mathbf{m} + \sum_{i \in S} \mathbf{e}_{i}\);
	\item \(i \in S\) is admissible for \(\mathbf{m} + \sum_{j \in S, j<i} \mathbf{e}_{j}\).
\end{itemize}
This gives \eqref{Eq.d-characteristic-sequences-of-sequence} as soon as the preceding assertion is shown. For this, we use induction on \(s\). The case \fbox{\(s=1\)} is Remark \ref{Remarks.d-characteristic-sequences-of-sums}(i). Let \fbox{\(s=2\)}, \(S = \{i,j\}\) with \(i<j\), write 
\(\mathbf{m}^{(1)} \defeq \mathbf{m} + \mathbf{e}_{i}\), and suppose \fbox{\(v_{k}^{(d)}(\mathbf{n}) = v\)}. By Lemma \ref{Lemma.d-characteristic-sequences-on-WIZ}, \(v_{k}^{(d)}(\mathbf{m}^{(1)}) = v_{k+1}^{(d)}(\mathbf{m}^{(1)}) = v\), so
\(\mathbf{m}^{(1)} \in \mathcal{W}(d,k)\). If \(v_{k}^{(d)}(\mathbf{n}) = v+1\) then, again by Remark \ref{Remarks.d-characteristic-sequences-of-sums}(i), \(\mathbf{m}^{(1)} \notin \mathcal{W}(d,k)\) and \(v_{k+1}^{(d)}(\mathbf{m}^{(1)}) = v+1\).

Suppose \fbox{\(s > 2\)} and that the assertion holds for \(s' < s\). Let \(j < j' < j'' \dots\) be the elements of \(S\) and \(\mathbf{m}^{(1)} \defeq \mathbf{m} + \mathbf{e}_{j}\), \(\mathbf{m}^{(2)} \defeq \mathbf{m}^{(1)} + \mathbf{e}_{j'}\) \dots
We are done if \(\mathbf{m}^{(1)}\) or \(\mathbf{m}^{(2)} \in \mathcal{W}(d,k)\), applying the induction hypothesis. Assume that neither of \(\mathbf{m}^{(1)}\), \(\mathbf{m}^{(2)}\) belongs to \(\mathcal{W}(d,k)\). By Corollary \ref{Corollary.Condition-for-k-capped}, 
\(j \geq \rho \defeq \rho_{k}^{(d)}(\mathbf{m})\), and by Lemma \ref{Lemma.Membership-of-sums-of-vertices}(iii), \(m_{j'} \leq v-d\), which implies 
\begin{equation} \label{Eq.Condition-for-mi}
	m_{i} \leq v-d \text{ for } i \geq j', \text{ in particular for } i=j'',j''', \dots
\end{equation}
By Remark \ref{Remarks.d-characteristic-sequences-of-sums}(i),
\[
	v(B_{k}^{(d)}(\mathbf{m}^{(1)})) = v, \qquad v(B_{k+1}^{(d)}(\mathbf{m}^{(1)})) = v+1,
\]
that is, the number \(k\), but not \(k+1\) occurs as a box number in the \(v\)-column of \(\diag^{(d)}(\mathbf{m}^{(1)})\). As in the case for \(i = j'\), successively adding \(\mathbf{e}_{i}\) (\(i = j'', j''', \dots\)), i.e., replacing 
\(\mathbf{m}^{(1)}\) with \(\mathbf{m}^{(2)}\), \(\mathbf{m}^{(3)}\), \dots doesn't change the box numbers in the \(v\)-th and \((v+1)\)-th column of \(\diag^{(d)}\) due to \eqref{Eq.Condition-for-mi}. Hence finally
\[
	v(B_{k}^{(d)}(\mathbf{m}^{(s)})) \neq v(B_{k+1}^{(d)}(\mathbf{m}^{(s)}))
\]
with \(\mathbf{m}^{(s)} = \mathbf{n}\), which contradicts \(\mathbf{n} \in \mathcal{W}(d,k)\). Therefore, \(\mathbf{m}^{(1)}\) or \(\mathbf{m}^{(2)} \in \mathcal{W}(d,k)\) as wanted, and the proof is complete. ~\hfill \mbox{\(\square\)}

\begin{Lemma} \label{Lemma.On-vertices-and-simplices}
	Let \(\sigma\) be a simplex in \(\mathcal{W}(d,k)\). Put \(\mathbf{m} \defeq \min \sigma\), \(\mathbf{n} \defeq \max \sigma\), \(v \defeq v_{k}^{(d)}(\mathbf{n})\), \(\rho \defeq \rho_{k}^{(d)}(\mathbf{n})\). Suppose that either 
	\begin{enumerate}[label=\(\mathrm{(\alph*)}\)]
		\item there is no simplex \(\sigma' \subset \mathcal{W}(d,k)\) with \(\sigma \subset \sigma'\), \(\min \sigma' = \mathbf{m}\), and such that \(\mathbf{n} < \mathbf{n}' \defeq \max \sigma\), or
		\item there is no simplex \(\sigma' \subset \mathcal{A}(d,k)\) with \(\sigma \subset \sigma'\), \(\max \sigma' = \mathbf{n}\), and such that \(\mathbf{m}' \defeq \min \sigma' < \mathbf{m}\), and moreover, \(v < d\)
	\end{enumerate}
	holds. Then either \(\mathbf{n} = \mathbf{m} + \mathbf{y}\) or \(\mathbf{n} = \mathbf{m} + \mathbf{y} - \mathbf{e}_{i}\) for some \(i\) with \(\rho \leq i < r\). In the latter case, the function \(v_{k}^{(d)}\) is constant on
	\(\sigma\), \(\rho_{k}^{(d)}(\mathbf{m}) = r-1\), and \(\mathbf{n}\) is \(k\)-capped.
\end{Lemma}

\begin{Remark}
	Although the conclusion is the same for cases (a) and (b) of the Lemma, the requirements are different and not perfectly dual. In (b) we admit \(\sigma' \subset \mathcal{A}(d,k)\), and we have the supplementary requirement
	\(v < d\). Correspondingly, the proofs in cases (a), (b) are similar, but not dual to each other. While we manage with the preceding lemmas in case (a), we have to improvise and argue directly with \(d\)-diagrams in case (b).
\end{Remark}

\begin{proof}[Proof of Lemma \ref{Lemma.On-vertices-and-simplices}, Case (a)]
	\begin{enumerate}[wide, label=(\roman*)]
		\item We have \(\mathbf{m} \leq \mathbf{n} \leq \mathbf{m} + \mathbf{y}\). Suppose that 
		\begin{equation} \label{Eq.Assumption-on-order}
			\mathbf{n} < \mathbf{m} + \mathbf{y}.
		\end{equation}
		By the maximality property of \(\sigma\) and Proposition \ref{Proposition.Sums-of-vertices}(ii),
		\begin{equation}
			\rho_{k}^{(d)}(\mathbf{m}) = r-1.
		\end{equation}
		\item For each vertex \(\mathbf{a}\) with \(\mathbf{m} \leq \mathbf{a} \leq \mathbf{n}\),
		\begin{equation} \label{Eq.d-characteristic-sequences-of-sandwich}
			v_{k}^{(d)}(\mathbf{a}) = v_{k+1}^{(d)}(\mathbf{a}) = v;
		\end{equation}
		otherwise by Lemma \ref{Lemma.d-characteristic-sequences-on-WIZ}, \(v_{k}^{(d)}(\mathbf{m}) + 1 = v_{k+1}^{(d)}(\mathbf{m}) + 1 = v = v_{k}^{(d)}(\mathbf{m} + \mathbf{y}) = v_{k+1}^{(d)}(\mathbf{m} + \mathbf{y})\), which would imply 
		\(\mathbf{m} + \mathbf{y} \in \mathcal{W}(d,k)\).
		\item We show that
		\begin{equation} \label{Eq.Condition-for-nj-and-mj}
			n_{j} = m_{j}+1 \quad \text{for } 1 \leq j < \rho.
		\end{equation}
		For, take the least \(j < \rho\) that violates \eqref{Eq.Condition-for-nj-and-mj}. Then \(n_{j} = m_{j}\) and \(j\) is admissible for \(\mathbf{n}\) (for \(j > 1\), \(n_{j} = m_{j} \leq m_{j-1} = n_{j-1} -1\)), so \(\mathbf{n}' \defeq \mathbf{n} + \mathbf{e}_{j} \in \mathcal{W}\) and
		\(\mathbf{n}' \leq \mathbf{m} + \mathbf{y}\). By Lemma \ref{Lemma.Boxes-of-sums-of-certain-vertices}, \(\mathbf{n}'\) even belongs to \(\mathcal{W}(d,k)\), which conflicts with the maximality property of \(\sigma\).
		\item Similarly, if \(n_{\rho} = m_{\rho}\) then \(\rho\) is admissible for \(\mathbf{n}\), and Lemma \ref{Lemma.Boxes-of-sums-of-certain-vertices}(iii) gives that \(\mathbf{n}\) is \(k\)-capped. Hence
		\begin{equation} \label{Eq.Either}
			\text{\textbf{either} } n_{\rho} = m_{\rho} \text{ and } \mathbf{n} \text{ is \(k\)-capped \textbf{or} } n_{\rho} = m_{\rho} +1.
		\end{equation}
		\item Consider the sequence
		\[
			\mathbf{m} = \mathbf{m}^{(0)} < \mathbf{m}^{(1)} < \dots < \mathbf{m}^{(\rho-1)} \eqdef \mathbf{a} \leq \mathbf{n} < \mathbf{m} + \mathbf{y},
		\]
		where \(\mathbf{m}^{(j)} \defeq \mathbf{m}^{(j-1)} + \mathbf{e}_{j}\) for \( 1 \leq j < \rho\). Lemma \ref{Lemma.Boxes-of-sums-of-certain-vertices} guarantees that all the \(\mathbf{m}^{(j)}\) are in \(\mathcal{W}(d,k)\), 
		where \(v_{k}^{(d)}\) and \(\rho_{k}^{(d)}\) are constant along \(\mathbf{m}^{(0)}, \dots, \mathbf{m}^{(\rho-1)}\). Therefore,
		\begin{equation} \label{Eq.rho-of-sandwich}
			\rho_{k}^{(d)}(\mathbf{a}) = \rho_{k}^{(d)}(\mathbf{m}) = r-1.
		\end{equation}
		\item Let \(S = \{ j_{1} < j_{2} < \dots < j_{s}\}\) be the set of those \(j\) (\(\rho \leq j < r\)) such that \(n_{j} = a_{j} + 1\) (note that \(a_{j} = m_{j}\) for these \(j\)). Consider the analogous sequence
		\( \mathbf{a} = \mathbf{a}^{(0)} < \mathbf{a}^{(1)} < \dots < \mathbf{a}^{(s)} = \mathbf{n}\) with \(\mathbf{a}^{(t)} \defeq \mathbf{a}^{(t-1)} + \mathbf{e}_{j_{t}}\) (\(1 \leq t \leq s\)). From \eqref{Eq.Assumption-on-order}, \(s\) 
		must be less than \( (r-1) - (\rho -1)\), i.e., \(s \leq r-\rho-1\). Again, \(j_{t}\) is admissible for \(\mathbf{a}^{(t-1)}\), so all the \(\mathbf{a}^{(t)}\) belong to \(\mathcal{W}\), and by \eqref{Eq.d-characteristic-sequences-of-sandwich} 
		even to \(\mathcal{W}(d,k)\). By Lemma \ref{Lemma.Boxes-of-sums-of-certain-vertices}, the function \(\rho_{k}^{(d)}\) decreases at most by 1 in each step \(\mathbf{a}^{(t-1)} < \mathbf{a}^{(t)}\), which gives 
		\(s \geq \rho_{k}^{(d)}(\mathbf{a}) - \rho_{k}^{(d)}(\mathbf{n}) = r-\rho-1\). Thus
		\begin{equation} \label{Eq.s-equals-r-rho-1}
			s = r-\rho-1.
		\end{equation}
		\item Combining \eqref{Eq.Condition-for-nj-and-mj} and \eqref{Eq.s-equals-r-rho-1} we find \(n_{j} = m_{j} + 1\) for all \(j\) (\(1 \leq j < r\)) with one exception \(i\), for which \(\rho \leq i < r\) holds. This gives the largest part of the assertion.
		\item Suppose \(\mathbf{n}\) is not \(k\)-capped. Then by \eqref{Eq.Either}, \(n_{\rho} = m_{\rho}+1\). Let \(i > \rho\) be the exceptional index with \(n_{i} = m_{i}\). It is admissible for \(\mathbf{n}\) and 
		\(\mathbf{n}' = \mathbf{n} + \mathbf{e}_{i} \in \mathcal{W}(d,k)\) by Lemma \ref{Lemma.Boxes-of-sums-of-certain-vertices}, which however conflicts with the maximality assumption on \(\mathbf{n}\). Hence \(\mathbf{n}\) is \(k\)-capped, and all claimed properties are shown.
	\end{enumerate}
\end{proof}

\begin{proof}[Proof of Lemma \ref{Lemma.On-vertices-and-simplices}, Case (b)]
	\begin{enumerate}[wide, label=(\roman*)]
		\item We have \(\mathbf{n} - \mathbf{y} \leq \mathbf{m} \leq \mathbf{y}\). Suppose that
		\begin{equation} \label{Eq.Assumption-on-order-2}
			\mathbf{n} - \mathbf{y} < \mathbf{m}.
		\end{equation}
		By the maximality property of \(\sigma\), \(\mathbf{n} - \mathbf{y} \notin \mathcal{A}(d,k)\), so by Corollary \ref{Corollary.Membership-of-AIZ}(ii) and the condition \(v < d\), \(\mathbf{n} \notin \mathcal{W}(d,k+1)\), that is, \(\mathbf{n}\) is \(k\)-capped.
		\item For each vertex \(\mathbf{a}\) of \(\mathcal{A}\) with \(\mathbf{m} \leq \mathbf{a} \leq \mathbf{n}\),
		\begin{equation} \label{Eq.d-characteristic-sequence-of-sandwich-2}
			v_{k}^{(d)}(\mathbf{a}) = v_{k+1}^{(d)}(\mathbf{a}) = v,
		\end{equation}
		since otherwise (by Lemma \ref{Lemma.d-characteristic-sequences-on-WIZ} and \ref{Subsubsection.Validity-of-lemma-for-AIZ}), \(v-1 = v_{k}^{(d)}(\mathbf{m}) = v_{k+1}^{(d)}(\mathbf{m}) = v_{k}^{(d)}(\mathbf{n} - \mathbf{y}) = v_{k+1}^{(d)}(\mathbf{n} - \mathbf{y})\), which would imply 
		\(\mathbf{n} - \mathbf{y} \in \mathcal{A}(d,k)\).
		\item For \(1 \leq j < \rho\),
		\begin{equation} \label{Eq.mj-equals-nj-1}
			m_{j} = n_{j} - 1
		\end{equation}
		holds. Viz, suppose to the contrary that \(m_{j} = n_{j}\) for some \(j\). As \(\mathbf{n}\) is \(k\)-capped, \(n_{j} > v\). Then \(B_{\ell}^{(d)}\) isn't changed for \(\ell=k,k+1\) under 
		\(\mathbf{m} \leadsto \mathbf{m}' = \mathbf{m} - \mathbf{e}_{j}\), as a view to the picture for \(\diag^{(d)}(\mathbf{m})\) shows:
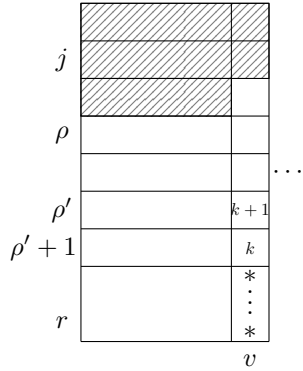
		\begin{center}
			\begin{tikzpicture}[scale=0.5]
				\draw[pattern=north east lines, pattern color=gray] (0,6) rectangle (4,7);
				\draw[pattern=north east lines, pattern color=gray] (0,7) rectangle (5,9);
				\draw (0,0) -- (5,0);
				\draw (0,2) -- (5,2);
				\draw (0,3) -- (5,3);
				\draw (0,4) -- (5,4);
				\draw (0,5) -- (5,5);
				\draw (0,6) -- (5,6);
				\draw (0,7) -- (5,7);
				\draw (0,8) -- (5,8);
				\draw (0,9) -- (5,9);
				\draw (0,0) -- (0,9);
				\draw (4,0) -- (4,9);
				\draw (5,0) -- (5,9);
				\node (rr) at (-0.5,0.5) {\(r\)};
				\node (rrho+1) at (-1,2.5) {\(\rho'+1\)};
				\node (rrhoprime) at (-0.5,3.5) {\(\rho'\)};
				\node (rho) at (-0.5,5.5) {\(\rho\)};
				\node (rj) at (-0.5,7.5) {\(j\)};
				\node (cv) at (4.5,-0.5) {\(v\)};
				\node[scale=0.65] (k) at (4.5, 2.5) {\(k\)};
				\node[scale=0.65] (k+1) at (4.5,3.5) {\(k+1\)};
				\node (star1) at (4.5,0.25) {\(*\)};
				\node (star2) at (4.5, 1.75) {\(*\)};
				\node (dots) at (4.5, 1.25) {\(\vdots\)};
				\node (cdots) at (5.5,4.5) {\(\cdots\)};
				\node (label) at (11,0) {\(\text{Here, } \rho' \defeq \rho_{k}^{(d)}(\mathbf{m}) \geq \rho.\)};
			\end{tikzpicture}
		\end{center}
		Therefore, \(\mathbf{m}' \in \mathcal{A}(d,k)\), which is impossible as \(\mathbf{y} - \mathbf{n} \leq \mathbf{m}' < \mathbf{m}\).
		\item Suppose that \(\rho' \defeq \rho_{k}^{(d)}(\mathbf{m}) < r-1\). Then even
		\begin{equation} \label{Eq.mj-equals-nj-1-2}
			m_{j} = n_{j} - 1 \quad \text{for all} \quad j < \rho'
		\end{equation}
		holds. Otherwise, let \(j\) satisfy \(\rho \leq j < \rho'\) and \(m_{j} = n_{j}\), and put as usual \(\mathbf{m}' \defeq \mathbf{m} - \mathbf{e}_{j}\). Then the box numbers in the \(v\)-column of \(\diag^{(d)}\) increase by 1 upon 
		\(\mathbf{m} \leadsto \mathbf{m}'\). As there are at least 3 such boxes with indices larger or equal to \(\rho'\) in \(\diag^{(d)}(\mathbf{m})\), we see that \(\mathbf{m}' \in \mathcal{A}(d,k)\), which is forbidden.
		\item Still under the assumption \(\rho' < r-1\), we find
		\begin{equation} \label{Eq.mj-equals-nj-1-3}
			m_{j} = n_{j} -1 \quad \text{for} \quad j \geq \rho'.
		\end{equation}
		For, assume \(m_{j} = n_{j}\) for \(j\) with \(\rho \leq j < r\). Then the set of box numbers in the \(v\)-column of \(\diag^{(d)}(\mathbf{m})\) includes \(k+1,k, \dots, k+\rho'-r+1\). Upon 
		\(\mathbf{m} \leadsto \mathbf{m}' \defeq \mathbf{m} - \mathbf{e}_{j}\),
		\begin{description}
			\item[either] these boxes remain boxes of \(\diag^{(d)}(\mathbf{m}')\), with numbers increased by 1,
			\item[or] (if \(m_{j} = n_{j} = 0\) and \(v = d-1\)), the box \( (j,d-1)\) becomes a non-box in \(\diag^{(d)}(\mathbf{m}')\), in which case the set of box numbers in the \(v\)-th column of \(\diag^{(d)}(\mathbf{m}')\) includes at least
			\(k+1, \dots, k+\rho'-r+2\).
		\end{description}
		In both cases, \(v(B_{k}^{(d)}(\mathbf{m}')) = v(B_{k+1}^{(d)}(\mathbf{m}')) = v\), so \(\mathbf{m}' \in \mathcal{A}(d,k)\), which cannot happen. Therefore, \(m_{j} = n_{j}-1\).
		\item Combining \eqref{Eq.mj-equals-nj-1-2} and \eqref{Eq.mj-equals-nj-1-3} we find that under the assumption \(\rho' < r-1\), \(m_{j} = n_{j} - 1\) for all \(j\) with \(1 \leq j < r\). As this contradicts \eqref{Eq.Assumption-on-order-2}, 
		we find
		\begin{equation}
			\rho' = \rho_{k}^{(d)}(\mathbf{m}) = r-1.
		\end{equation}
		\item Define the sequence \(\mathbf{n} = \mathbf{n}^{(0)} > \mathbf{n}^{(1)} > \dots > \mathbf{n}^{(\rho-1)} \eqdef \mathbf{a} \geq \mathbf{m}\) by 
		\[
			\mathbf{n}^{(j)} \defeq \mathbf{n}^{(j-1)} - \mathbf{e}_{\rho -j} \qquad (1 \leq j < \rho).
		\]
		By \eqref{Eq.d-characteristic-sequence-of-sandwich-2}, \(\mathbf{n}^{(j)} \in \mathcal{A}(d,k)\) for each \(j\), and the definition is such that \(\mathbf{n}^{(j)} \in \mathcal{W}\), so \(\mathbf{n}^{(j)} \in \mathcal{W}(d,k)\). 
		Another view to \(\diag^{(d)}(\mathbf{n})\)
		\begin{center}
			\begin{tikzpicture}[scale=0.65]
				\draw[pattern=north east lines, pattern color=gray] (0,4) rectangle (5,6);
				\draw (0,0) -- (5,0);
				\draw (0,2) -- (5,2);
				\draw (0,3) -- (5,3);
				\draw (0,4) -- (5,4);
				\draw (0,6) -- (5,6);
				\draw (0,0) -- (0,6);
				\draw (4,0) -- (4,6);
				\draw (5,0) -- (5,6);
				\node (rrho+1) at (-1,2.5) {\(\rho+1\)};
				\node (rrho) at (-0.5,3.5) {\(\rho\)};
				\node (cv) at (4.5,-0.5) {\(v\)};
				\node[scale=0.65] (k) at (4.5, 2.5) {\(k\)};
				\node[scale=0.65] (k+1) at (4.5,3.5) {\(k+1\)};
				\node (star1) at (4.5,0.25) {\(*\)};
				\node (star2) at (4.5, 1.75) {\(*\)};
				\node (dots) at (4.5, 1.25) {\(\vdots\)};
			\end{tikzpicture}
		\end{center}
		shows that \(B_{\ell}^{(d)}\) for \(\ell = k,k+1\) remains unchanged in \(\mathbf{n}^{(0)} > \dots > \mathbf{a}\). Hence
		\begin{equation}
			\rho_{k}^{(d)}(\mathbf{a}) = \rho_{k}^{(d)}(\mathbf{n}) = \rho.
		\end{equation}
		\item Let \(S = \{ j_{1} < j_{2} < \dots < j_{s} \}\) be the set of those \(j\) (\(\rho \leq j < r\)) for which \(m_{j} = n_{j} - 1\). Consider 
		\[
			\mathbf{a} = \mathbf{a}^{(0)} > \mathbf{a}^{(1)} > \dots > \mathbf{a}^{(s)} = \mathbf{m}
		\]
		with \(\mathbf{a}^{(t)} \defeq \mathbf{a}^{(t-1)} - \mathbf{e}_{j_{s+1-t}}\) (\(1 \leq t \leq s\)). Again, all the \(\mathbf{a}^{(t)}\) belong to \(\mathcal{W}(d,k)\), and by \eqref{Eq.Assumption-on-order-2}, \(s < (r-1)-(\rho-1)\), i.e., \(s \leq r-\rho-1\). Similar to 
		case (a), in each
		step \( \mathbf{a}^{(t-1)} > \mathbf{a}^{(t)}\), the quantity \(\rho_{k}^{(d)}\) may increase at most by 1, which gives \(s \geq \rho_{k}^{(d)}(\mathbf{m}) - \rho_{k}^{(d)}(\mathbf{a}) = r-\rho-1\). Summarizing,
		\begin{equation}
			s = r-\rho-1,
		\end{equation}
		which says that for all indices \(j\) with \(1 \leq j < r\) with one exception \(i\), \(m_{j} = n_{j}-1\) holds, and the exception \(i\) satisfies \(\rho \leq i < r\). That is, \(\mathbf{m} = \mathbf{n} - \mathbf{y} + \mathbf{e}_{i}\), which finishes
		the proof of Lemma \ref{Lemma.On-vertices-and-simplices}.
	\end{enumerate}
\end{proof}

\subsection{} The preceding lemmas suggest the following notation. For vertices \(\mathbf{m}\), \(\mathbf{n}\) of \(\mathcal{A}\), write 
\subsubsection{} \(\mathbf{m} \triangleleft \mathbf{n}\) if \(\mathbf{m} < \mathbf{n}\), but there exists no vertex \(\mathbf{a}\) of \(\mathcal{A}\) with \(\mathbf{m} < \mathbf{a} < \mathbf{n}\).\stepcounter{equation}%

This is equivalent with \(\mathbf{n} = \mathbf{m} + \mathbf{e}_{i}\) for some \(i\), \(1 \leq i < r\). Then the sequence of Lemma \ref{Lemma.On-ordered-vertices} may be written
\begin{equation}
	\mathbf{m} = \mathbf{m}^{(0)} \triangleleft \mathbf{m}^{(1)} \triangleleft \dots \triangleleft \mathbf{m}^{(s)} = \mathbf{n},
\end{equation}
and the occurring differences \(\mathbf{m}^{(t)} - \mathbf{m}^{(t-1)}\) are precisely the vectors \(\mathbf{e}_{j}\) with \(j \in S\), ordered increasingly.

The hardest part of the labor being carried out, we now easily harvest the fruit.

\begin{Theorem} \label{Theorem.Simplicial-complexes-strongly-equidimensional}
	The simplicial complexes \(\mathcal{W}(k)\), \(\mathcal{W}(d,k)\), \(\mathcal{A}(k)\), \(\mathcal{A}(d,k)\), \(\mathcal{BT}(k)\), \(\mathcal{BT}(d,k)\) are strongly equidimensional of dimension \(r-2\).
\end{Theorem}

\begin{proof}
	\begin{enumerate}[wide, label=(\roman*)]
		\item As written in the header of this section, the result for \(\mathcal{W}(d,k)\) implies the result for all the other complexes, as \(\mathcal{W}(d,k) = \mathcal{W}(k)\) for \(d \geq k\), \(\mathcal{A}(d,k) = W \mathcal{W}(d,k)\), 
		\(\mathcal{BT}(d,k) = \Gamma \mathcal{W}(d,k)\), etc.
		We are thus reduced to show that each simplex of \(\mathcal{W}(d,k)\) is contained in a simplex of dimension \(r-2\) of \(\mathcal{W}(d,k)\).
		\item Let \(\sigma\) be a maximal simplex in \(\mathcal{W}(d,k)\). We must show that it has dimension \(r-2\). Put \(\mathbf{m} \defeq \min \sigma\), \(\mathbf{n} \defeq \max \sigma\). By Lemma \ref{Lemma.On-vertices-and-simplices}, we have either
		(a) \(\mathbf{n} = \mathbf{m} + \mathbf{y}\) or (b) \(\mathbf{n} = \mathbf{m} + \mathbf{y} - \mathbf{e}_{i}\) for some \(i\).
		\item Write \(\sigma = \{ \mathbf{m}^{(0)}, \dots, \mathbf{m}^{(c)} \}\) with \(\mathbf{m} = \mathbf{m}^{(0)} < \mathbf{m}^{(1)} < \dots < \mathbf{m}^{(c)} = \mathbf{n}\), where \(c = \dim \sigma\). By Lemma \ref{Lemma.On-ordered-vertices}, and since
		\(\sigma\) is maximal, each of the steps \(\mathbf{m}^{(j-1)} < \mathbf{m}^{(j)}\) is either of shape (1) \(\mathbf{m}^{(j-1)} \triangleleft \mathbf{m}^{(j)}\) or (2) \(\mathbf{m}^{(j-1)} \triangleleft \mathbf{a} \triangleleft \mathbf{m}^{(j)}\)
		with some vertex \(\mathbf{a}\) of \(\mathcal{W}\) not in \(\mathcal{W}(d,k)\). Furthermore, shape (2) occurs precisely once in case (a), and never in case (b). Therefore, \(c = \dim \sigma = r-2\) in both cases. 
	\end{enumerate}
\end{proof}

\section{Boundarylessness} \label{Section.Boundarylessness}

\subsection{} In this section, we show that \(\mathcal{A}(d,k)\) is boundaryless in the sense of \ref{Subsection.Convention-for-simplicial-complexes}. This immediately gives the same property for the complexes \(\mathcal{A}(k)\), \(\mathcal{BT}(k)\) and \(\mathcal{BT}(d,k)\). Note 
that we don't prove the assertion for \(\mathcal{W}(k)\) and \(\mathcal{W}(d,k)\), as it fails at the boundary \(\bigcup_{1 \leq j < r} \mathcal{W}_{j}\) of \(\mathcal{W}\). Now the result we are heading to is easily shown by ad hoc means if
\(d=1\) or \(k=1\). Namely, if \(d=1\) then \(\mathcal{W}(d,k) = \mathcal{W}(1,k) = \mathcal{W}_{r-k}\) is one of the walls of \(\mathcal{W}\). Similarly, if \(k=1\) then \(\mathcal{W}(d,1) = \mathcal{W}(1) = \mathcal{W}_{r-1}\). The crucial 
point in the proof given is to show that certain vertices \(\mathbf{n}'\) actually belong to \(\mathcal{A}(d,k) = W\mathcal{W}(d,k)\). This is easy in the described cases, since 
\(\mathcal{W}_{i} = \{ \mathbf{x} \in \mathcal{W} \mid x_{i} = x_{i+1}\}\). 

On the other hand, we make heavy use of preceding results that in general require \(d\) and \(k\) to be at least 2. Therefore we keep our assumptions accumulated so far:
\subsubsection{} The numbers \(r\), \(d\) and \(k\) satisfy \(r \geq 3\), \(d\geq 2\), \(2 \leq k < rd\). \stepcounter{equation}%

The proof scheme we give in this framework also applies to \(d\) or \(k = 1\), but some of the references to earlier results would have to be replaced by easy ad hoc arguments. The next lemma motivates the proof of Theorem \ref{Theorem.Simplicial-complexes-boundaryless}.

\begin{Lemma} \label{Lemma.Maximal-simplex-of-A}
	Let \(\tau\) be a maximal simplex of \(\mathcal{A}\). With \(\mathbf{m} \defeq \min \tau\), write \(\tau = \{ \mathbf{m}^{(j)} \mid 0 \leq j < r \}\), where
	\(\mathbf{m} = \mathbf{m}^{(0)} \triangleleft \mathbf{m}^{(1)} \triangleleft \dots \triangleleft \mathbf{m}^{(r-1)} = \mathbf{m} + \mathbf{y}\).
	For a given \(\mathbf{n} \in \tau\), let \(\sigma\) be the simplex \(\tau \smallsetminus \{ \mathbf{n}\}\) and \(\mathbf{n}' \defeq s_{\sigma}(\mathbf{n})\) (see \eqref{Eq.Reflection-at-<-sigma->}). Then
	\[
		\mathbf{n}' = \begin{cases} \mathbf{m}^{(r)} \defeq \mathbf{m}^{(1)} + \mathbf{y},	&\text{if } \mathbf{n} = \mathbf{m}^{(0)} = \mathbf{m}; \\ \mathbf{m}^{(-1)} \defeq \mathbf{m}^{(r-2)} - \mathbf{y},	&\text{if } \mathbf{n} = \mathbf{m}^{(r-1)}, \\ \mathbf{m}^{(j-1)} + \mathbf{m}^{(j+1)} - \mathbf{m}^{(j)},	&\text{if } \mathbf{n} = \mathbf{m}^{(j)} \text{ with } 1 \leq j \leq r-2. \end{cases}
	\]
\end{Lemma}

\begin{proof}
	It suffices to see that in each case \(\tau' \defeq \sigma \cup \{ \mathbf{n}' \}\) with the stated \(\mathbf{n}'\) is a simplex. This is obvious in the first two cases, see \ref{Subsubsection.Vertices-are-neighbors}. For the last one, let \(w\)
	be the permutation of \( \{1,2,\dots, r-1\}\) such that for each \(j \geq 1\), \(\mathbf{m}^{(j-1)} + \mathbf{e}_{w(j)} = \mathbf{m}^{(j)}\). Then the stated \(\mathbf{n}' = \mathbf{m}^{(j-1)} + \mathbf{m}^{(j+1)} - \mathbf{m}^{(j)}\) equals
	\(\mathbf{m}^{(j-1)} + \mathbf{e}_{w(j+1)}\) and \(\mathbf{n}' + \mathbf{e}_{w(j)} = \mathbf{m}^{(j+1)}\). That is, \(\mathbf{m}^{(j-1)} \triangleleft \mathbf{n}' \triangleleft \mathbf{m}^{(j+1)}\) with the order reversed in which 
	\(\mathbf{e}_{w(j)}\) and \(\mathbf{e}_{w(j+1)}\) are added. This shows that \(\tau' = \tau \smallsetminus \{\mathbf{n}\} \cup \{ \mathbf{n}'\}\) is a simplex.
\end{proof}

\begin{Lemma} \label{Lemma.Vertex-sandwich}
	Let \(\mathbf{a}\), \(\mathbf{b}\), \(\mathbf{n}\), \(\mathbf{h}\) be vertices of \(\mathcal{W}\), where \(\mathbf{a}, \mathbf{b}, \mathbf{n} \in \mathcal{W}(d,k)\) and \(\mathbf{h} \notin \mathcal{W}(d,k)\). Assume that
	\begin{enumerate}[label=\(\mathrm{(\roman*)}\)]
		\item \( \mathbf{a} \triangleleft \mathbf{n} \triangleleft \mathbf{h} \triangleleft \mathbf{b}\) or
		\item \( \mathbf{a} \triangleleft \mathbf{h} \triangleleft \mathbf{n} \triangleleft \mathbf{b}\)
	\end{enumerate}
	holds. Then there exists \(\mathbf{n}' \in \mathcal{A}(d,k)\), \(\mathbf{n}' \neq \mathbf{n}\), such that \(\mathbf{a} < \mathbf{n}' < \mathbf{b}\) is satisfied.
\end{Lemma}

\begin{proof}
	\begin{enumerate}[wide, label=(\roman*)]
		\item Let \(j,j',j''\) be the indices such that \(\mathbf{n} = \mathbf{a} + \mathbf{e}_{j}\), \(\mathbf{h} = \mathbf{n} + \mathbf{e}_{j'}\), \(\mathbf{b} = \mathbf{h} + e_{j''}\). Put \(v \defeq v_{k}^{(d)}(\mathbf{n})\),
		\(\rho \defeq \rho_{k}^{(d)}(\mathbf{n})\). By Lemma \ref{Lemma.Boxes-of-sums-of-certain-vertices} and Lemma \ref{Lemma.Membership-of-sums-of-vertices}, \(\mathbf{n}\) is \(k\)-capped, \(j' \geq \rho\), and \(v-d < n_{j''}\). Again from Lemma \ref{Lemma.Boxes-of-sums-of-certain-vertices}, \(v_{k}^{(d)}(\mathbf{a}) = v\) and \(\rho_{k}^{(d)}(\mathbf{a}) = \rho\)
		or \(\rho + 1\). Now \(j''\) might fail to be admissible for \(\mathbf{a}\), that is, \(\mathbf{n}' \defeq \mathbf{a} + \mathbf{e}_{j''}\) possibly doesn't belong to \(\mathcal{W}\). Therefore we cannot directly apply Lemma \ref{Lemma.Boxes-of-sums-of-certain-vertices}; nevertheless
		its proof also works to show that \(\mathbf{n}' \in \mathcal{A}(d,k)\) under our conditions.
		\item Let \(j,j',j''\) be such that \(\mathbf{h} = \mathbf{a} + \mathbf{e}_{j}\), \(\mathbf{n} = \mathbf{h} + \mathbf{e}_{j'}\), \(\mathbf{b} = \mathbf{n} + \mathbf{e}_{j''}\). Put \(v \defeq v_{k}^{(d)}(\mathbf{a})\),
		\(\rho \defeq \rho_{k}^{(d)}(\mathbf{a})\). Again, \(\mathbf{a}\) is \(k\)-capped and \(v-d < a_{j'}\) by Lemma \ref{Lemma.Membership-of-sums-of-vertices}. If \(\mathbf{a} + \mathbf{e}_{j''} \in \mathcal{A}(d,k)\) then \(\mathbf{n}' \defeq \mathbf{a} + \mathbf{e}_{j''}\) is
		as wanted. Otherwise, the conditions on \(j'\) are such that the proof of Lemma \ref{Lemma.Membership-of-sums-of-vertices} applies to show that \(\mathbf{n}' \defeq \mathbf{a} + \mathbf{e}_{j''} + \mathbf{e}_{j'} \in \mathcal{A}(d,k)\).
	\end{enumerate}
\end{proof}

We are now prepared to show the result of this section.

\begin{Theorem} \label{Theorem.Simplicial-complexes-boundaryless}
	The simplicial complexes \(\mathcal{A}(k)\), \(\mathcal{A}(d,k)\), \(\mathcal{BT}(k)\), \(\mathcal{BT}(d,k)\) are boundaryless.
\end{Theorem}

\begin{proof}
	\begin{enumerate}[wide, label=(\roman*)]
		\item As already mentioned, it suffices to show the assertion for \(\mathcal{A}(d,k)\). Thus let the maximal simplex \(\tau\) of \(\mathcal{A}(d,k)\) be given, \(\mathbf{n} \in \tau\), and 
		\(\sigma \defeq \tau \smallsetminus \{ \mathbf{n}\}\). We must show that there exists a vertex \(\mathbf{n}' \in \mathcal{A}(d,k)(\mathds{Z}) \smallsetminus \tau\) such that \(\tau' \defeq \sigma \cupdot \{ \mathbf{n}' \}\)
		is a (necessarily maximal) simplex of \(\mathcal{A}(d,k)\). Without restriction, \(\tau \subset \mathcal{W}(d,k)\).
		\item As usual, we must consider several cases (and subcases), for each of which we construct \(\mathbf{n}'\), using Lemma \ref{Lemma.Maximal-simplex-of-A} as a guideline. First consider the case 
		\begin{enumerate}[label=(\arabic*)]
			\item There exist \(\mathbf{a}\) and \(\mathbf{b}\) in \(\sigma\) such that \(\mathbf{a} \triangleleft \mathbf{n} \triangleleft \mathbf{b}\).
		\end{enumerate}
		Then \(\mathbf{n}' \defeq \mathbf{a} + \mathbf{b} - \mathbf{n}\) satisfies \(\mathbf{a} \triangleleft \mathbf{n}' \triangleleft \mathbf{b}\), and (the generalization of) Lemma \ref{Lemma.Boxes-of-sums-of-certain-vertices} gives that \(\mathbf{n}' \in \mathcal{A}(d,k)\) is as
		wanted.
		\item Suppose that \(\tau\) contains no gaps, that is
		\begin{enumerate}[label=(\arabic*)] \stepcounter{enumii} %
			\item \(\tau = \{ \mathbf{m}^{(j)} \mid 0 \leq j \leq r-2 \} \) with
		\end{enumerate}	
		\begin{equation} \label{Eq.Sequence-of-ms}
			\mathbf{m} = \mathbf{m}^{(0)} \triangleleft \mathbf{m}^{(1)} \triangleleft \dots \triangleleft \mathbf{m}^{(r-2)}
		\end{equation}
		and either
		\begin{equation} \tag{\arabic{enumii}a}
			\mathbf{n} = \mathbf{m}
		\end{equation}
		or
		\begin{equation} \tag{\arabic{enumii}b}
			\mathbf{n} = \mathbf{m}^{(r-2)}.
		\end{equation}
		(The cases \(\mathbf{n} = \mathbf{m}^{(j)}\) with \(1 \leq j \leq r-3\) are covered by (1).) Note that the \(\tau\) of this form are those labelled case (b) in Theorem \ref{Theorem.Simplicial-complexes-strongly-equidimensional}. On these, the value of \(v_{k}^{(d)}\) is constant, say 
		\(v \defeq v_{k}^{(d)}(\mathbf{m}^{(j)})\).
		
		Suppose we are in \fbox{Case (2a)}. Consider a simplex \(\tau'\) of \(\mathcal{W}(d,k)\) contained in \([\mathbf{m}^{(1)}, \mathbf{m}^{(1)} + \mathbf{y}]\), encompassing \(\sigma = \{ \mathbf{m}^{(1)}, \dots, \mathbf{m}^{(r-2)}\}\), and 
		maximal with these properties. Lemma \ref{Lemma.On-vertices-and-simplices} (Case (a)) shows that for \(\mathbf{n}' \defeq \max \tau'\) either \(\mathbf{m}^{(r-2)} \triangleright \mathbf{n}' \triangleright \mathbf{m}^{(1)} + \mathbf{y}\) or
		\(\mathbf{n}' = \mathbf{m}^{(1)} + \mathbf{y}\) holds. In either case, \(\tau' = \sigma \cup \{\mathbf{n}'\} = \{ \mathbf{m}^{(1)}, \dots, \mathbf{m}^{(r-2)}\} \cup \{\mathbf{n}'\}\) is as wanted.
		
		Now assume \fbox{Case (2b)}. If \(v \geq d\), then \(\mathbf{n}' \defeq \mathbf{m}^{(r-3)} - \mathbf{y} \in \mathcal{A}(d,k)\) (see Corollary \ref{Corollary.Membership-of-AIZ}(ii)), hence 
		\(\tau' = \{ \mathbf{n}' \} \cup \sigma = \{\mathbf{n}' \} \cup \{\mathbf{m}^{(0)}, \dots, \mathbf{m}^{(r-3)}\}\) is as desired. Therefore we can assume that \(v < d\). Let \(\tau'\) be a simplex of \(\mathcal{A}(d,k)\)
		contained in \([\mathbf{m}^{(r-3)} - \mathbf{y}, \mathbf{m}^{(r-3)}]\), encompassing \(\sigma\), and maximal with these properties.
		
		Lemma \ref{Lemma.On-vertices-and-simplices} (Case (b)) shows that \(\tau'\) is strictly larger than \(\sigma\); hence \(\tau' = \{ \mathbf{n}' \} \cup \sigma \) with some \(\mathbf{n}'\) (that must satisfy \(\mathbf{n}' = \mathbf{m}^{(r-3)} - \mathbf{y}\) or
		\(\mathbf{m}^{(r-3)} - \mathbf{y} \triangleleft \mathbf{n}' \triangleleft \mathbf{m}^{(0)} \), but these properties are not relevant).
		\item Suppose that
		\begin{enumerate}[label=(\arabic*)] \stepcounter{enumii} \stepcounter{enumii} %
			\item \(\tau\) contains a gap, that is, cannot be written in the form \eqref{Eq.Sequence-of-ms}.
		\end{enumerate}
		There exists a vertex \(\mathbf{h}\) of \(\mathcal{W}\) such that \(\varphi \defeq \tau \cup \{\mathbf{h}\}\) is a maximal simplex of \(\mathcal{W}\). It may be written \(\varphi = \{ \mathbf{m}^{(j)} \mid 0 \leq j < r\} \), where
		\begin{equation} \label{Eq.Sequence-of-ms-ends-at-m+y}
			\mathbf{m} = \mathbf{m}^{(0)} \triangleleft \mathbf{m}^{(1)} \triangleleft \dots \triangleleft \mathbf{m}^{(r-1)} = \mathbf{m} + \mathbf{y}.
		\end{equation}
		The cases \(\mathbf{h}= \mathbf{m}^{(j)}\) with \(j = 0\) or \(r-1\) are covered by (2); so we may assume \(\mathbf{h} = \mathbf{m}^{(j)}\) with \(1 \leq j \leq r-2\). We distinguish the subcases
		\begin{equation} \tag{\arabic{enumii}a}
			\mathbf{m} \leq \mathbf{n} < \mathbf{h}
		\end{equation}
		and
		\begin{equation} \tag{\arabic{enumii}b}
			\mathbf{h} < \mathbf{n} \leq \mathbf{m}^{(r-1)}
		\end{equation}
		with sub-subcases 
		\begin{align}
				\mathbf{m}	&< \mathbf{n} \triangleleft \mathbf{h} \tag{\arabic{enumii}a1}  \label{Case3a1} \\
				\mathbf{m}	&< \mathbf{n} \triangleleft \mathbf{b} < \mathbf{h} 	&& (\text{some } \mathbf{b} \in \sigma) \tag{\arabic{enumii}a2} \label{Case3a2}\\
				\mathbf{m} 	&= \mathbf{n} \triangleleft \mathbf{h} \tag{\arabic{enumii}a3} \label{Case3a3}\\
				\mathbf{m}	&= \mathbf{n} \triangleleft \mathbf{b} < \mathbf{h} \tag{\arabic{enumii}a4} \label{Case3a4}
		\intertext{and}
				\mathbf{h}	&\triangleleft \mathbf{n} < \mathbf{m}^{(r-1)} \tag{\arabic{enumii}b1} \label{Case3b1} \\
				\mathbf{h}	&\triangleleft \mathbf{n} = \mathbf{m}^{(r-1)} \tag{\arabic{enumii}b2} \label{Case3b2} \\
				\mathbf{h}	&< \mathbf{a} \triangleleft \mathbf{n} < \mathbf{m}^{(r-1)}	&&(\text{some } \mathbf{a} \in \sigma) \tag{\arabic{enumii}b3} \label{Case3b3} \\
				\mathbf{h}	&< \mathbf{a} \triangleleft \mathbf{n} = \mathbf{m}^{(r-1)}. \tag{\arabic{enumii}b4} \label{Case3b4}
		\end{align}
		We will see in a moment that all these cases are covered by cases treated earlier by our lemmas.		
		\item Replacing \(\mathbf{m}\) in \eqref{Case3a2} by the lower neighbor \(\mathbf{a}\) of \(\mathbf{n}\) in \eqref{Eq.Sequence-of-ms-ends-at-m+y} if necessary, we see that \eqref{Case3a2} is in fact covered by (1), and the 
		same holds for \eqref{Case3b3} (where \(\mathbf{m}^{(r-1)}\) is replaced by the upper neighbor \(\mathbf{b}\) of \(\mathbf{n}\)).
		\item Sub-subcases \eqref{Case3a1} and \eqref{Case3b1} are such that there exist \(\mathbf{a}\) and \(\mathbf{b} \in \sigma\) such that \(\mathbf{a} \triangleleft \mathbf{n} \triangleleft \mathbf{h} \triangleleft \mathbf{b}\)
		or \(\mathbf{a} \triangleleft \mathbf{h} \triangleleft \mathbf{n} \triangleleft \mathbf{b}\), respectively, and are covered by Lemma \ref{Lemma.Vertex-sandwich}.
		\item As to \eqref{Case3a3}, \(\sigma = \{ \mathbf{m}^{(2)}, \dots, \mathbf{m}^{(r-1)} \}\), and the same argument as in (2a), using Lemma \ref{Lemma.On-vertices-and-simplices} Case (a), shows that either 
		\(\mathbf{n}' \defeq \mathbf{m}^{(2)} + \mathbf{y} \in \mathcal{W}(d,k)\), or there exists \(\mathbf{n}' \in \mathcal{W}(d,k)\) with \(\mathbf{m}^{(r-1)} \triangleleft \mathbf{n}' \triangleleft \mathbf{m}^{(2)} + \mathbf{y}\). 
		For \eqref{Case3a4}, \(\sigma = \{ \mathbf{m}^{(1)}, \dots \hat{\mathbf{m}}^{(j)}, \dots, \mathbf{m}^{(r-1)} \}\) with \(\mathbf{m}^{(j)} = \mathbf{h}\) omitted, and Lemma \ref{Lemma.On-vertices-and-simplices} Case (a) in conjunction with Lemma \ref{Lemma.Membership-of-sums-of-vertices} yields 
		that \(\mathbf{n}' \defeq \mathbf{m}^{(1)} + \mathbf{y} \in \mathcal{W}(d,k)\).
		\item In sub-subcases \eqref{Case3b2} and \eqref{Case3b4}, \(\sigma = \{\mathbf{m}^{(0)}, \dots, \mathbf{m}^{(r-3)}\}\) or \(\{ \mathbf{m}^{(0)}, \dots, \hat{\mathbf{m}}^{(j)}, \dots, \mathbf{m}^{(r-2)}\}\), respectively, where as before
		\(\mathbf{h} = \mathbf{m}^{(j)}\) is omitted. Let \(v \defeq v_{k}^{(d)}(\mathbf{m}^{(r-3)})\) in the former and \(v = v_{k}^{(d)}(\mathbf{m}^{(r-2)})\) is the latter case. If \(v \geq d\) then \(\mathbf{n}' \defeq \mathbf{n}^{(r-3)} - \mathbf{y}\)
		(resp. \(\mathbf{n}' \defeq \mathbf{n}^{(r-2)} - \mathbf{y}\)) belongs to \(\mathcal{A}(d,k)\) by Corollary \ref{Corollary.Membership-of-AIZ}(ii). So suppose \(v < d\), in which case we may apply Lemma \ref{Lemma.On-vertices-and-simplices} Case (b) to deduce that
		\begin{equation}
			\text{either } \mathbf{n}' \defeq \mathbf{m}^{(r-3)} - \mathbf{y} \in \mathcal{A}(d,k), \text{ or there exists } \mathbf{n}' \in \mathcal{A}(d,k) \text{ with } \mathbf{m}^{(r-3)} - \mathbf{y} \triangleleft \mathbf{n}' \triangleleft \mathbf{m}^{(0)}; \tag{\arabic{enumii}b2}
		\end{equation}
		\begin{equation}
			\mathbf{n}' \defeq \mathbf{m}^{(r-2)} - \mathbf{y} \in \mathcal{A}(d,k), \tag{\arabic{enumii}b4}
		\end{equation}
		respectively. (In \eqref{Case3b4}, the function \(v_{k}^{(d)}\) is not constant on \(\sigma\), which therefore has a \enquote{downward} extension \(\tau' = \{ \mathbf{n}' \} \cup \sigma\), and \(\mathbf{n}'\) must equal
		\(\mathbf{m}^{(r-2)} - \mathbf{y}\).)
		
		In each of the (sub-)cases, we have constructed some vertex \(\mathbf{n}' \in \mathcal{A}(d,k)\) with the wanted properties, and thus the theorem is shown.
	\end{enumerate}
\end{proof}

\section{The involution \( (\,.\,)\hat{} \)} \label{Section.Involution}

We study the involution on \(\mathcal{A}\) and \(\mathcal{BT}\) which corresponds to the non-trivial automorphism of the Dynkin diagram.

\subsection{} Let \(\mathbf{G}\) be the group scheme \(\PGL(r)\) with its subgroups \(\mathbf{B}\) and \(\mathbf{T}\), the standard Borel subgroup of upper triangular matrices and the standard torus of diagonal matrices. For
simplicity, we write matrices where classes of matrices modulo scalar matrices are meant.

For \(g \in \mathbf{G}(K_{\infty}) = \PGL(r, K_{\infty})\), we let \(g^{t}\) be its transpose, and
\begin{equation}
	(\,.\,)^{*} \colon G(K_{\infty}) \longrightarrow G(K_{\infty})
\end{equation}
is the automorphism \(g \mapsto g^{*} \defeq (g^{t})^{-1} = (g^{-1})^{t}\). It is well-known that \((\,.\,)^{*}\) is not inner. (Recall that we always assume \(r \geq 3\). For \(r=2\), \((\,.\,)^{*}\) is the inner automorphism induced by the element
\( (\begin{smallmatrix} 0 & 1 \\ -1 & 0 \end{smallmatrix})\).) We have \(\mathbf{T}^{*} = \mathbf{T}\) and \(\mathbf{B}^{*} = \mathbf{B}^{-}\), the group of lower triangular matrices. Let \(w\) be the well-defined maximal element 
of the Weyl group \(W = W(\mathbf{T},\Phi)\) with respect to the Borel subgroup \(\mathbf{B}\). Then
\begin{equation}
	w = \begin{pmatrix} 0 & \cdots & 0 & 1 \\ \vdots & 	& 1 & 0 \\ 0 & & & \vdots \\ 1 & 0 & \cdots & 0 \end{pmatrix} \text{ with zeroes off the anti-diagonal}, w^{2} = 1, w\mathbf{B}w^{-1} = \mathbf{B}^{-}.
\end{equation}
Further, the two involutions \( (\,.\,)^{w} \colon g \mapsto wgw^{-1}\) and \( (\,.\,)^{*} \colon g \mapsto g^{*}\) commute. Therefore
\begin{equation}
	(\,.\,)\hat{} \colon g \longmapsto g\hat{} \defeq wg^{*}w^{-1} = (wgw^{-1})^{*}
\end{equation}
is an involution of \(\mathbf{G}\), too. As the set of vertices \(\mathcal{BT}(\mathds{Z})\) of \(\mathcal{BT}\) is given by \(\mathbf{G}(K_{\infty})/\mathbf{G}(O_{\infty})\) and \(\mathbf{G}(O_{\infty})\hat{} = \mathbf{G}(O_{\infty})\),
we finally get a simplicial involution on \(\mathcal{BT}\), also termed \( (~)\hat{}\). It stabilizes the data \(\mathbf{T}\), \(\mathbf{B}\) and \(\mathcal{A}\), \(\mathcal{W}\), and its effect on \(\mathbf{T}\) is
\begin{equation}
	\begin{pmatrix} t_{1} & \cdots & 0 \\ \vdots & \ddots & \vdots \\ 0 & \cdots & t_{r} \end{pmatrix} \longmapsto \begin{pmatrix} t_{r}^{-1} & \cdots & 0 \\ \vdots & \ddots & \vdots \\ 0 & \cdots & t_{1}^{-1} \end{pmatrix}.
\end{equation}
Correspondingly, its effect on \(\mathcal{A}(\mathds{Z}) \overset{\cong}{\longrightarrow} \{ [L_{\mathbf{n}}] \mid \mathbf{n} \in \mathds{Z}^{r} \}\) (recall the notation from \eqref{Eq.Vertex-set-of-full-subcomplex} and \eqref{Eq.Equivalence-of-vertices}) is 
\[
	(n_{1}, \dots, n_{r}) \longmapsto ({-}n_{r}, \dots, {-}n_{1}),
\]
which in normalized coordinates is 
\begin{align} \label{Eq.Normalized-coordinates}
	\mathbf{n} = (n_{1}, \dots, n_{r-1}, 0) 	&\mapsto \text{class of } (n_{1} - n_{r}, n_{1} - n_{r-1}, \dots, n_{1} - n_{1}) \nonumber \\
																		&= (n_{1}, n_{1} - n_{r-1}, \dots, n_{1}-n_{2}, 0) \eqdef \mathbf{n}\hat{}.
\end{align}
One easily checks that for the simple roots \(\boldsymbol{\alpha}_{i}\) of \(\Phi\) the rule  \(\boldsymbol{\alpha}_{i}(\mathbf{n}\hat{}) = \boldsymbol{\alpha}_{r-i}(\mathbf{n})\) holds, or briefly
\begin{equation}
	\boldsymbol{\alpha}_{i}\hat{} = \boldsymbol{\alpha}_{r-i}.
\end{equation}
That is, \( (\,.\,)\hat{}\) corresponds to the inversion of the Dynkin diagram \tikz[baseline, scale=0.75]{\draw (0,0) -- (1,0); \draw (2,0) -- (3,0); \node (dots) at (1.5,0) {\(\cdots\)}; \draw[fill=black] (0,0) circle (2pt);  \draw[fill=black] (1,0) circle (2pt);  \draw[fill=black] (2,0) circle (2pt);  \draw[fill=black] (3,0) circle (2pt);} of type \(A_{r-1}\) of \(\Phi\).

Now we can state and prove the symmetry property of \(\mathcal{W}(d,k)\) with respect to \((\,.\,)\hat{}\).

\begin{Theorem} \label{Theorem.Involution-and-membership}
	Let \(d\) be a natural number and \(1 \leq k < rd\). For \(\mathbf{n} \in \mathcal{W}(\mathds{Z})\), the equivalence
	\begin{equation} \label{Eq.Mirrored-membership}
		\mathbf{n} \in \mathcal{W}(d,k) \Longleftrightarrow \mathbf{n}\hat{} \in \mathcal{W}(d,rd-k)
	\end{equation}
	holds. Therefore, \(\mathcal{W}(d,k)\hat{} = \mathcal{W}(d,rd-k)\) and, since \( (\,.\,)\hat{}\) is defined on \(\mathcal{BT}\), the analogous properties \(\mathcal{BT}(d,k)\hat{} = \mathcal{BT}(d,rd-k)\), 
	\(\mathcal{A}(d,k)\hat{} = \mathcal{A}(d,rd-k)\) hold.
\end{Theorem}

\begin{proof}
	\begin{enumerate}[label=(\roman*), wide]
		\item If \(d=1\), then \(\mathcal{W}(1,k) = \mathcal{W}_{r-k}\), the \((r-k)\)-th wall of \(\mathcal{W}\), characterized by \(\mathbf{n} \in \mathcal{W}_{i} \Leftrightarrow n_{i} = n_{i+1}\). Then \eqref{Eq.Mirrored-membership} is immediate from \eqref{Eq.Normalized-coordinates}. Hence 
		we may assume that \fbox{\(d \geq 2\)} and work in our usual framework using \(d\)-diagrams.
		\item We will show more precisely: The map
		\begin{align} \label{Eq.Map-delta}
			\delta \colon \diag^{(d)}(\mathbf{n})	&\longrightarrow \diag^{(d)}(\mathbf{n}\hat{}) \\
																(i,v)			&\longmapsto (r+1-i, n_{1} + d-1-v) \nonumber
		\end{align}
		is well-defined and bijective, and sends
		\begin{equation} \label{Eq.Boxes-under-involution}
			B_{k}^{(d)}(\mathbf{n}) \text{ to } B_{rd+1-k}^{(d)}(\mathbf{n}\hat{}).
		\end{equation}
		\item As both sets \(\diag^{(d)}(\mathbf{n})\) and \(\diag^{(d)}(\mathbf{n}\hat{})\) have the same cardinality, it suffices for \eqref{Eq.Map-delta} to show that \(\delta\) is well-defined and satisfies \(\delta^{2} = \id\). Now
		\begin{align*}
			(i,v) \in \diag^{(d)}(\mathbf{n}) \Longleftrightarrow n_{i} \leq v < n_{i} + d	&&(1 \leq i < r)
		\end{align*}
		and, after a little calculation, this turns out to be equivalent with
		\[
			\mathbf{n}\hat{}_{r+1-i} \leq n_{1} + d-1 - v < \mathbf{n}\hat{}_{r+1-i} + d
		\]
		for all \(i\), that is, \(\delta\) maps in fact to \(\diag^{(d)}(\mathbf{n}\hat{})\). The equation \(\delta^{2} = 1\) is trivial.
		\item Recall the order relation \enquote{\(\leq\)} on \(\diag^{(d)}(\mathbf{n})\) given by \eqref{Eq.Order-of-boxes}. It is obvious from the definition of \(\delta\) that it reverses this order, which gives \eqref{Eq.Boxes-under-involution}.
	\end{enumerate}
\end{proof}

\subsection{} Coming back to the arithmetic significance of \(\mathcal{BT}(d,k) = \mathcal{BT}({}_{a}\ell_{k}) = \lambda(\Omega({}_{a}\ell_{k}))\), where \(a \in A = \mathds{F}[T]\) has degree \(d\), we find the following relationship. Let
\(a \in A\) of degree \(d\) be given, and consider for \(1 \leq k <rd\) the modular forms \({}_{a}\ell_{k}\). Then the zero sets \(\Omega({}_{a}\ell_{k})\) of \({}_{a}\ell_{k}\) and \(\Omega({}_{a}\ell_{rd-k})\) of \({}_{a}\ell_{rd-k}\) are related by
\begin{equation} \label{Eq.Involution-on-Omegas}
	\lambda(\Omega({}_{a}\ell_{k}))\hat{} = \lambda(\Omega({}_{a}\ell_{rd-k})).
\end{equation}

\subsection{} Note that the considerations of this section also apply for \(r=2\). Here \(\mathcal{BT}(d,k) = \mathcal{BT}^{2}(d,k)\) is a finite set of vertices of the Bruhat-Tits tree \(\mathcal{BT} = \mathcal{BT}^{2}\). The involution
\((\,.\,)\hat{}\) is trivial, and \eqref{Eq.Mirrored-membership} resp. \eqref{Eq.Involution-on-Omegas} degenerate to \(\mathcal{W}(d,k) = \mathcal{W}(d,k)\hat{}\) resp. \(\lambda(\Omega({}_{a}\ell_{k})) = \lambda(\Omega({}_{a}\ell_{2d-k}))\). This is in accordance with the symmetry property
in Theorem 5.1 of \cite{Gekeler2011}.

\section{Examples and concluding remarks} \label{Section.Examples}

Note that our Theorems \ref{Theorem.Simplicial-complexes-connected}, \ref{Theorem.Simplicial-complexes-strongly-equidimensional}, \ref{Theorem.Simplicial-complexes-boundaryless} and \ref{Theorem.Involution-and-membership} along with Proposition \ref{Proposition.Characterization-of-Wdk-as-set-of-real-points} assert all the ingredients of the Main Theorem 1.8, which thereby is proved. Here we present examples for the simplicial complexes \(\mathcal{W}(d,k)\).
For reasons of presentation, we restrict to the case \(r=3\), where \(\mathcal{W}(d,k)\) has dimension 1, i.e., is a graph embedded into \(\mathcal{W} = \mathcal{W}^{3}\) (see \ref{Picture.Weyl-Chamber} and the legend given there).

\begin{Example}[\(r=3\), the graphs \(\mathcal{W}(2,k)\), \(1 \leq k \leq 5\)] ~ 
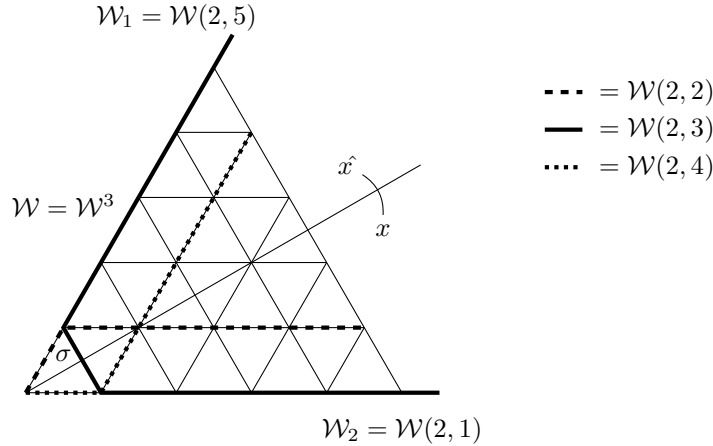
\label{Example.W2k}
	\begin{center}
		\begin{tikzpicture}
			  \foreach \row in {0, 1, ...,\rows} {
       			 \draw ($\row*(0.5, {0.5*sqrt(3)})$) -- ($(\rows,0)+\row*(-0.5, {0.5*sqrt(3)})$);
       		 	\draw ($\row*(1, 0)$) -- ($(\rows/2,{\rows/2*sqrt(3)})+\row*(0.5,{-0.5*sqrt(3)})$);
       			 \draw ($\row*(1, 0)$) -- ($(0,0)+\row*(0.5,{0.5*sqrt(3)})$);
   			 }
   			 \draw[ultra thick] (5.5,0) -- (1,0) -- (0.5,{0.5*sqrt(3}) -- ($1.1*(2.5,{2.5*sqrt(3)})$);
   			 \draw[ultra thick, dashed] (0,0) -- (0.5,{0.5*sqrt(3)}) -- (4.5,{0.5*sqrt(3)});
   			 \draw[ultra thick, dotted] (0,0) -- (1,0) -- (3,{2*sqrt(3)});
   			 \draw (0,0) -- ($3.5*(1.5,{0.5*sqrt(3)})$);
   			 \node (W) at (0.5,2.5) {\(\mathcal{W}=\mathcal{W}^{3}\)};
   			 \node (W1) at (2,5) {\(\mathcal{W}_{1} = \mathcal{W}(2,5)\)};
   			 \node (W2) at (5,-0.5) {\(\mathcal{W}_{2} = \mathcal{W}(2,1)\)};
   			 \node (sigma) at (0.5,0.5) {\(\sigma\)};
   			 \node (xhat) at ($(4.5,{1.5*sqrt(3)})+0.5*(-0.5,{+0.5*sqrt(3)})$) {\(x\hat{}\)};
   			 \node (x) at ($(4.5,{1.5*sqrt(3)})+0.5*(0.5,{-0.5*sqrt(3)})$) {\(x\)};
   			 \draw (x) to[bend right=30] (xhat);
   			 \draw[ultra thick, dashed] (6.9,4) -- (7.4,4);
   			 \draw[ultra thick] (6.9,3.5) -- (7.4,3.5);
   			 \draw[ultra thick, dotted] (6.9,3) -- (7.4,3);
   			 \node[right] (W22) at (7.5,4) {\(= \mathcal{W}(2,2)\)};
   			 \node[right] (W23) at (7.5,3.5) {\(=\mathcal{W}(2,3)\)};
   			 \node[right] (W24) at (7.5,3) {\(=\mathcal{W}(2,4)\)}; 
		\end{tikzpicture}
	\end{center} 
	Note that \(\mathcal{W}(2,1) = \mathcal{W}(1) = \) the wall \(\mathcal{W}_{2}\) and \(\mathcal{W}(2,2)\) have been drawn in [V] Figure 1. By symmetry, this gives \(\mathcal{W}(2,5) = \mathcal{W}(2,1)\hat{} = \mathcal{W}_{1}\) and 
	\(\mathcal{W}(2,4) = \mathcal{W}(2,2)\hat{}\). Let \(a \in A\) have degree 2, e.g., \(a = T^{2}\).
	Then \(f \defeq {}_{a} \ell_{1} {}_{a}\ell_{2} {}_{a}\ell_{3}\) has the property: The edges of the leftmost triangle \(\sigma\) of \(\mathcal{W}\) belong to \(\mathcal{W}(f)\), while
	 \(\overset{\circ}{\sigma}(\mathds{Q}) \cap \mathcal{W}(f) = \varnothing\). Hence the zero set \(\mathcal{W}(f)\) is not a full subcomplex of \(\mathcal{W}\) and \(f\) is not simplicial. Examples of modular forms of this shape that 
	 are not simplicial are abundant.
\end{Example}

\begin{Example}[\(r=3\), the graphs \(\mathcal{W}(3,k)\), \(k=3\) or \(4\)] ~ 
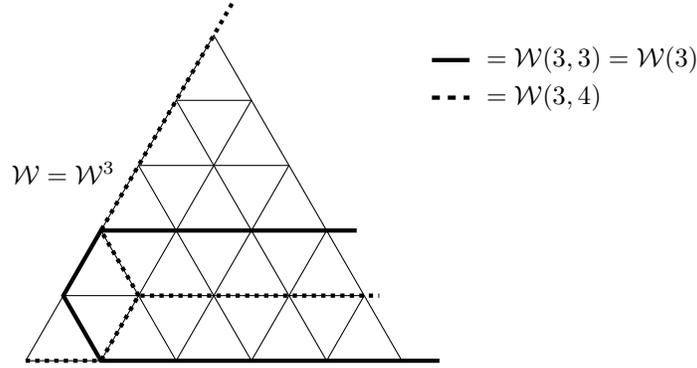
\label{Example.W3k}
	\begin{center}
		\begin{tikzpicture}
			  \foreach \row in {0, 1, ...,\rows} {
       			 \draw ($\row*(0.5, {0.5*sqrt(3)})$) -- ($(\rows,0)+\row*(-0.5, {0.5*sqrt(3)})$);
       		 	\draw ($\row*(1, 0)$) -- ($(\rows/2,{\rows/2*sqrt(3)})+\row*(0.5,{-0.5*sqrt(3)})$);
       			 \draw ($\row*(1, 0)$) -- ($(0,0)+\row*(0.5,{0.5*sqrt(3)})$);
   			 }
   			 \draw[ultra thick] ($1.1*(5,0)$) -- (1,0) -- (0.5,{0.5*sqrt(3)}) -- (1,{sqrt(3)}) -- (4.4,{sqrt(3)});
   			 \draw[ultra thick, dotted] (0,0) -- (1,0) -- (1.5,{0.5*sqrt(3)}) -- (4.7,{0.5*sqrt(3)});
   			 \draw[ultra thick, dotted] (1.5,{0.5*sqrt(3)}) -- (1,{sqrt(3)}) -- ($1.1*(2.5,{2.5*sqrt(3)})$);
   			 \node (W) at (0.5,2.5) {\(\mathcal{W}=\mathcal{W}^{3}\)};
   			  \draw[ultra thick] (5.4,4) -- (5.9,4);
   			 \draw[ultra thick, dashed] (5.4,3.5) -- (5.9,3.5);
   			 \node[right] (W22) at (6,4) {\(= \mathcal{W}(3,3) = \mathcal{W}(3)\)};
   			 \node[right] (W23) at (6,3.5) {\(=\mathcal{W}(3,4)\)};
		\end{tikzpicture}
	\end{center} 
	As \(\mathcal{W}(3,k) = \mathcal{W}(k)\) for \(k=1,2\) has been given in \ref{Example.W2k} and \(\mathcal{W}(3,k)\hat{} = \mathcal{W}(3,9-k)\), we may restrict to \(k=3\) or \(4\).
\end{Example}

\begin{Example}[\(r = 3\), the graphs \(\mathcal{W}(4,k)\), \(k = 4,5,6\)] ~ 
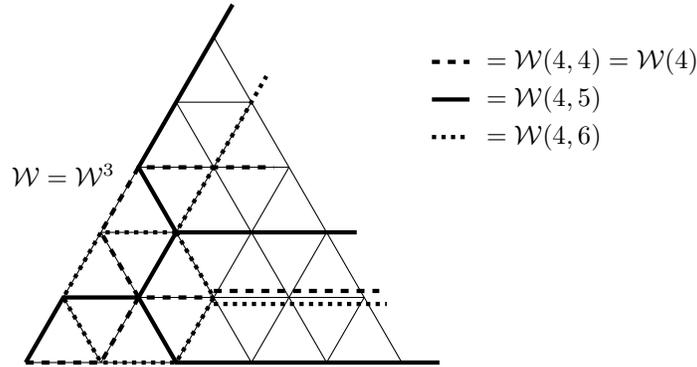
\label{Example.W4k}
	\begin{center}
		\begin{tikzpicture}
			  \foreach \row in {0, 1, ...,\rows} {
       			 \draw ($\row*(0.5, {0.5*sqrt(3)})$) -- ($(\rows,0)+\row*(-0.5, {0.5*sqrt(3)})$);
       		 	\draw ($\row*(1, 0)$) -- ($(\rows/2,{\rows/2*sqrt(3)})+\row*(0.5,{-0.5*sqrt(3)})$);
       			 \draw ($\row*(1, 0)$) -- ($(0,0)+\row*(0.5,{0.5*sqrt(3)})$);
   			 }
   			 \draw[ultra thick] (0,0) -- (0.5,{0.5*sqrt(3)}) -- (1.5,{0.5*sqrt(3)}) -- (2,0) -- (5.5,0);
   			 \draw[ultra thick] (1.5,{0.5*sqrt(3)}) -- (2,{sqrt(3)}) -- (4.4,{sqrt(3)});
   			 \draw[ultra thick] (2,{sqrt(3)}) -- (1.5,{1.5*sqrt(3)}) -- (2.75,{2.75*sqrt(3)});
   			 
   			 \draw[ultra thick, dashed] (0,0) -- (1,0) -- (1.5,{0.5*sqrt(3)}) -- (2.5,{0.5*sqrt(3)}); 
   			 \draw[ultra thick, dashed](2.5,{0.55*sqrt(3)}) -- (4.7,{0.55*sqrt(3)});
   			 \draw[ultra thick, dashed] (1.5,{0.5*sqrt(3)}) -- (1,{sqrt(3)}) -- (1.5,{1.5*sqrt(3)}) -- (3.3,{1.5*sqrt(3)});
   			 
   			 \draw[ultra thick, dotted] (2, {sqrt(3)}) -- (2.5, {0.5*sqrt(3)}) -- (2,0) -- (1,0) -- (0.5,{0.5*sqrt(3)}) -- (1,{sqrt(3)}) -- (2,{sqrt(3)}) -- ($(1,0) + 1.1*(2,{2*sqrt(3)})$);
   			 \draw[ultra thick, dotted] (2.5,{0.45*sqrt(3)}) -- (4.8,{0.45*sqrt(3)});
   			 \node (W) at (0.5,2.5) {\(\mathcal{W}=\mathcal{W}^{3}\)};
   			  \draw[ultra thick, dashed] (5.4,4) -- (5.9,4);
   			 \draw[ultra thick] (5.4,3.5) -- (5.9,3.5);
   			 \draw[ultra thick, dotted] (5.4,3) -- (5.9,3);
   			 \node[right] (W44) at (6,4) {\(= \mathcal{W}(4,4) = \mathcal{W}(4)\)};
   			 \node[right] (W45) at (6,3.5) {\(=\mathcal{W}(4,5)\)};
   			 \node[right] (W46) at (6,3) {\(=\mathcal{W}(4,6)\)};
		\end{tikzpicture}
	\end{center} 
	As \(\mathcal{W}(4,k) = \mathcal{W}(k)\) for \(k=1,2,3\) is in \ref{Example.W2k}  and \ref{Example.W3k} and \(\mathcal{W}(4,k)\hat{} = \mathcal{W}(4,12-k)\), we restrict to \(k=4,5,6\). Note that the graph \(\mathcal{W}(4,6)\) has a non-trivial cycle.
\end{Example}

\begin{Remark}[about the involution \( (\,.\,)\hat{}\)]
	A very satisfactory explanation/interpretation of Theorem \ref{Theorem.Involution-and-membership} would be given by the existence of a lift of \( (\,.\,)\hat{}\) to \(\Omega\), i.e., an involution \( \boldsymbol{\omega} \mapsto \boldsymbol{\omega}\boldsymbol{\hat{}}\)
	on \(\Omega\) such that 
	\begin{equation} \label{Eq.Diagram-involution-lift}
		\begin{tikzcd}
			\Omega \ar[d, "\lambda"'] \ar[r, "(\,.\,)\boldsymbol{\hat{}}"]	&\Omega \ar[d, "\lambda"]\\
			\mathcal{BT} \ar[r, "(\,.\,)\hat{}"']													&\mathcal{BT}
		\end{tikzcd}
	\end{equation}
	commutes. In view of \eqref{Eq.Normalized-coordinates}, a natural candidate for \( (\,.\,)\boldsymbol{\hat{}}\) would be the map
	\begin{equation} \label{Eq.Candidate-for-involution-lift}
		\boldsymbol{\omega} = (\omega_{1}, \dots, \omega_{r-1}, 1) \longmapsto \left(\omega_{1}, \frac{\omega_{1}}{\omega_{r-1}}, \dots, \frac{\omega_{1}}{\omega_{2}}, 1\right).
	\end{equation}
	This however fails for a number of reasons. First, and most important, \(\Omega\) is not stable under the map described in \eqref{Eq.Candidate-for-involution-lift} (I owe this hint to Andreas Schweizer). Second, a reasonable
	lift \( (\,.\,)\boldsymbol{\hat{}}\) as in \eqref{Eq.Diagram-involution-lift} had to interchange the zeroes of \( {}_{a}\ell_{k}\) with those of \( {}_{a}\ell_{rd-k}\), which however is excluded for \enquote{weight reasons}.
	
	Let \({}_{d}L_{k}\) be the \(\mathds{Q}\)-valued function on \(\mathcal{BT}(\mathds{Q})\) given by 
	\begin{equation}
		 {}_{d}L_{k}(\boldsymbol{x}) \defeq \log_{q} \lVert {}_{a}\ell_{k} \rVert_{\mathbf{x}},
	\end{equation}	
	where \(\lVert f \rVert_{\boldsymbol{x}}\) is the spectral norm \(\sup_{\boldsymbol{\omega} \in \lambda^{-1}(\boldsymbol{x})} \lvert f(\boldsymbol{x}) \rvert\) on the affinoid \(\lambda^{-1}(\boldsymbol{x})\) (see [I], [II], [IV]).
	(Actually, \(\lVert {}_{a}\ell_{k} \rVert_{\boldsymbol{x}}\) depends only on \(d = \deg a\), see [V] Proposition 3.8.) Now, as \( {}_{a}\ell_{k}\) has turned out to be simplicial, \( {}_{d}L_{k}\) encodes the zero set 
	\(\mathcal{BT}({}_{a}\ell_{k}) = \mathcal{BT}(d,k)\), see Proposition 1.8 in [V]. Therefore, as a substitute for some diagram \eqref{Eq.Diagram-involution-lift}, we expect(?) a simple relation between \( {}_{d}L_{k}\hat{}\) and
	\({}_{d}L_{rd-k}\) or perhaps between the van der Put transforms \( P({}_{a}\ell_{k})\hat{}\) and \(P({}_{a}\ell_{rd-k})\) (see [V] Section 1) that implies Theorem \ref{Theorem.Involution-and-membership}.
\end{Remark}

\subsection{} Similar to \(\mathcal{BT}(d,k)\), which presents a coarse picture of the zero locus \(\Omega({}_{a}\ell_{k})\) of \({}_{a}\ell_{k}\) in \(\Omega\), the object \(\Gamma \backslash \mathcal{BT}(d,k)\) may be seen as a coarse
picture of the quotient
\[
	\begin{tikzcd}
		\Gamma \backslash \Omega({}_{a}\ell_{k}) \ar[d, "\cong"'] \ar[r, hook]	& \Gamma \backslash \Omega \ar[d, "\cong"] \\
		M^{r}({}_{a}\ell_{k}) \ar[r, hook]															& M^{r}.
	\end{tikzcd}
\]
Here \(M^{r}\) is (the set of \(C_{\infty}\)-points of) the moduli scheme of rank-\(r\) Drinfeld \(A\)-modules. Now, despite the fact that \(\Gamma\) acts simplicially on \(\mathcal{BT}\) and its subcomplex \(\mathcal{BT}(d,k)\),
the quotient modulo \(\Gamma\) doesn't inherit a structure as a simplicial complex. This comes from the (well-known but annoying) fact that, e.g. in the  case of dimension 1, endpoints of \(1\)-simplices \(\sigma, \tau\) are
possibly identified under \(\Gamma\), but \(\sigma\) and \(\tau\) are not. This leads to double edges 
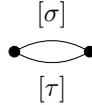
\begin{center}
	\begin{tikzpicture}
		\draw(0,0) to[bend right=30] (1,0); 
		\draw (1,0) to[bend right=30] (0,0); 
		\draw[fill=black] (0,0) circle (2pt); 
		\draw[fill=black] (1,0) circle (2pt);
		\node (GammasigmaGamma) at (0.5,0.5) {\([\sigma]\)};
		\node (GammatauGamma) at (0.5,-0.5) {\([\tau]\)};
	\end{tikzpicture}
\end{center} 
and similar phenomena in \(\Gamma \backslash \mathcal{BT}\). Hence the topological
space \(\Gamma \backslash \mathcal{BT}(d,k)(\mathds{R})\) has only the weaker structure of cell complex. Nevertheless, the next topics to deal with should be:
\begin{itemize}
	\item Study the spaces \(\Gamma \backslash \mathcal{BT}(d,k)(\mathds{R})\), the natural map \( \mathcal{W}(d,k)(\mathds{R}) \to \Gamma \backslash \mathcal{BT}(d,k)(\mathds{R})\), their Betti numbers, etc.;
	\item Work out the relationship between, say, the \(\ell\)-adic cohomology of the moduli scheme \(M({}_{a}\ell_{k})\) and the (co-)homology of \(\Gamma \backslash \mathcal{BT}(d,k)(\mathds{R})\);
	\item Instead of the zero sets of one modular form \( {}_{a}\ell_{k}\) in \(\Omega\), \(\mathcal{BT}\), \(\Gamma \backslash \mathcal{BT}\), study the simultaneous zero sets of families of modular forms in \(\Omega\), \(\mathcal{BT}\)
	and their quotients modulo \(\Gamma\), or modulo congruence subgroups \(\Gamma'\) of \(\Gamma\). For the relatively simple case of the family \( \{g_{k} \mid 1 \leq k < r, k \neq j\}\) where \(j\) is fixed, see \cite{Gekeler2019}.
\end{itemize}

\subsection{} As the reader will have noticed, \(q = \#(\mathds{F})\) has completely vanished from the largest part of the paper. This is due to the fact that we mainly worked with \(\mathcal{A}\) and \(\mathcal{W}\) and the combinatorics
of its subcomplexes. These are independent of \(q\) and are entirely determined through the root system \(\Phi\) and its Weyl group \(W\). The quantity \(q\) occurs only in the way how the different apartments 
\(\mathcal{A}' \cong \mathcal{A}\) are glued together, i.e., how many neighbors \(v' \in \mathcal{BT}(\mathds{Z})\) there are for a given vertex \(v = \mathbf{n}\) in \(\mathcal{A}\). Therefore, the simplicial complexes
\(\mathcal{A}(k)\), \(\mathcal{A}(d,k)\), \(\mathcal{BT}(k)\), \(\mathcal{BT}(d,k)\) are fundamental structures for \(\Phi\), and should play a role beyond their describing the zero sets of the Drinfeld modular forms \({}_{a}\ell_{k}\).

\end{document}